\documentclass[preprint]{imsart}
\RequirePackage[OT1]{fontenc}
\RequirePackage{amsthm, amsmath, amsfonts, amssymb,multirow}
\RequirePackage[colorlinks,citecolor=blue,urlcolor=blue]{hyperref}

\usepackage{amsmath,amssymb,amsthm,bm,graphicx,enumitem,subfigure}
\usepackage{hyperref}
\usepackage[justification=justified,singlelinecheck=false]{caption}

\RequirePackage{xr}

\startlocaldefs
\newtheorem{thm}{Theorem}[section]
\newtheorem{lemma}{Lemma}[section]
\newtheorem{prop}{Proposition}[section]
\newtheorem{cor}{Corollary}[section]

\newtheorem{remark}{Remark}[section]

\newtheorem{cond}{Condition}[section]
\newtheorem{ex}{Example}[section]

\newcommand{\argmax}{\arg\!\max}

\newcommand{\ZZ}{{\mathbb  Z}}

\allowdisplaybreaks[1]

\endlocaldefs

\begin{document}

\begin{frontmatter}
\title{Large Sample Theory for Merged Data from Multiple Sources}
\runtitle{Asymptotic Theory for Merged Data from Multiple Sources}

\author{\fnms{Takumi } \snm{Saegusa}\ead[label=e1]{tsaegusa@math.umd.edu}}
\address{University of Maryland\\ College Park, MD  20742 \\ \printead{e1}\\}
\affiliation{University of Maryland}

\runauthor{Saegusa}

\begin{abstract}
We develop large sample theory for merged data from multiple sources. 
Main statistical issues treated in this paper are (1) the same unit
potentially appears in multiple datasets from overlapping data
sources, (2) duplicated items are not identified, and (3) a sample
from the same data source is dependent due to sampling without replacement. 
We propose and study a new weighted empirical process 
and extend empirical process theory to a dependent and biased sample with duplication.
Specifically, we establish the uniform law of large numbers and uniform central
limit theorem over a class of functions along with several
empirical process results under conditions identical to  those in the i.i.d.\@ setting.
As applications, we study infinite-dimensional $M$-estimation and
develop its consistency, rates of convergence, and asymptotic normality.
Our theoretical results are illustrated with simulation studies and a
real data example.
\end{abstract}



\end{frontmatter}

\tableofcontents
\section{Introduction}
\label{sec:1}
Many organizations nowadays collect massive datasets from various sources including online  surveys, business transactions, social media, and scientific research.
In contrast to well-controlled small data, the representativeness of these datasets often critically depends on technology for data collection.
A promising remedy to reduce potential selection bias is to merge multiple samples with different coverages.
Data integration problems, however, have not been fully studied in view of basic limit theorems such as the law of large numbers (LLN) and the central limit theorem (CLT).
Main statistical challenges we focus on here
are (1) potential duplicated selection  from overlapping sources of
different sizes, (2) the lack of identification of duplicated items
across datasets, and (3) dependence among observations in each source
induced by sampling without replacement.
Because large parts of statistical theory rely on the assumption that observations are independent and identically distributed (i.i.d.), the analysis of merged data from multiple sources requires a novel approach in theory and methods.

The basic setting considered in this paper is as follows:

\noindent $\bullet$ Our interest lies in a statistical model
$\mathcal{P}$ for a vector of variables $X$ taking values in a
measurable space $(\mathcal{X},\mathcal{A})$. 
Suppose $X\sim P_0\in\mathcal{P}$.

\noindent $\bullet$ 
Let  $V = (\check{X},U) \in\mathcal{V}$ 
where $\check{X}$ is a coarsening of $X$ and $U$ is a vector of
auxiliary variables that 
do not contain information about 
the model $\mathcal{P}$. 
The space $\mathcal{V}$ consists of $J$ overlapping ``(population)
data sources'' $\mathcal{V}^{(1)},\ldots,\mathcal{V}^{(J)}$  with
$\cup_{j=1}^J\mathcal{V}^{(j)}  = \mathcal{V}$ and
$\mathcal{V}^{(j)}\cap \mathcal{V}^{(j')} \neq \emptyset$ for some
$(j,j')$. Variables $V$ determine source membership.

\noindent $\bullet$ For data collection, 
a large sample is drawn from a population:
 let $V_1,\ldots, V_N$ be i.i.d.\@ as $V$.
Unit $i$ belongs to source $j$ if $V_i\in\mathcal{V}^{(j)}$.
Sample size in source $j$ is $N^{(j)} = \#\{i\leq N:V_i\in\mathcal{V}^{(j)}\}$. 

\noindent $\bullet$ Next, a random sample of size $n^{(j)}$ is drawn  without replacement from source $j$
with sampling probability 
$\pi^{(j)}(V_i)= (n^{(j)}/N^{(j)})I\{V_i\in\mathcal{V}^{(j)}\}$.
For selected units, we observe $X$. 
We repeat the same process for all sources.
As $n^{(j)}\leq N^{(j)}$ holds, $n^{(j)}$ and $\pi^{(j)}(\cdot)$ are a
random variable and a random function, respectively. 

\noindent $\bullet$ Finally, multiple datasets from different sources are combined. Our proposed estimation method estimates the parameters of the model $\mathcal{P}$.

\begin{figure}[htbp]
\includegraphics[width=11cm]{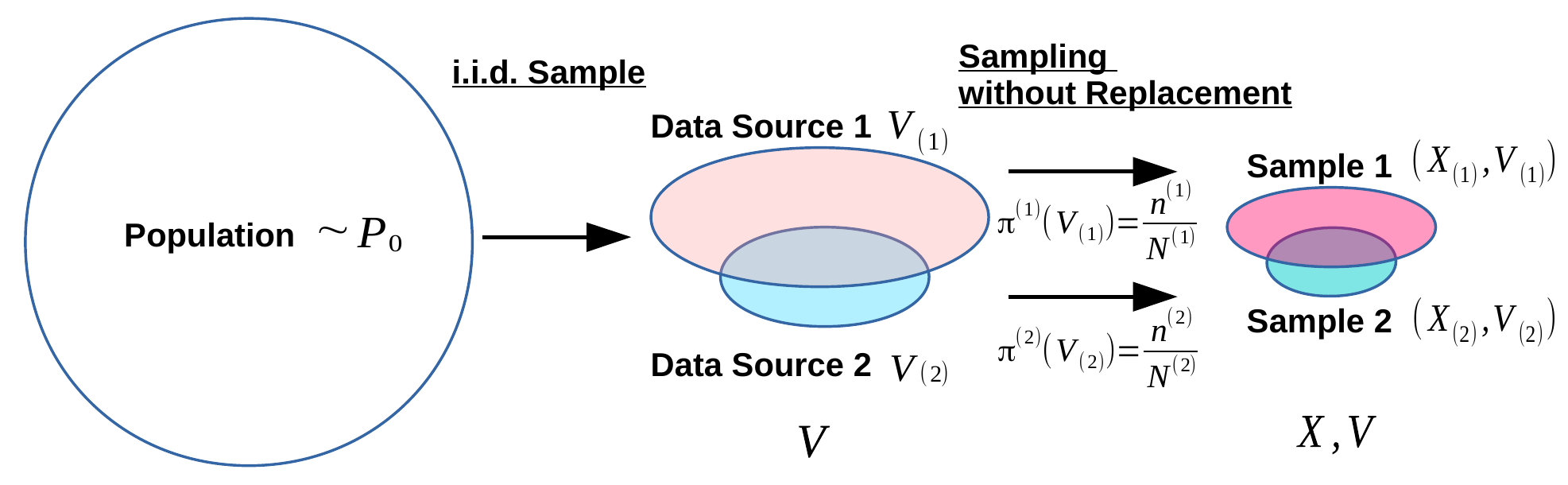}
\caption{Sampling scheme for merged data from multiple sources with $J=2$.}
\label{fig:2}
\end{figure}
The two-stage formulation  is crucial in
describing duplicated selection.  
A large sample is drawn from a population ({\it sampling from population}), and units are classified  into one or more (sample) data sources.
Next, subsamples are drawn without replacement from each data source ({\it finite-population sampling}) to generate multiple datasets.
The sample at the first stage serves as a finite population to allow for repeated selection of the same units.


Information that statisticians have at their disposal is the
$X$- and $V$-values of the selected items from different sources,
membership information on (other) data sources to which selected items belong,
and the realizations of $N^{(j)}$ and $n^{(j)}$.
A special case where $V$-values are also available for non-sampled
items is treated in Section \ref{sec:4}.

Our framework covers a number of applications.
Typical examples are opinion polls \cite{Brick2006}, public
health surveillance \cite{Hu15032011}, and health interview surveys \cite{Cervantes2006} where data sources are lists of cell- and landline-phone users.
Duplicated records in databases are important issues in business operations \cite{Herzog:2007:DQR:1534235}.  
Scientific research has considered combining 
 face-to-face, telephone and online 
surveys \cite{DeLeeuw2005,Dillman:2014}.
Our setting also covers the situation where one data source is entirely contained in another.
This case is highly useful for studying rare disease and rare exposure represented as smaller data sources \cite{Iachan1993, Kalton1986}. 
Applications include the synthesis of existing clinical and
epidemiological studies with 
surveys, disease registries, and healthcare databases \cite{MR3494641,RSSA:RSSA12136, Metcalf2009}. 

Despite scientific and financial benefits of data integration, many
important models have never been studied in our setting due to the lack of 
probabilistic tools to study a dependent and biased sample with
duplication.
We address this issue by extending empirical process
theory with applications to infinite-dimensional $M$-estimation in mind. 
This theory provides essential tools for the analysis of
semi- and non-parametric inference  (see e.g. \cite{MR2724368,MR1385671}). 
It originated in the study of the uniform law of
large numbers (U-LLN) and the uniform central limit theorem (U-CLT) in the i.i.d.\@ setting
\cite{Cantelli1933,MR0047288,MR665285,MR757767,Glivenko1933}.
The i.i.d.\@ assumption has been relaxed in several directions
including triangular arrays \cite{MR1473875,MR1834583}, 
martingale difference \cite{MR996990}, Markov chains \cite{MR1340827},
and stationary processes \cite{MR0464344}. 
The study of dependent empirical processes arising from
complex sampling was initiated by \cite{MR2325244} 
for stratified samples followed by
\cite{MR3059418}.
Beyond stratified samples, \cite{MR3619696} and \cite{MR3670194}  studied  the U-CLT for
rejective sampling and  single stage sampling, respectively. 

Our sampling scheme is markedly different from those in
the above literature in important ways. 
A basic technique to analyze dependent empirical processes
is to find a hidden (nearly) independent structure as seen in
\cite{MR3619696} that utilized similarity between independent Poisson
sampling and rejective sampling.
This method 
needs a simple dependence structure
but our merged data have 
complex multitiered dependence:
First, items within the same source are dependent due to sampling
without replacement. Second, items across overlapping sources
are dependent because they are potentially identical.
Previous studies 
focused on dependence within a sample but 
our theory addresses dependence within and between samples at the same time.
Another difference is that simple inverse probability
weighting adopted in \cite{MR3619696,MR3670194,MR2325244} is not valid
 in our setting.
This technique corrects selection bias from data sources but does not
account for bias from duplicated selection.

We build large sample theory on a novel weighted empirical process that integrates information from multiple sources.
Our main contribution is the U-LLN and U-CLT over a class of functions.
We only assume that an index set is Glivenko-Cantelli or Donsker as in
\cite{MR2325244,MR3059418}.  
This implies that if the U-LLN or the U-CLT holds for the i.i.d.\@ sample, the corresponding results hold for merged data without additional conditions. 
This formulation is of practical importance because fair comparison can be made between previous scientific conclusions from i.i.d.\@ samples and the ones from the analysis of merged data without worrying about differences in assumptions.
This generality makes a contrast with \cite{MR3619696} that assumes
the uniform entropy condition and
\cite{MR3670194} that assumes a priori the existence of the finite-dimensional CLT.

Another contribution is theory of infinite-dimensional $M$-estimation
for merged data. 
Previous research tended to focus
on the U-CLT with limited applications as a result (e.g. statistical
functionals in \cite{MR3619696, MR3670194}), but
the U-LLN and maximal
inequalities are essential to obtain consistency and rates
of convergence for $M$-estimators.
We obtain a set of empirical process tools beyond the U-CLT, and derive consistency, rates of convergence, and asymptotic normality of our estimators.
We obtain optimal calibration \cite{MR1173804,Rao1994} and optimal
weights in our weighted empirical process that improve efficiency of
our estimators.
We study several examples including the Cox proportional hazards
models \cite{MR0341758} 
and illustrate the finite sample performance of our methods through
numerical studies in several different scenarios.

Our theory can be viewed as a non-trivial extension of
\cite{MR2325244,MR3059418} for stratified
samples to overlapping ``strata.''
In stratified sampling,
the i.i.d.\@ sample from population is stratified and finite
population sampling is carried out in each stratum. 
One may consider our sampling scheme as ``stratified sampling''
with non-negligible intersections among strata.
The approach of \cite{MR2325244,MR3059418} is, however, not applicable
to our setting due to issues of multitiered dependence and inverse
probability weighting discussed above.
In particular, their proof  exploited the disjoint
nature of strata and reduced weak convergence to multiple convergence
within strata. 
This method addresses dependence within strata but does not cover dependence across 
``strata'' arising from duplicated selection
(see Section \ref{sec:3} for details).
Note that our framework is more general than previously
studied sampling designs including stratified sampling in that it
accommodates those designs in place of finite population sampling. 
In the Appendix \ref{sec:st}, we treat stratified sampling  at
the second stage of sampling in the data integration context.

The rest of the paper is organized as follows.
In Section \ref{sec:2}, we introduce our weighted empirical process
and discuss more on our sampling framework.
We present the U-LLN and several variants of U-CLTs in Section \ref{sec:3}.
Calibration methods are treated in Section \ref{sec:4}. 
We study infinite-dimensional $M$- and $Z$-estimation and their applications in Section \ref{sec:5}. 
Finite sample properties of proposed methods are illustrated in numerical studies in Section \ref{sec:6}. 
Section \ref{sec:7} discusses differences between our framework and
those in sampling theory.
All proofs and additional simulation are given in the Appendix.

\section{Sampling and Empirical Process}
\label{sec:2}
We review basic settings and introduce our weighted empirical process. 


\subsection{Sampling}
Let $R_i^{(j)}\in \{0,1\}$ be a sampling indicator
from source $j$. 
Simple random sampling from each source is carried out independently. 
Thus, sampling indicators $(R_1^{(j)},\ldots,R_N^{(j)})$ and $(R_1^{(j')},\ldots,R_N^{(j')})$ with $j\neq j'$ are conditionally independent given $V_1,\ldots,V_N$. 
However, sampling indicators within the same source 
are not independent but are only exchangeable due to sampling without replacement.
The unit that does not belong to source $j$ (i.e., $V_i\notin
\mathcal{V}^{(j)}$) automatically has $R_i^{(j)} = 0$. Throughout we denote inverse
probability weighting by $R_i^{(j)}/\pi^{(j)}(V_i)$ with convention $0/0=0$.

To enumerate units within a data source, we write e.g. $X_{(j),i}$ to
mean the observation of $X$ for the unit $i$ in source $j$ with index
$i$ going from 1 through $N^{(j)}$ (see e.g. \ref{eqn:hempj}).
The limits of sampling probabilities are
$\lim_{N\rightarrow\infty}\pi^{(j)}(v) =
p^{(j)}I\{v\in\mathcal{V}^{(j)}\}$ where $p^{(j)}\geq  c>0$ for some constant $c$.
We assume $N$ is known. In the Appendix \ref{sec:N}, we
consider the case of unknown $N$ which may be the case in practice.
For additional notations, let $W = (X,U) \in \mathcal{X}\times \mathcal{U}\equiv \mathcal{W}$ with $W\sim \tilde{P}_0$.
The conditional measure given membership in source $j$ is denoted as $P_0^{(j)}$, i.e., for measurable $A\subset \mathcal{W}$, $P_0^{(j)}(A) = \tilde{P}_0(A\cap\mathcal{V}^{(j)})/\nu^{(j)}$ where $\nu^{(j)}\equiv\tilde{P}_0(V\in \mathcal{V}^{(j)})$ is 
 membership probability in source $j$.
The conditional probability measure for $R_i^{(j)}$ given
$N^{(j)},i=1,\ldots,N, j=1,\ldots,J$, is denoted as $P_{R,N}$. 
The probability measure $P^{\infty}$ is defined such that its
projection of the first $N$ coordinates is $\tilde{P}_0^{N}\times
P_{R,N}$.

\subsection{Assumption of Unidentified Duplication}

Duplicated items are not identified in our setting, which reflects the lack of communication between sampling procedures.
Instead, we assume that we can identify additional data source
membership of selected items by checking their $V$. 
This assumption is not too restrictive.  
For example, telephone surveys can ask an additional question whether
to own both landline and cell phones. 
When medical studies are merged, comparison of
inclusion and exclusion criteria suffices.
Identifying duplication, on the other hand, produces
unavoidable errors.
Important identifiers
such as names, addresses, and
social security numbers are usually not disclosed for a privacy reason, and
even these variables suffer typographical errors
and inconsistent abbreviations \cite{doi:10.1080/01621459.1969.10501049,Winkler1995}. 
Correcting bias from imperfect record linkage requires a correctly
specified model of linking errors \cite{MR2915160,MR2156832}.
Our proposed method avoids these practical
difficulties, and remains valid even when identification is possible.

\subsection{Hartley-Type Empirical Process}
The empirical measure is a fundamental object in empirical process theory. 
This cannot be computed in our setting because of non-selected items
and unidentified duplicated selection.
As an alternative, we propose to study Hartley's estimator
\cite{Hartley1962, MR0228137} of a distribution function in place of
the empirical measure.

Hartley's estimator \cite{Hartley1962, MR0228137}  was originally proposed for estimation of population total and average in multiple-frame surveys in sampling theory where 
multiple samples are drawn from overlapping sampling frames.
Viewing sampling frames as data sources in our context, Hartley's
estimator of the sample average $\mathbb{P}_NX$ of $X$ when $J=2$ is defined as
\begin{eqnarray*}
\mathbb{P}_N^{\textup{H}} X 
&\equiv& \frac{1}{N}\sum_{i=1}^N \left(\frac{R_i^{(1)}\rho^{(1)}(V_i)}{\pi^{(1)}(V_i)}+\frac{R_i^{(2)}\rho^{(2)}(V_i)}{\pi^{(2)}(V_i)} \right)X_i, 
\end{eqnarray*}
where the weight function $\rho$ for duplicated selection is given by 
\begin{eqnarray*}
\rho(v) = (\rho^{(1)}(v), \rho^{(2)}(v)) \equiv \left\{
\begin{array}{ll}
(1,0) & \quad \mbox{if} \quad v\in \mathcal{V}^{(1)} \mbox{ and } v\notin \mathcal{V}^{(2)}, \\
(0,1) & \quad \mbox{if} \quad v\notin \mathcal{V}^{(1)} \mbox{ and } v\in \mathcal{V}^{(2)}, \\
(c^{(1)},c^{(2)}) &\quad \mbox{if} \quad v\in \mathcal{V}^{(1)}\cap \mathcal{V}^{(2)}, \\
\end{array}
\right.
\end{eqnarray*}
for positive constants $c^{(1)},c^{(2)}$ with $c^{(1)}+c^{(2)}=1$.
Duplicated selection and missing observations are properly addressed by the weight function $\rho(v)$ and the inverse probability weights respectively. 
In fact, this estimator is unbiased for $E(X)$ 
because $\rho^{(1)}(v) +\rho^{(2)}(v) = 1$ for all $v$ and $E[R^{(j)}_i|X_i,V_i,N^{(j)},n^{(j)}] = \pi^{(j)}(V_i)$.
Moreover, identification of duplicated items is not necessary to
compute this estimator because  the two sums 
\begin{equation}
\label{eqn:svydecomp}
\mathbb{P}_N^{\textup{H}} X = 
\frac{1}{N}\sum_{i=1}^N \frac{R_i^{(1)}\rho^{(1)}(V_i)}{\pi^{(1)}(V_i)}X_i+\frac{1}{N}\sum_{i=1}^N\frac{R_i^{(2)}\rho^{(2)}(V_i)}{\pi^{(2)}(V_i)} X_i,
\end{equation} 
 in $\mathbb{P}_N^{\textup{H}}X$ can be computed separately based on each subsample.

Motivated by Hartley's estimator, we define the {\it Hartley-type empirical measure} (H-empirical measure) for $J=2$ by 
\begin{eqnarray*}
\mathbb{P}_N^{\textup{H}}  &\equiv& \frac{1}{N}\sum_{i=1}^N \left(\frac{R_i^{(1)}\rho^{(1)}(V_i)}{\pi^{(1)}(V_i)}+\frac{R_i^{(2)}\rho^{(2)}(V_i)}{\pi^{(2)}(V_i)} \right)\delta_{(X_i,V_i)}. 
\end{eqnarray*}
This is an unbiased estimator of the empirical measure $\mathbb{P}_N \equiv N^{-1}\sum_{i=1}^N\delta_{(X_i,V_i)}$ given $(X_i,V_i),i=1,\ldots,N$.
Note, however, that $\mathbb{P}_N^{\textup{H}}$ is not a probability measure since point masses do not add up to 1 in general. 
The {\it Hartley-type empirical process} (H-empirical process) is defined by 
\begin{equation*}
\mathbb{G}_N^{\textup{H}} = \sqrt{N}(\mathbb{P}_N^{\textup{H}}  - \tilde{P}_0).
\end{equation*}

When there are more than two sources, we define the weight function $\rho=(\rho^{(1)},\ldots,\rho^{(J)}):\mathcal{V}\mapsto [0,1]^J$ that is constant on a mutually exclusive subset of $\mathcal{V}$ determined by $\mathcal{V}^{(j)}$'s:
\begin{eqnarray*}
\rho^{(j)}(v) = \left\{
\begin{array}{ll}
1, & v\in \mathcal{V}^{(j)}\cap \left(\cup_{m\neq j} \mathcal{V}^{(m)}\right)^c,\\
c^{(j)}_{k_1,\ldots,k_l}, & v \in \mathcal{V}^{(j)}\cap\left(\cap_{m=1}^l \mathcal{V}^{(k_m)} \right) \cap \left(\cup_{m\notin\{j,k_1,\ldots,k_l\}} \mathcal{V}^{(m)}\right)^c,\\
0, & v\notin \mathcal{V}^{(j)},\\
\end{array}\right.
\end{eqnarray*}
with $j,k_1,\ldots,k_l$ all different and $\sum_{j=1}^J \rho^{(j)}(v) =1$. 
The H-empirical measure is defined by
\begin{equation*}
\mathbb{P}_N^{\textup{H}} 
\equiv \frac{1}{N} \sum_{i=1}^N\sum_{j=1}^J \frac{R_i^{(j)}\rho^{(j)}(V_i)}{\pi^{(j)}(V_i)} \delta_{(X_i,V_i)} .
\end{equation*}
and the H-empirical process is defined by $\mathbb{G}_N^{\textup{H}} = \sqrt{N}(\mathbb{P}_N^{\textup{H}}  - \tilde{P}_0)$.

Let $\mathcal{F}$ be a class of measurable functions on $(\mathcal{X},\mathcal{A})$ that serves as the index set for the H-empirical process. 
As a stochastic process indexed by $\mathcal{F}$, $\mathbb{G}_N^{\textup{H}}$ evaluated at $f\in\mathcal{F}$ is a random variable
$\mathbb{G}_N^{\textup{H}} f =\sqrt{N}(\mathbb{P}_N^{\textup{H}} -\tilde{P}_0)f=\sqrt{N}(\mathbb{P}_N^{\textup{H}} f-\tilde{P}_0f)$ where $\tilde{P}_0f$ is the expectation of $f(X)$ under $\tilde{P}_0$, and $\mathbb{P}_N^{\textup{H}} f$ is the ``expectation'' of $f(X)$ under $\mathbb{P}_N^{\textup{H}}$ given by
\begin{equation*}
\mathbb{P}_N^{\textup{H}} f
\equiv \frac{1}{N} \sum_{i=1}^N\sum_{j=1}^J \frac{R_i^{(j)}\rho^{(j)}(V_i)}{\pi^{(j)}(V_i)} f(X_i) .
\end{equation*}
We often omit variables of a function in ``expectations'' as in
$\mathbb{P}_N^{\textup{H}} f$ and $\mathbb{G}_N^{\textup{H}} f$.

\section{Limit Theorems: Uniform WLLN and CLT}
\label{sec:3}

The U-LLN and U-CLT for the H-empirical process lay the groundwork for the analysis of  merged data from multiple sources. 
The critical issue for establishing these theorems is multitiered
dependence.
This is not a difficult problem in the finite-population framework 
where only sampling indicators $R^{(j)}_i$ are random. 
For example, the two terms of $\mathbb{P}_N^{\textup{H}}X$ in
(\ref{eqn:svydecomp}) are independent in this framework, and each
admits a finite-population CLT (e.g. \cite{hajek1960}) to yield the sum of independent normal random variables as a limit \cite{lu2010}.
A similar idea appears in the analysis of stratified samples.
For the derivation of the U-CLT, \cite{MR2325244} decomposed their weighted empirical process into stratum-wise empirical processes and showed their conditional weak convergence  to independent Gaussian processes given data. 
Because strata do not overlap unlike our case, conditional independence automatically becomes unconditional to complete their proof.
Unfortunately, this conditional argument is not valid in our setting due to dependence across overlapping data sources.

Our approach consists of two key ideas: (1) the decomposition of the H-empirical process into data sources with centering by appropriate variables, and (2) bootstrap asymptotics for establishing unconditional asymptotic normality. 
Our decomposition ensures unconditional independence, and bootstrap asymptotics bridges unconditional and conditional convergence.

Our decomposition emulates two stages of the sampling procedure:
\begin{equation*}
\mathbb{G}_N^{\textup{H}} = \sqrt{N}(\mathbb{P}_N^{\textup{H}}-\tilde{P}_0) 
= \sqrt{N}(\mathbb{P}_N-\tilde{P}_0) + \sqrt{N}(\mathbb{P}_N^{\textup{H}}-\mathbb{P}_N).
\end{equation*}
The first term is the empirical process
$\mathbb{G}_N=\sqrt{N}(\mathbb{P}_N-\tilde{P}_0) $ for the i.i.d.\@
sample which corresponds to sampling from population at the first stage. 
This process weakly converges to the 
Brownian bridge 
by the U-CLT for the i.i.d.\@ sample.
The second term corresponding to  
sampling from data sources is further decomposed.
Note that 
$\mathbb{P}_Nf = \sum_{j=1}^J\mathbb{P}_N\rho^{(j)}(V)f(X)$
by the fact that $\sum_{j=1}^J\rho^{(j)}(v)=1$ for every $v$.
Combining this with the decomposition of $\mathbb{P}_N^{\textup{H}}$ in (\ref{eqn:svydecomp}) with a general $J$ yields
\begin{eqnarray*}
(\mathbb{P}_N^{\textup{H}}-\mathbb{P}_N)f&=&\sum_{j=1}^J \frac{1}{N}\sum_{i=1}^N\left(\frac{R^{(j)}_i}{\pi^{(j)}(V_i)}-1\right)\rho^{(j)}(V_i)f(X_i)\\
&\equiv&\sum_{j=1}^J(\mathbb{P}_N^{\textup{H},(j)}-\mathbb{P}_N) \rho^{(j)}f .
\end{eqnarray*}
As in the finite-population framework, the conditional covariance of  $(\mathbb{P}_N^{\textup{H},(j)}-\mathbb{P}_N) \rho^{(j)}f$ with different $j$'s is zero  given data $(X_i,V_i),i=1,\ldots,N,$ because sampling from different data sources 
 (i.e., $R^{(j)}$s and $R^{(j')}$s) is independent. 
Moreover, their conditional expectations 
given data 
 are also zero because $E[R^{(j)}_i|X_i,V_i,N^{(j)},n^{(j)}] = \pi^{(j)}(V_i)$.
It follows from the total law of covariance (i.e., $\textup{Cov}(X,Y) = E[\textup{Cov}(X,Y|Z)] + \textup{Cov}(E[X|Z], E[Y|Z])$) 
that any two of summands in the last display are uncorrelated. 
The same argument applies to the relationship between each summand
and $\sqrt{N}(\mathbb{P}_N-\tilde{P}_0)f$.
Hence we obtain the decomposition of $\mathbb{G}_N^{\textup{H}}$ into $J+1$ uncorrelated pieces:
\begin{equation*}
\label{eqn:decomp}
\mathbb{G}_N^{\textup{H}} f= \sqrt{N}(\mathbb{P}_N-\tilde{P}_0) f+ \sum_{j=1}^J\sqrt{N}(\mathbb{P}_N^{\textup{H},(j)}-\mathbb{P}_N) \rho^{(j)}f.
\end{equation*}
If we show each summand converges to a Gaussian process, the limiting process of $\mathbb{G}_N^{\textup{H}}$ is the sum of $J+1$ independent Gaussian processes.

To establish weak convergence of the second term in the last display,
we adopt the bootstrap asymptotic theory. 
The key observation is to view  sampling from a data source $j$ as a single realization of the $m$-out-of-$n$ bootstrap with $m=n^{(j)}$ and $n=N^{(j)}$
where a bootstrap sample of size $m$ is drawn from a sample of size $n$ without replacement. 
To see this, rewrite $(\mathbb{P}_N^{\textup{H},(j)}-\mathbb{P}_N) \rho^{(j)}f$ by $(N^{(j)}/N) (\hat{\mathbb{P}}^{(j)}_{n^{(j)}}-\mathbb{P}^{(j)}_{N^{(j)}})\rho^{(j)}f$ where
\begin{equation}
\hat{\mathbb{P}}^{(j)}_{n^{(j)}} \equiv \frac{1}{n^{(j)}}\sum_{i=1}^{N^{(j)}}R^{(j)}_{(j),i}\delta_{(X_{(j),i},V_{(j),i})}, 
\qquad 
\mathbb{P}^{(j)}_{N^{(j)}} \equiv \frac{1}{N^{(j)}}\sum_{i=1}^{N^{(j)}}\delta_{(X_{(j),i},V_{(j),i})}.
\label{eqn:hempj}
\end{equation}
Here we enumerate the items within data source $j$.
Focusing on source $j$, $\mathbb{P}^{(j)}_{N^{(j)}}\rho^{(j)}f$ is the sample mean of $\rho^{(j)}(V)f(X)$ before sampling at the second stage while
$\hat{\mathbb{P}}^{(j)}_{n^{(j)}}\rho^{(j)}f$
is the sample mean  after sampling.
In view of the $m$-out-of-$n$ bootstrap,
the former is an average in the original sample while the latter is
a bootstrap average, and hence their difference is expected to yield
asymptotic normality with appropriate scaling.
Although $m/n=n^{(j)}/N^{(j)} \rightarrow p^{(j)}\neq 0$ unlike the
usual $m$-out-of-$n$ bootstrap method, 
asymptotics in our case can be treated as the special case of the
exchangeably weighted bootstrap studied by \cite{MR1245301}.
Theory of \cite{MR1245301} emphasized conditional weak convergence, but it is not difficult to extend their proof to unconditional one.
Accordingly we obtain the sum of independent Gaussian processes as the limit of $\mathbb{G}_N^{\textup{H}}$.
In the Appendix \ref{sec:uclt}, we make this heuristic argument rigorous.

Below we write $P^*$ and $E^*$ to mean outer probability of $P^{\infty}$ and expectation with respect to $P^*$. Since empirical process theory concerns the supremum of random elements, we use these notations to take care of measurability issues. For more details, see Section 1.2 of \cite{MR1385671}. A reader not interested in technical details can replace these by $\tilde{P}_0$ and $E$ without harm. 

\subsection{Uniform Law of Large Numbers}
The U-LLN holds for the empirical measure $\mathbb{P}_N$ in the i.i.d.\@ setting if the index set $\mathcal{F}$ is a Glivenko-Cantelli class (see e.g. p.81 of \cite{MR1385671}). 
This Glivenko-Cantelli property is sufficient for the U-LLN for  merged data from multiple sources. The following result is obtained by applying the bootstrap U-LLN \cite{MR1385671} to our decomposition of $\mathbb{G}_N^{\textup{H}}$.

\begin{thm}
\label{thm:gc}
Suppose that $\mathcal{F}$ is $P_0$-Glivenko-Cantelli.
Then 
\begin{equation*}
\lVert \mathbb{P}_N^{\textup{H}} - \tilde{P}_0\rVert_{\mathcal{F}} = \sup_{f\in\mathcal{F}}|(\mathbb{P}_N^{\textup{H}} -\tilde{P}_0)f|\rightarrow_{P^*} 0,
\end{equation*}
where $\lVert \ell\rVert_{\mathcal{F}}=\sup_{f\in\mathcal{F}}|\ell(f)|$ for a functional $\ell$ on $\mathcal{F}$.
\end{thm}

\subsection{Uniform Central Limit Theorem}
The empirical process $\mathbb{G}_N$  in the i.i.d.\@ setting weakly converges to a Gaussian process if the index set $\mathcal{F}$ is a Donsker class (see e.g. p.81 of \cite{MR1385671}).
This Donsker property is sufficient for the U-CLT for the H-empirical process $\mathbb{G}_N^{\textup{H}}$. 
This is an expected consequence from bootstrap asymptotics which does not need additional conditions.

\begin{thm}
\label{thm:D}
Suppose that $\mathcal{F}$ is $P_0$-Donsker.
Then 
\begin{equation*}
\mathbb{G}_N^{\textup{H}}(\cdot) \rightsquigarrow \mathbb{G}^{\textup{H}}(\cdot) \equiv \mathbb{G}(\cdot) + \sum_{j=1}^J \sqrt{\nu^{(j)}} \sqrt{\frac{1-p^{(j)}}{p^{(j)}}}\mathbb{G}^{(j)}(\rho^{(j)}\cdot)
\end{equation*}
in the class $\ell^{\infty}(\mathcal{F})$ of uniformly bounded
functionals on
$\mathcal{F}$ where the $P_0$-Brownian bridge process $\mathbb{G}$ and the $P_0^{(j)}$-Brownian bridge processes $\mathbb{G}^{(j)}$ are independent.
The covariance function $\upsilon(\cdot,\cdot) = \textup{Cov}(\mathbb{G}^{\textup{H}} \cdot,\mathbb{G}^{\textup{H}} \cdot)$ on $\mathcal{F}\times \mathcal{F}$ is
\begin{eqnarray*}
\upsilon(f,g)
=\textup{Cov}_0(f,g)+\sum_{j=1}^J \nu^{(j)}\frac{1-p^{(j)}}{p^{(j)}}\textup{Cov}_0^{(j)}(\rho^{(j)}f,\rho^{(j)}g).
\end{eqnarray*}
where $\textup{Cov}_0$ and $\textup{Cov}_{0}^{(j)}$ are  covariances under $P_0$ and $P_0^{(j)}$ respectively.
\end{thm}

The asymptotic variance here admits natural interpretations.
Consider  $\mathbb{G}_N^{\textup{H}} f$ for estimation of $P_0f$ for instance. Its asymptotic variance is
\begin{eqnarray*}
\textup{AV}(\mathbb{G}_N^{\textup{H}} f)  = \underbrace{\textup{Var}_0\{f(X)\}}_{\mbox{population variance}} + \sum_{j=1}^J \underbrace{ \nu^{(j)}\frac{1-p^{(j)}}{p^{(j)}} \textup{Var}_{0}^{(j)}\{\rho^{(j)}(V)f(X)\}}_{\mbox{design variance from source $j$}},
\end{eqnarray*}
where $\textup{Var}_0(f)=\textup{Cov}_0(f,f)$ and $\textup{Var}_0^{(j)}(f)=\textup{Cov}^{(j)}_0(f,f)$.
The first and second terms correspond to sampling from population and  data sources respectively. If we would obtain the i.i.d.\@ sample instead, the asymptotic variance is only the first term $\textup{Var}_0\{f(X)\}$. This can be obtained from our formula if we would sample all items from each data source (i.e., $p^{(j)}=1$). 
This implies that as long as we sample all the items at the second
stage, combining multiple datasets does not increase the difficulty of estimation.
If data source $j$ is large (i.e.,
$\tilde{P}_0(V\in\mathcal{V}^{(j)})=\nu^{(j)}$ is large), its contribution to asymptotic variance becomes larger.
Each quantity in the variance formula is easily estimated by Hartleys' estimator of moments
(see also the Appendix \ref{sec:data} for variance estimators for several regression
models).

\begin{remark}
In Theorems \ref{thm:gc} and \ref{thm:D}, we assume Glivenko-Cantelli
and Donsker properties of $\mathcal{F}$ with respect to $P_0$ in
order to emphasize that these properties in the i.i.d.\@ setting are
sufficient for our setting. 
A brief inspection of our proof reveals that our theorems hold valid for $\tilde{P}_0$-Glivenko-Cantelli
and $\tilde{P}_0$-Donsker classes of functions defined on $\mathcal{W}=\mathcal{X}\times \mathcal{U}$.  
\end{remark}

\subsubsection{Finite-Population sampling}
Finite-population sampling concerns randomness only from the selection of units, and is often of interest in sampling theory.
As expected from our interpretation of asymptotic variance in Theorem \ref{thm:D}, we only obtain design variance from sources in this framework.
\begin{cor}
\label{cor:fpD}
Suppose that $\mathcal{F}$ is $P_0$-Donsker.
Then 
\begin{equation*}
\mathbb{G}_N^{\textup{H,fin}}(\cdot)\equiv \sqrt{N}(\mathbb{P}_N^{\textup{H}}-\mathbb{P}_N)(\cdot) \rightsquigarrow \mathbb{G}^{\textup{H,fin}}(\cdot)\equiv
\sum_{j=1}^J \sqrt{\nu^{(j)}} \sqrt{\frac{1-p^{(j)}}{p^{(j)}}}\mathbb{G}^{(j)}(\rho^{(j)}\cdot)
\end{equation*}
in $\ell^{\infty}(\mathcal{F})$ conditionally on $(X_1,V_1),(X_2,V_2)\ldots,$ with the covariance function $\upsilon^{\textup{fin}}(\cdot,\cdot) = \textup{Cov}(\mathbb{G}^{\textup{H,fin}} \cdot,\mathbb{G}^{\textup{H,fin}} \cdot)$ on $\mathcal{F}\times \mathcal{F}$given by
\begin{eqnarray*}
\upsilon^{\textup{fin}}(f,g)
= \sum_{j=1}^J \nu^{(j)}\frac{1-p^{(j)}}{p^{(j)}}\textup{Cov}_0^{(j)}(\rho^{(j)} f,\rho^{(j)}g).
\end{eqnarray*}
\end{cor}

\subsubsection{Bernoulli Sampling}
Sampling without replacement is often replaced by Bernoulli sampling for mathematical convenience. 
To see its consequence, we consider Bernoulli sampling within sources where
selections from source $j$ are i.i.d.\@ Bernoulli$(p^{(j)})$.
Data from the same source then become independent, but dependence remains between datasets from overlapping sources. 
We write $\mathbb{G}_N^{\textup{H,Ber}}$ for the H-empirical process in this case.

\begin{thm}
\label{thm:BerD}
Suppose that $\mathcal{F}$ is $P_0$-Donsker.
Then $\mathbb{G}_N^{\textup{H,Ber}}\rightsquigarrow \mathbb{G}^{\textup{H,Ber}}$  in $\ell^{\infty}(\mathcal{F})$  where $\mathbb{G}^{\textup{H,Ber}}$ is  the zero-mean Gaussian process $\mathbb{G}^{\textup{H,Ber}}$ with covariance function
$\upsilon^{\textup{Ber}}(\cdot,\cdot) = \textup{Cov}(\mathbb{G}^{\textup{H,Ber}} \cdot,\mathbb{G}^{\textup{H,Ber}} \cdot)$ on $\mathcal{F}\times \mathcal{F}$ given by
\begin{equation*}
\upsilon^{\textup{Ber}}(f,g)=\textup{Cov}_0( f,g) + \sum_{j=1}^J\nu^{(j)} \frac{1-p^{(j)}}{p^{(j)}} P_0^{(j)}\left\{(\rho^{(j)})^2fg\right\}.
\end{equation*}
\end{thm}

Bernoulli sampling yields larger asymptotic variance than sampling
without replacement.
As expected from the decomposition of the asymptotic variance, the difference appears only in the design variances.
\begin{cor}[Finite-Population Correction]
The asymptotic variance is smaller when subsamples from sources are obtained from sampling without replacement than from Bernoulli sampling. In particular, 
\begin{equation*}
\textup{AV}(\mathbb{G}_N^{\textup{H}} f) = \textup{AV}(\mathbb{G}_N^{\textup{H,Ber}} f) - \sum_{j=1}^J\nu^{(j)} \frac{1-p^{(j)}}{p^{(j)}} \{P_0^{(j)}\rho^{(j)}(V)f(X) \}^{ 2}.
\end{equation*}
\end{cor}

\subsubsection{Optimal $\rho$}
We derive the optimal weight function $\rho$ based on  our U-CLT.
We propose the use of the optimal $\rho$ under Bernoulli sampling 
which only involves $p^{(j)}$ determined by design.
The optimal $\rho$ under sampling without replacement involves an
estimand itself and should differ from
parameter to parameter.
We show the optimal choice under Bernoulli sampling works well under sampling without replacement in simulation studies in Section \ref{sec:5}.
 
\begin{prop}[Optimal $\rho$ under Bernoulli Sampling]
\label{prop:optBer}
Let $f:\mathcal{X}\rightarrow \mathbb{R}^k$ be arbitrary with $P_0f^{2}<\infty$.
Let $od(p)=(1-p)/p$. 
When $J=2$,
the optimal function $\rho$  that minimizes the asymptotic variance of
$\mathbb{G}^{\textup{H,Ber}}_Nf$ has  
\begin{equation*}
c^{(1)}=\frac{od(p^{(2)})}{od(p^{(1)})+od(p^{(2)})} ,\quad c^{(2)}=\frac{od(p^{(1)})}{od(p^{(1)})+od(p^{(2)})}.
\end{equation*}
When $J\geq 2$,
the optimal function $\rho$  that minimizes the asymptotic variance of
$\mathbb{G}^{\textup{H,Ber}}_Nf$ has (1) $c^{(j)}_{k_1,\ldots,k_l}=0$
if $p^{(j)}<1$ and $p^{(k_m)}=1$ for some $m$,
(2) arbitrary $c^{(j)}_{k_1,\ldots,k_l}$ 
if $p^{(j)}=1$, and (3)
\begin{eqnarray*}
c^{(j)}_{k_1,\ldots,k_l} = \frac{\prod_{m=1}^lod(p^{(k_m)})}{\prod_{m=1}^lod(p^{(k_m)}) + od(p^{(j)})\sum_{n=1}^l\prod_{m=1}^lod(p^{(k_m)})/od(p^{(k_n)}) },
\end{eqnarray*}
if $p^{(j)},p^{(k_m)} <1,m=1,\ldots,l$.
\end{prop}

For sampling without replacement, we treat the case $J=2$ only.
The general case can be similarly derived via quadratic programming. 
\begin{prop}[Optimal $\rho$ under Sampling without Replacement]
\label{prop:optWR}
  Let $f:\mathcal{X}\rightarrow \mathbb{R}^k$ be a function with $P_0f^{2}<\infty$.
  Let $Y_f \equiv f(X)I\{V\in\mathcal{V}^{(1)}\cap \{\mathcal{V}^{(2)}\}^c\}$ and $Z_f \equiv f(X)I\{V\in\mathcal{V}^{(1)}\cap \mathcal{V}^{(2)}\}$.
 Define 
\begin{equation*}
c_f\equiv\frac{-\nu^{(1)}od(p^{(1)})P_0^{(1)}Y_fP_0^{(1)}Z_f\ +\nu^{(2)}od(p^{(2)})\left\{P_0^{(2)}Y_fP_0^{(2)}Z_f -\textup{Var}_0^{(2)}(Z_f)\right\} }{\nu^{(1)}od(p^{(1)})\textup{Var}_0^{(1)}(Z_f) +\nu^{(2)}od(p^{(2)})\textup{Var}_0^{(2)}(Z_f)}.
\end{equation*}
When $J=2$, the optimal function $\rho$ that minimizes the asymptotic
variance of $\mathbb{G}^{\textup{H}}_Nf$ has $c^{(1)}=0\vee c_f \wedge
1$ and $c^{(2)} = 1-c$.
\end{prop}

In a finite-population framework, \cite{MR2324141} derived optimal $\rho$ for general complex surveys. 
Their optimal $\rho$ agrees with ours under Bernoulli sampling, but they differ under sampling without replacement.
The difference is due to their probabilistic framework where
\cite{MR2324141} minimizes variance of $\mathbb{P}_N^{\textup{H}} f$
(which is zero in the limit) rather than the asymptotic variance of  $\mathbb{G}_N^{\textup{H}}f$.

\section{Calibration}
\label{sec:4}
The H-empirical process is computed from selected units only.
If information on  auxiliary variables $V$ are available for non-selected units,
calibration methods improve efficiency of our estimator.
The key idea for calibration is that a statistic computed from sampled
units (e.g. $\mathbb{P}_N^{\textup{H}} V$) is 
approximately equal to  a statistic computed from all units
(e.g. $\mathbb{P}_NV$).
Adjusting weights in $\mathbb{P}_N^{\textup{H}}$ that induce similarity between two statistics makes selected units more representative of the population. 
Different methods use different pairs of two statistics. 
Below, we first introduce the extension of \cite{Ranalli2015} to a general
$J\geq 2$ and then propose our method. 

The original calibration \cite{MR1173804} ((2.3) of p. 377) equates the Horvitz-Thompson estimator \cite{MR0053460} of $\tilde{P}_0V$ and sample average $\mathbb{P}_NV$ in order to improve the Horvitz-Thompson estimator of $P_0X$. 
Along the same line, \cite{Ranalli2015} imposed a constraint on Hartley's estimator $\mathbb{P}_N^{\textup{H}}V$ and sample average $\mathbb{P}_NV$ to improve $\mathbb{P}_N^{\textup{H}}X$ when $J=2$.
For a general $J$ we consider as its extension
the following {\it calibration equation}:
\begin{equation}
\label{eqn:caleqn}
\mathbb{P}_N^{\textup{H}} G(V^T\alpha) V = \mathbb{P}_NV,
\end{equation}
with a solution $\hat{\alpha}_N^{\textup{c}}$.
Here $G$ is a fixed function (see \cite{MR1173804} for some choice of $G$).
Using $G(V^T\hat{\alpha}_N^{\textup{c}})$, the calibrated H-empirical measure is defined as 
\begin{equation*}
\mathbb{P}_N^{\textup{H,c}} (\cdot) \equiv \mathbb{P}_N^{\textup{H}} G(V^T\hat{\alpha}_N^{\textup{c}})(\cdot) =\frac{1}{N}\sum_{i=1}^N\sum_{j=1}^J \frac{R^{(j)}_i\rho^{(j)}(V_i)}{\pi^{(j)}(V_i)}G(V_i^T\hat{\alpha}_N^{\textup{c}})\delta_{(X_i,V_i)},
\end{equation*}
and the calibrated H-empirical process  is defined as $\mathbb{G}_N^{\textup{H,c}} \equiv \sqrt{N}(\mathbb{P}_N^{\textup{H,c}}-\tilde{P}_0)$.
Other variants in \cite{Ranalli2015} can be extended by changing the range of summation. For example, if we replace $V$ by a vector with elements $VI_{\mathcal{V}^{(j)}}(V),j=1,\ldots,J,$ in (\ref{eqn:caleqn}), we obtain  {\it data-source-specific calibration} 
\begin{equation*}
\frac{1}{N}\sum_{i:V_i\in\mathcal{V}^{(j)}} \sum_{j=1}^J\frac{R_i^{(j)}\rho^{(j)}(V_i)}{\pi^{(j)}(V_i)}G(V_i^T\alpha^{(j)})V_i
=\frac{1}{N}\sum_{i:V_i\in\mathcal{V}^{(j)}}V_i, \quad j=1,\ldots,J.
\end{equation*}
The left-hand is computed from all selected units that belong to source $j$.

Our proposed method exploits the asymptotic variance formula
$\upsilon(f,f)$ in Theorem \ref{thm:D}.
We target the reduction of design variances in
\begin{equation*}
\textup{Var}_{0}^{(j)}\{\rho^{(j)}(V)f(X)\}=P_0^{(j)}\{\rho^{(j)}(V)f(X)-P_0^{(j)}\rho^{(j)}(V)f(X)\}^{\otimes
  2}.
\end{equation*} 
The key observations are (1) the conditional variance is obtained from the sample from the same source (units with $R^{(j)}=1$), and (2) variables of interest are $\rho^{(j)}(V)f(X)-P_0^{(j)}\rho^{(j)}(V)f(X)$. 
Our method is thus characterized by the following three points:
(1) calibration is carried out within a subsample from the same source,
(2) variables used are
$\rho^{(j)}(V)V$ with centering, and (3) Horvitz-Thompson estimators
are equated with sample averages.
To be specific, we propose  the {\it sample-specific calibration equation} 
\begin{equation}
\label{eqn:scaleqn}
\frac{1}{N^{(j)}}\sum_{i:V_i\in\mathcal{V}^{(j)}} \frac{R_i^{(j)}G^{(j)}_{\alpha^{(j)}}(V_i)}{\pi^{(j)}(V_i)}\left\{\rho^{(j)}(V_i)V_i-\mathbb{P}_{N^{(j)}}^{(j)}\rho^{(j)}(V)V\right\}
=0,
\end{equation}
$j=1,\ldots, J$ with solution $\hat{\alpha}_N^{\textup{sc}}=(\hat{\alpha}_N^{\textup{sc},(1)},\ldots,\hat{\alpha}_N^{\textup{sc},(J)})^T$ where 
\begin{equation*}
G^{(j)}_\alpha(v) \equiv G\left[\left\{\rho^{(j)}(v)v-\mathbb{P}_{N^{(j)}}^{(j)}\rho^{(j)}(V)V\right\}^T\alpha^{(j)}\right].
\end{equation*}
The right-hand side of (\ref{eqn:scaleqn}) is the average of
empirically centered variables, and hence equals zero.
The left-hand side is computed from selected items from source $j$ in
contrast to data-source-specific calibration that uses all items
sampled from both source $j$ and its overlapping sources.
We define the {\sl H-empirical measure with sample-specific calibration} by
\begin{eqnarray*}
\mathbb{P}_{N}^{\textup{H,sc}}
&\equiv& 
\frac{1}{N}\sum_{i=1}^N\sum_{j=1}^J\frac{R_i^{(j)}\rho^{(j)}(V_i)}{\pi^{(j)}(V_i)}G^{(j)}_{\hat{\alpha}^{\textup{sc},(j)}_N}(V_i)\delta_{(X_i,V_i)},
\end{eqnarray*}
and the corresponding  H-empirical process by $\mathbb{G}_{N}^{\textup{H,sc}}\equiv\sqrt{N}(\mathbb{P}_{N}^{\textup{H,sc}}-\tilde{P}_0)$.

We assume the following condition for calibration methods.
\begin{cond}
\label{cond:cal}
\noindent 
(a) $\hat{\alpha}_N^{\textup{c}}$ and $\hat{\alpha}_N^{\textup{sc}}$
are solutions of (\ref{eqn:caleqn}) and (\ref{eqn:scaleqn}). 

\noindent (b)  $V\in \mathbb{R}^k$ has bounded support with $\mathcal{V}^{(j)}\neq \{0\},j=1,\ldots,J$.  

\noindent(c)  $G$ is a strictly increasing, continuously differentiable, bounded function on 
$\mathbb{R}$ such that $G(0)=1$. Its derivative $\dot{G}$ is strictly positive and bounded.

\noindent(d) $\tilde{P}_0V^{\otimes 2}$ and every
$\textup{Var}_{0}^{(j)}\{\rho^{(j)}(V)V\}$ satisfying $P^{(j)}_0\rho^{(j)}(V)>0$ are finite and positive definite. 

\end{cond}
Condition \ref{cond:cal} (a) ensures the existence of solutions to
calibration equations. Under conditions (b)-(d), probability of their
existence with the choice $G(x)=1+x$ tends to 1 as
$N\rightarrow\infty$. When $V$ is bounded, $G(x)=1+x$ can be
considered as a bounded function that satisfies (c). In this case,
\begin{eqnarray*}
  \hat{\alpha}^{\textup{sc}(j)}_N  =\left\{\frac{1}{N^{(j)}}\sum_{i:V_i\in\mathcal{V}^{(j)}}\frac{R_i^{(j)}}{\pi^{(j)}(V_i)}\left(V^{\rho^{(j)}}_i\right)^{\otimes 2}\right\}^{-1} \frac{1}{N^{(j)}}\sum_{i:V_i\in\mathcal{V}^{(j)}}\frac{R_i^{(j)}}{\pi^{(j)}(V_i)}V^{\rho^{(j)}}_i
\end{eqnarray*} 
where $V_i^{\rho^{(j)}}\equiv \rho^{(j)}(V_i)V_i-\mathbb{P}_{N^{(j)}}^{(j)}\rho^{(j)}(V)V$.
The probability of the existence of the matrix inverse above tends to
1 due to (d). A similar argument applies to $\hat{\alpha}_N^{\textup{c}}$. 
Note that the choice of $G$ does not affect the limiting processes in
the uniform CLT for calibrated H-empirical processes below. 
\begin{thm}
\label{thm:calD}
Suppose $\mathcal{F}$ is $P_0$-Donsker with $\lVert P_0\rVert_{\mathcal{F}}<\infty$. 
Under Condition \ref{cond:cal},
\begin{eqnarray*}
 &&\mathbb{G}_N^{\textup{H,c}}(\cdot) \rightsquigarrow \mathbb{G}^{\textup{H,c}}(\cdot) \equiv \mathbb{G}(\cdot) + \sum_{j=1}^J \sqrt{\nu^{(j)}} \sqrt{\frac{1-p^{(j)}}{p^{(j)}}}\mathbb{G}^{(j)}(\rho^{(j)}I-Q_{\textup{c}}^{(j)})(\cdot),\\
&&\mathbb{G}_N^{\textup{H,sc}}(\cdot) \rightsquigarrow \mathbb{G}^{\textup{H,sc}}(\cdot) \equiv \mathbb{G}(\cdot) + \sum_{j=1}^J \sqrt{\nu^{(j)}} \sqrt{\frac{1-p^{(j)}}{p^{(j)}}}\mathbb{G}^{(j)}(\rho^{(j)}I-Q_{\textup{sc}}^{(j)})(\cdot),
\end{eqnarray*}
in $\ell^{\infty}(\mathcal{F})$.
Here $\mathbb{G}$ and $\mathbb{G}^{(j)}$ are the same as in Theorem \ref{thm:D}, 
$I$ is the identity map, and  $Q_{\textup{c}}^{(j)}$ and $Q_{\textup{sc}}^{(j)}$ are maps from the class of functions on $\mathcal{X}$ to the class of linear maps on $\mathcal{V}$ defined by
\begin{eqnarray*}
&&Q_{\textup{c}}^{(j)} (f)[v] = \tilde{P}_0(f(X)V^T)\{\tilde{P}_0V^{\otimes 2}\}^{-1}\rho^{(j)}(v)v,\\
&&Q_{\textup{sc}}^{(j)}(f)[v]=P_0^{(j)}\{\rho^{(j)}(V)f(X)(\rho^{(j)}(V)V-P_0^{(j)}\rho^{(j)}(V)V)^T\} \\
&& \qquad\times\left\{\textup{Var}_{0}^{(j)}\left(\rho^{(j)}(V)V\right)\right\}^{-1}\left\{\rho^{(j)}(v)
   v-P_0^{(j)}\rho^{(j)}(V)V\right\}I\{v\in \mathcal{V}^{(j)}\}.
\end{eqnarray*}
Covariance functions $\upsilon^{\#}(\cdot,\cdot) =
\textup{Cov}(\mathbb{G}^{\textup{H},\#}
\cdot,\mathbb{G}^{\textup{H},\#} \cdot)$ on $\mathcal{F}\times
\mathcal{F},\#\in\{\textup{c,sc}\}$ are 
\begin{eqnarray*}
\upsilon^{\#}(f,g)
=\textup{Cov}_0(f,g)+\sum_{j=1}^J \nu^{(j)}\frac{1-p^{(j)}}{p^{(j)}}\textup{Cov}_0^{(j)}\left(\rho^{(j)}f-Q_{\#}^{(j)}(f),\rho^{(j)}g-Q_{\#}^{(j)}(g)\right).
\end{eqnarray*}
\end{thm}

To compare above methods, 
define  the class $\mathcal{C}$ of  estimators of $P_0f$ for arbitrary $f$ with $P_0f^{\otimes 2} <\infty$ whose asymptotic variance takes the form of
\begin{eqnarray*}
\textup{Var}_0\{f(X)\} + \sum_{j=1}^J\nu^{(j)}\frac{1-p^{(j)}}{p^{(j)}} \textup{Var}_{0}^{(j)}\left[\rho^{(j)}(V)f(X) - L^{(j)}_f\left\{\rho^{(j)}(V)V\right\}\right]
\end{eqnarray*}
where $L^{(j)}_f(v)$ is a linear function of $v$ that depends on $f$. 
Note that calibration and the sample-specific calibration have $L^{(j)}_f\{\rho^{(j)}(v)v\}=Q_{\textup{c}}^{(j)}(f)[v]$ and $L^{(j)}_f\{\rho^{(j)}(v)v\}=Q_{\textup{sc}}^{(j)}(f)[v]$. 
The optimal $L_f^{(j)}\{\rho^{(j)}(v)v\}$ is the orthogonal projection of $\rho^{(j)}(v)f(x)$ onto the linear span of  $\rho^{(j)}(v)v-P_0^{(j)}\{\rho^{(j)}(V)V\}$ with respect to the pseudo-metric $d^{(j)}(f,g) = \{\textup{Var}_{0}^{(j)}(f-g)\}^{1/2}$. 
This is exactly $L^{(j)}_f\{\rho^{(j)}(v)v\}=Q_{\textup{sc}}^{(j)}(f)[v]$. Thus, we obtain the following theorem.

\begin{thm}
Sample-specific calibration is optimal among $\mathcal{C}$ with improved asymptotic variance over a non-calibrated estimator:
\begin{equation*}
\textup{AV}(\mathbb{G}_N^{\textup{H,sc}}f ) = \textup{AV}(\mathbb{G}_N^{\textup{H}} f) - \sum_{j=1}^J\nu^{(j)} \frac{1-p^{(j)}}{p^{(j)}} \textup{Var}_{0}^{(j)}\left(Q_{\textup{sc}}^{(j)}f\right).
\end{equation*}
\end{thm}
The performance of methods based on \cite{Ranalli2015} depends on specific situations. See our simulation in Section \ref{sec:6}. 

\section{Applications to Infinite-dimensional $M$-Estimation}
\label{sec:5}
An estimator 
in a statistical model is often characterized as a maximizer of a
 criterion function or a zero of estimating equations. 
The former estimator is called an $M$-estimator and the latter a $Z$-estimator. 
A canonical example for both cases is the maximum likelihood estimator (MLE) which maximizes likelihood and solves likelihood equations. 
In the i.i.d.\@ setting, empirical process theory plays a major role
in studying both estimators in a  general setting where parameters are
infinite-dimensional \cite{MR3012400,MR2202406,
MR1333176,MR1385671}.
In this section, we apply H-empirical process results to study
limiting properties of infinite-dimensional $M$- and $Z$-estimation
for data integration.

Suppose $\mathcal{P}$ is the collection of probability measures $P_\theta$ on $(\mathcal{X},\mathcal{A})$ parametrized by $\theta\in \Theta$ where $\Theta$ is a subset of a Banach space $(\mathcal{B},\lVert\cdot \rVert)$. 
The true distribution is $P_0 = P_{\theta_0}\in \mathcal{P}$.
Let $\mathcal{M}=\{m_\theta:\theta\in\Theta\}$ be a set of criterion functions on $\mathcal{X}$.
In the i.i.d.\@ setting, the $M$-estimator is defined as 
\begin{equation*}
\hat{\theta} = \argmax_{\theta\in\Theta}\mathbb{P}_N m_{\theta}(X).
\end{equation*}
Our proposed $M$-estimator $\hat{\theta}_N$ replaces the empirical measure by the H-empirical measure:
\begin{equation*}
\hat{\theta}_N = \argmax_{\theta\in\Theta}\mathbb{P}_N^{\textup{H}} m_{\theta}(X)
\end{equation*}
In the following, we establish consistency and rates of convergence of our $M$-estimator, while we consider $Z$-estimation for asymptotic normality.
Treating two estimators interchangeably can be justified because the
$M$-estimator often (nearly) solves estimating equations obtained from
the criterion function.
This relationship 
must be verified in each specific model.

\subsection{Consistency}
\label{sbsec:5.1}
The following theorem concerns consistency of our proposed $M$-estimator.
The key assumption is the Glivenko-Cantelli property of $\mathcal{M}$ by which our U-LLN applies.
\begin{thm}
\label{thm:const}
Suppose that $\mathcal{M}$ is $P_0$-Glivenko-Cantelli, and that for every $\epsilon>0$,
$
P_0m_{\theta_0} > \sup_{\theta:\lVert \theta-\theta_0\rVert>\epsilon} P_0m_{\theta}.
$
Then 
\begin{equation*}
\lVert \hat{\theta}_N-\theta_0\rVert \rightarrow_{P^*}0.
\end{equation*}
\end{thm}
In certain semiparametric models MLEs do not exist and nonparametric MLEs are considered as alternatives. 
In this case, the parameter space
for optimization may not be the same as the original space, 
 and consistency must be carefully proved based on properties of a specific model.
Our U-LLN continues to be helpfull for this purpose (see Example \ref{ex:coxr}).

\subsection{Rate of Convergence}
\label{sbsec:5.2}
A rate of convergence appears in Condition \ref{cond:wlenonreg1} for one of our
$Z$-theorems, namely, Theorem \ref{thm:zthm2}.
In the i.i.d.\@ case, convergence rates are often obtained by the peeling device (\cite{MR822047}, see also Theorem 3.2.5 of \cite{MR1385671}) together with maximal inequalities for the empirical process.
Instead of obtaining maximal inequalities of H-empirical processes for different $\mathcal{M}$ each time, we directly compare maximal inequalities for the empirical and H-empirical processes to obtain the following theorem. 
This theorem ensures the same rate of convergence both in the i.i.d.\@ setting and our setting.
Below, we denote $ a \lesssim b$ to mean $a \le K b$ for some constant $K \in (0,\infty)$.

\begin{thm}
\label{thm:rate}
Suppose for every $\theta$ in a neighborhood of $\theta_0$,
\begin{equation}
\label{eqn:ratecond1}
P_0(m_\theta-m_{\theta_0})\lesssim 
-\lVert \theta-\theta_0\rVert^2.
\end{equation}
For every $N$ and sufficiently small $\delta>0$, it holds that
\begin{equation}
\label{eqn:mxineq}
E^*\sup_{\lVert \theta-\theta_0\rVert<\delta} \left|\mathbb{G}_N(m_\theta-m_{\theta_0})\right|
\lesssim \phi_N(\delta)
\end{equation}
for functions $\phi_N$ such that $\delta\mapsto \phi_N(\delta)/\delta^\alpha$ 
is decreasing for some $\alpha<2$ (not depending on $N$).
If $\hat{\theta}_N\rightarrow_{P^*}\theta_0$ and
$\mathbb{P}_N^{\textup{H}} m_{\hat{\theta}_N}\geq \mathbb{P}_N^{\textup{H}} m_{\theta_0}-O_{P^*}(r^{-2}_N)$,
then $r_N\lVert \hat{\theta}_N-\theta_0\rVert=O_{P^*}(1)$ 
for every $r_N$ such that $r_N^2\phi_N(1/r_N)\leq \sqrt{N}$
for every $N$.
\end{thm}

\subsection{Infinite-dimensional $Z$-theorem}
We consider asymptotic distributions of our $Z$-estimators by  extending two infinite-dimensional $Z$-theorems (Theorem 3.3.1 of \cite{MR1385671} and Theorem 6.1 of \cite{MR1394975}) in the i.i.d.\@ setting to our setting.
The first theorem concerns estimators with regular parametric rate of convergence.
The second theorem specializes in semiparametric models with non-regular rate of convergence for nuisance parameters.  
The estimators are obtained by replacing $\mathbb{P}_N$ by $\mathbb{P}_N^{\textup{H}}$ in estimating equations. 
We also consider calibration methods in the previous section in these theorems.

\subsubsection{Parametric rate of convergence for nuisance parameters}
\label{sbsec:5.3.1}
Let $\hat{\theta}_{N}$ and $\hat{\theta}_{N,\#}$ be estimators of $\theta$ obtained as solutions to the estimating equations given by
\begin{eqnarray*}
&&\left\lVert\Psi^{\textup{H}}_{N}(\theta)\right\rVert_{\mathcal{H}}\equiv \left\lVert \mathbb{P}_N^{\textup{H}} B_\theta\right\rVert_{\mathcal{H}} = o_{P^*}(N^{-1/2}),\\
&&\left\lVert\Psi^{\textup{H}}_{N,\#}(\theta)\right\rVert_{\mathcal{H}}\equiv \left\lVert \mathbb{P}_N^{\textup{H},\#} B_\theta\right\rVert_{\mathcal{H}} = o_{P^*}(N^{-1/2}),\quad \#\in \{\textup{c,sc}\},
\end{eqnarray*}
respectively where $B_\theta$ is a map from some set $\mathcal{H}$ to $L_2(P_\theta)$ indexed by $\theta$.
Recall, for example, $\left\lVert \mathbb{P}_N^{\textup{H}} B_\theta\right\rVert_{\mathcal{H}} =\sup_{h\in\mathcal{H}}|\mathbb{P}_N^{\textup{H}} B_\theta(h)|$ (see also Example \ref{ex:para}).
Let $\Psi(\theta)\equiv P_0B_\theta$ and $\Psi_N(\theta)\equiv \mathbb{P}_NB_\theta$ be maps from $\Theta$ to $\ell^\infty(\mathcal{H})$.
We assume:

\begin{cond}
\label{cond:zthm1-1}
For the true parameter $\theta_0\in\Theta$, $\Psi(\theta_0)=0$. The set $\{B_{\theta_0}(h):h\in\mathcal{H}\}$ is $P_0$-Donsker and $\{(B_\theta-B_{\theta_0})(h):\theta\in\Theta,h\in\mathcal{H}\}$ is $P_0$-Glivenko-Cantelli with an integrable envelope.
\end{cond}
\begin{cond}
\label{cond:zthm1-2}
Suppose that $\Psi$ is Fr\'{e}chet differentiable at $\theta_0$;
\begin{equation*}
\left\lVert \Psi(\theta)-\Psi(\theta_0)-\dot{\Psi}_0(\theta-\theta_0)\right\rVert_{\mathcal{H}}=o\left(\lVert \theta-\theta_0\rVert\right).
\end{equation*}
Moreover, $\dot{\Psi}_0$ is continuously invertible at $\theta_0$ with inverse denoted as $\dot{\Psi}_0^{-1}$
\end{cond}
\begin{cond}
\label{cond:zthm1-3}
For any $\delta_N\downarrow 0$, the following stochastic equicontinuity condition holds at $\theta_0$;
\begin{equation*}
\sup_{\lVert\theta-\theta_0\rVert\leq \delta_N}\left\lVert \frac{\sqrt{N}(\Psi_N-\Psi)(\theta)-\sqrt{N}(\Psi_N-\Psi)(\theta_0)}{1+\sqrt{N}\lVert\theta-\theta_0\rVert}\right\rVert_{\mathcal{H}}  = o_{P^*}(1).
\end{equation*}
\end{cond}

Now we present the following infinite-dimensional $Z$-theorem. 
\begin{thm}
\label{thm:zthm1}
Suppose that Conditions \ref{cond:zthm1-1}-\ref{cond:zthm1-3} hold and that estimators $\hat{\theta}_N,\hat{\theta}_{N,\#},$ with $\#\in \{\textup{c,sc}\}$ are consistent for $\theta_0$.
Then 
\begin{eqnarray*}
&&\sqrt{N}(\hat{\theta}_{N}-\theta_0)
\rightsquigarrow -\dot{\Psi}_0^{-1}\mathbb{G}^{\textup{H}}B_{\theta_0},\\
&&\sqrt{N}(\hat{\theta}_{N,\#}-\theta_0)
\rightsquigarrow -\dot{\Psi}_0^{-1}\mathbb{G}^{\textup{H},\#}B_{\theta_0}.
\end{eqnarray*}
\end{thm}

\subsubsection{Non-regular rate of convergence for nuisance parameters}
\label{sbsec:5.3.2}
We focus on a semiparametric model  $\mathcal{P} = \{p_{\theta,\eta}: \theta\in\Theta\subset \mathbb{R}^p,\eta\in H\}$, the collection of densities on $(\mathcal{X},\mathcal{A})$ where $\Theta\subset \mathbb{R}^p$, and $H$ is a subset of a Banach space $(\mathcal{B},\lVert\cdot \rVert)$. 
The true distribution is $P_0 = P_{\theta_0,\eta_0}\in \mathcal{P}$.
Estimator $(\hat{\theta}_N,\hat{\eta}_N)$ solves the Hartley-type likelihood equations
\begin{eqnarray}
\label{eqn:wlik2}
&&\Psi_{N,1}^{\textup{H}}(\theta,\eta,\alpha)=\mathbb{P}_N^{\textup{H}}\dot{\ell}_{\theta,\eta}
        =o_{P^*}\left(N^{-1/2}\right),\nonumber \\
&&\Psi_{N,2}^{\textup{H}}(\theta,\eta,\alpha)\left[\underline{h}_0\right]
         =\mathbb{P}_N^{\textup{H}} B_{\theta,\eta}[\underline{h}_0]=o_{P^*}\left(N^{-1/2}\right),
\end{eqnarray}
Here $\dot{\ell}_{\theta,\eta}\in \mathcal{L}_2^0(P_{\theta,\eta})^p$ is the 
score function for $\theta$, and the score operator 
$B_{\theta,\eta}:\mathcal{H}\mapsto \mathcal{L}_2^0(P_{\theta,\eta})$ 
is the bounded linear operator mapping a direction $h$ in some Hilbert space $\mathcal{H}$ of one-dimensional submodels for $\eta$ along which $\eta\rightarrow \eta_0$
(see e.g. \cite{MR1652247} for review of semiparametric models).
We write $B_{\theta,\eta}\left[\underline{h}\right]
=\left(B_{\theta,\eta}(h_1),\ldots,B_{\theta,\eta}(h_p)\right)^T$ for $\underline{h}=(h_1,\ldots,h_p)^T\in \mathcal{H}^p$, and $\underline{h}_0$ is defined in Condition \ref{cond:wlenonreg2} below.
We also write $\dot{\ell}_{0} = \dot{\ell}_{\theta_0,\eta_0}$ and $B_0=B_{\theta_0,\eta_0}$. 
We assume:
\begin{cond}
\label{cond:wlenonreg1}
An estimator $(\hat{\theta}_N,\hat{\eta}_N)$ of $(\theta_0,\eta_0)$ satisfies
$|\hat{\theta}_N-\theta_0|= o_{P^*}(1)$, and
$\lVert \hat{\eta}_N-\eta_0\rVert = O_{P^*}(N^{-\beta})$ for some $\beta>0$, and solves the estimating equations (\ref{eqn:wlik2}) where $\mathbb{P}_N^{\textup{H}}$
may be replaced by $\mathbb{P}_N^{\textup{H},\#}$ with the corresponding estimators $(\hat{\theta}_{N,\#},\hat{\eta}_{N,\#})$ where $\#\in\{\textup{c,sc}\}$.
\end{cond}
\begin{cond}
\label{cond:wlenonreg2}
There is an $\underline{h}_0= (h_{0,1},\ldots,h_{0,p})^T\in \mathcal{H}^p$ such that
\begin{equation*}
P_0 \{  ( \dot{\ell}_{0} -B_{0}[\underline{h}_0]  ) B_{0}(h)  \}= 0, \quad \mbox{for all } h\in \mathcal{H}.
\end{equation*}
Furthermore, 
$I_0\equiv P_0\left(\dot{\ell}_{0}-B_{0}[\underline{h}_0]\right)^{\otimes 2} $ is finite and nonsingular.
\end{cond}

\begin{cond}
\label{cond:wlenonreg3}
(1) For any $\delta_N\downarrow 0$ and $C>0$,
\begin{eqnarray*}
&&\sup_{|\theta-\theta_0|\leq \delta_N,\lVert \eta-\eta_0\rVert\leq CN^{-\beta}}
\left|\mathbb{G}_N(\dot{\ell}_{\theta,\eta}-\dot{\ell}_{0})\right| = o_{P^*}(1),\\
&&\sup_{|\theta-\theta_0|\leq \delta_N,\lVert \eta-\eta_0\rVert\leq CN^{-\beta}}
\left|\mathbb{G}_N(B_{\theta,\eta}-B_{0})[\underline{h}_0]\right| = o_{P^*}(1).
\end{eqnarray*}
(2) For some $\delta>0$ classes
$\{\dot{\ell}_{\theta,\eta}:|\theta-\theta_0| +\lVert \eta-\eta_0\rVert \leq \delta\}$
and $\{B_{\theta,\eta}\left[\underline{h}_0\right]:|\theta-\theta_0|+\lVert \eta-\eta_0\rVert\leq \delta\}$ are $P_0$-Glivenko-Cantelli and have integrable envelopes.
Moreover, $\dot{\ell}_{\theta,\eta}$ and $B_{\theta,\eta}[\underline{h}_0]$ are
continuous with respect to $(\theta,\eta)$  in $L_1(P_0)$.
\end{cond}

\begin{cond}
\label{cond:wlenonreg4}
For some $\alpha>1$ satisfying $\alpha \beta >1/2$ and for $(\theta,\eta)$
in the neighborhood $\{(\theta,\eta):|\theta-\theta_0|\leq \delta_N,\lVert \eta-\eta_0\rVert\leq CN^{-\beta}\}$,
\begin{eqnarray*}
&&\left|P_0\left[\dot{\ell}_{\theta,\eta}-\dot{\ell}_{0}
       +\dot{\ell}_{0}\left\{\dot{\ell}_{0}^T(\theta-\theta_0)
       +B_{0}(\eta-\eta_0)\right\}\right]\right|\\
&&\quad=o\left(|\theta-\theta_0|\right)+O\left(\lVert \eta-\eta_0\rVert^{\alpha}\right),\\
&&\left|P_0\left[(B_{\theta,\eta}-B_{0})[\underline{h}_0]
         +B_{0}[\underline{h}_0]\left\{\dot{\ell}_{0}^T
              (\theta-\theta_0)+B_{0}(\eta-\eta_0)\right\}\right]\right|\\
&&\quad =  o\left(|\theta-\theta_0|\right)+O\left(\lVert \eta-\eta_0\rVert^{\alpha}\right).
\end{eqnarray*}
\end{cond}

We then obtain the following infinite-dimensional $Z$-theorem:
\begin{thm}
\label{thm:zthm2}
Under Conditions \ref{cond:cal}, \ref{cond:wlenonreg1}-\ref{cond:wlenonreg4},
\begin{eqnarray*}
\begin{array}{llll}
\sqrt{N}(\hat{\theta}_N-\theta_0)
\rightsquigarrow \mathbb{G}^{\textup{H}}\tilde{\ell}_0
\sim \quad N_p(0,\upsilon(\tilde{\ell}_0,\tilde{\ell}_0)),\\
\sqrt{N}(\hat{\theta}_{N,\#}-\theta_0)
\rightsquigarrow \mathbb{G}^{\textup{H},\#}\tilde{\ell}_0
\sim  N_p(0,\upsilon^{\#}(\tilde{\ell}_0,\tilde{\ell}_0)),
\end{array}
\end{eqnarray*}
where  $\#\in\{\textup{c,sc}\}$, and $\upsilon$ and $\upsilon^{\#}$ are as
defined in Theorems~ \ref{thm:D} and \ref{thm:calD}.
\end{thm}

\subsection{Examples}
\begin{ex}[Parametric model]
\label{ex:para}
Consider the parametric model $\{dP_\theta/d\mu = p_\theta:\theta\in
\Theta\subset \mathbb{R}^p\}$ with a dominating measure $\mu$.
A natural estimator $\hat{\theta}_N$ of $\theta$ is a solution to the Hartley-type likelihood equation given by
\begin{equation*}
\frac{1}{N}\sum_{i=1}^N \sum_{j=1}^J\frac{R_i^{(j)}\rho^{(j)}(V_i)}{\pi^{(j)}(V_i)} \dot{\ell}_{\theta}(X_i) = 0,
\end{equation*}
where $\dot{\ell}_\theta = d\log p_{\theta}/d\theta$.
Let $\dot{\ell}_\theta(x) = (\dot{\ell}_{\theta,1}(x),\ldots,\dot{\ell}_{\theta,p}(x))^T$. For $\mathcal{H} = \{h_1,\ldots,h_p\}$, define the map by $h_i\mapsto B_{\theta}(h_i)=\dot{\ell}_{\theta,i}(x)$.
Then the above estimating equation can be written as $\lVert \Psi_N^{\textup{H}} (\theta)\rVert_{\mathcal{H}} =\sup_{h\in \mathcal{H}}|\mathbb{P}_N^{\textup{H}} B_{\theta}(h)|= 0$.
Square integrability of $\dot{\ell}_{\theta}$ for each $\theta$ under $P_{0}$ implies Condition \ref{cond:zthm1-1}. 
When the Fisher information matrix $I_0\equiv
P_0(\dot{\ell}_{0}^{\otimes 2})$ is invertible, and $\log
p_\theta$ is twice differentiable with respect to $\theta$ in a
neighborhood of $\theta_0$, Condition \ref{cond:zthm1-2} is satisfied. 
If we further assume $\log p_\theta$ is twice continuously differentiable in a neighborhood of $\theta_0$ and $\Theta$ is compact, Condition \ref{cond:zthm1-3} is met.  
Consistency follows from Theorem \ref{thm:const} if $\{B_\theta(h_i),i=1,\ldots,p,\theta\in\Theta\}$ has an integrable envelope.
Hence our first $Z$-theorem (Theorem \ref{thm:zthm1}) yields
\begin{equation*}
\sqrt{N}(\hat{\theta}_N-\theta_0)  \rightarrow_d \mathbb{G}^{\textup{H}}\tilde{\ell}_0\sim N\left(0,I_0^{-1} + \sum_{j=1}^J \nu^{(j)}\frac{1-p^{(j)}}{p^{(j)}} \textup{Var}_{0}^{(j)}(\rho^{(j)}\tilde{\ell}_{0})\right),
\end{equation*}
where $\tilde{\ell}_0=I_0^{-1}\dot{\ell}_0$. The cases for calibration are similar.
\end{ex}

\begin{ex}[Regular semiparametric model with $\eta$ as measure]
\label{ex:semipara1}
Consider the semiparametric model  $\mathcal{P} = \{p_{\theta,\eta}: \theta\in\Theta\subset \mathbb{R}^p,\eta\in H\}$ where the nuisance parameter $\eta$ is a measure. 
Several $Z$-theorems of the form of Theorem \ref{thm:zthm1} were applied to this case \cite{MR1385671} (see Section 3.3 of \cite{MR1385671} in the i.i.d.\@ setting and \cite{MR2325244, MR2391566,MR3059418} for stratified samples). 
We obtain a similar result from Theorem \ref{thm:zthm1}  by following arguments in \cite{MR3059418}.
The score operator in this model is $B_{\theta,\eta}:L_2(\eta)\mapsto L_2(P_{\theta,\eta})$ and its adjoint operator is denoted as $B_{\theta,\eta}^*:L_2(P_{\theta,\eta})\mapsto L_2(\eta)$.
As in \cite{MR1385671}, we assume $B^*_0B_0$ is continuously invertible and that $\Psi$ has continuously invertible Fr\'{e}chet derivative $\dot{\Psi}_0$ at $(\theta_0,\eta_0)$ with respect to $(\theta,\eta)$ of the  form
\begin{eqnarray*}
&&\dot{\Psi}_{11}(\theta-\theta_0)
       =-P_{0}\dot{\ell}_{0}\dot{\ell}_{0}^T(\theta-\theta_0),\\
&&\dot{\Psi}_{12}(\eta-\eta_0)=-\int B^*_{0}\dot{\ell}_{0}d(\eta-\eta_0),
 \\
&&\dot{\Psi}_{21}(\theta-\theta_0)h
       =-P_{0}B_{0}h\dot{\ell}_{0}^T(\theta-\theta_0), \quad h\in L_2(\eta),
\\
&&\dot{\Psi}_{22}(\eta-\eta_0)h=-\int B^*_{0}B_{0}h d(\eta-\eta_0),\quad h\in L_2(\eta).
\end{eqnarray*}
Further assuming consistency and asymptotic equicontinuity (see  \cite{MR2325244, MR2391566,MR3059418,MR1385671} for more details), Theorem \ref{thm:zthm1} yields  
\begin{equation*}
\sqrt{N}(\hat{\theta}_N-\theta_0) \rightarrow_d \mathbb{G}^{\textup{H}}\tilde{\ell}_0\sim
 N\left(0,I_0^{-1} + \sum_{j=1}^J\nu^{(j)}\frac{1-p^{(j)}}{p^{(j)}}\textup{Var}_{0}^{(j)}(\rho^{(j)}\tilde{\ell}_0)\right)
\end{equation*}
where $\tilde{I}_0=P_0[(I-B_0(B^*_0B_0)^{-1}B_0^*)\dot{\ell}_0\dot{\ell}_0^T]$ 
is the efficient information for $\theta$ and 
$\tilde{\ell}_0=\tilde{I}_0^{-1}(I-B_0(B^*_0B_0)^{-1}B_0^*)\dot{\ell}_0$ 
is the efficient influence function for $\theta$ in the i.i.d.\@ setting.
\end{ex}

\begin{ex}[Cox model with right-censored data]
\label{ex:coxr}
Let $T\sim F$ be a failure time, and $Z=(Z_1,Z_2)$ be covariates.
The Cox model specifies the relationship between covariates and the cumulative  hazard function by
\begin{equation*}
\Lambda(t|z)=\exp(\theta^Tz)\Lambda(t),
\end{equation*}
where $\theta\in \mathbb{R}^p$ is the regression parameter, and
$\Lambda$ is the baseline cumulative hazard function.
Under right censoring we do not always observe $T$ but  observe $Y\equiv \min\{T,C\}$ and $\Delta\equiv I(T\leq C)$ where $C$ is censoring time.
We assume there is some constant $\tau$ such that $P(T> \tau)>0$ and $P(C\geq \tau)=P(C=\tau)>0$ (see \cite{MR1915446} for other conditions).
We assume sources are formed based on  $V=(Y,\Delta,Z_2)$ and $Z_1$ is collected later. 
In the i.i.d.\@ setting, a nonparametric likelihood for one observation is $\ell_{\theta,\Lambda} (y, \delta , z)
 = \log\{(e^{\theta^Tz}\Lambda\{y\})^\delta e^{-\Lambda(y)e^{\theta^Tz}}\}$ where $\Lambda\{t\}$ is the jump of $\Lambda$ at $t$.
The score for $\theta$ and the score operator $B_{\theta,\Lambda}:\mathcal{H}\mapsto L_2(P_{\theta,\Lambda})$ are
\begin{eqnarray*}
&&\dot{\ell}_{\theta,\Lambda} (y, \delta , z)= z \{ \delta - e^{\theta^Tz}\Lambda(y) \},\\
&&B_{\theta,\Lambda}h (y, \delta , z) = \delta h(y)  -e^{\theta^Tz}\int_{[0,y]}hd\Lambda,
\end{eqnarray*}
where $\mathcal{H}$ is the unit ball  in the space $BV[0,\tau]$.
Here the score operator is obtained by differentiating $\ell_{\theta,\Lambda_t}$ with respect to $t$ at $t=0$ where $d\Lambda_t =(1+th)d\Lambda$.
Our proposed estimator $(\hat{\theta}_N,\hat{\Lambda}_N)$ is the solution to $\mathbb{P}_N^{\textup{H}} \dot{\ell}_{\theta,\Lambda}=0$ and $\mathbb{P}_N^{\textup{H}} B_{\theta,\Lambda}(h) = 0$,
whereby $\hat{\theta}_N$ is the solution to the weighted partial likelihood equation and $\hat{\Lambda}_N$ is the weighted Breslow estimator (see e.g. \cite{MR2325244}).
Consistency and conditions for asymptotic normality can be verified
along the same line as in \cite{MR3059418} by replacing their weighted empirical process results by our H-empirical process results.
Then Example \ref{ex:semipara1} yields
\begin{equation*}
\sqrt{N}(\hat{\theta}_N-\theta_0) \rightarrow_d 
\mathbb{G}^{\textup{H}}\tilde{\ell}_0\sim N\left(0,I_0^{-1} + \sum_{j=1}^J \nu^{(j)}\frac{1-p^{(j)}}{p^{(j)}} \textup{Var}_{0}^{(j)}(\rho^{(j)}\tilde{\ell}_{0})\right).
\end{equation*}
Here the efficient influence function $\tilde{\ell}_0=I_0^{-1}\ell^*$
in the i.i.d.\@ setting is computed from the efficient score   
\begin{equation*}
\ell^*_{0} (y, \delta , z)
=\delta(z-(M_1/M_0)(y))
         -e^{\theta_0^Tz}\int_{[0,y]}\left(z-(M_1/M_0)(t)\right)d\Lambda_0(t),
\end{equation*}
and the efficient information  
\begin{equation*}
I_{0} 
= E\left[\left(\ell^*_{0}\right)^{\otimes 2}\right]
=Ee^{\theta_0^TZ}\int_0^\tau \left(Z-(M_1/M_0)(y)\right)^{\otimes
  2}P(Y\geq y|Z)d\Lambda_0(y),
\end{equation*}
for $\theta$ in the i.i.d.\@ setting where $M_k(s) = P_{\theta_0,\Lambda_0}[Z^ke^{\theta_0^TZ}I(Y\geq s)]$, $k= 0, 1$.

\end{ex}

\begin{ex}[Cox model with current status data]
Let $T\sim F$ be a failure time, and $Z=(Z_1,Z_2)$ be covariates.
Under the case 1 interval censoring \cite{HuangWellner1997}, we do not observe 
$T$ but we only know whether an event occurs before an examination time $C$. 
We assume  sources are formed based on  $V=(C,\Delta,Z_2)$ and $Z_1$ are collected later. 
The likelihood in the i.i.d.\@ setting is 
$\ell(\theta,\Lambda)\equiv\delta\log \{1-e^{-\Lambda(c)\exp(\theta^Tz)}\}-(1-\delta)e^{\theta^Tz }\Lambda(c)$.
The score for $\theta$ and $\Lambda$ is then 
\begin{eqnarray*}
&&\dot{\ell}_{\theta,\Lambda}(c,\delta,z)=z\exp(\theta^Tz)\Lambda(c)(\delta r(c,z;\theta,\Lambda)-(1-\delta)),\\
&&B_{\theta,\Lambda}(h)(c,\delta,z)
=\exp(\theta^Tz)h(c)\left\{\delta r(c,z;\theta,\Lambda)-(1-\delta)\right\}
\end{eqnarray*}
where
$r(c,z;\theta,\Lambda) =\exp(-e^{\theta^Tz}\Lambda(c))/\{1-\exp(-e^{\theta^Tz}\Lambda(c))\}$ (see \cite{MR1394975} for details).
Our proposed  estimator $(\hat{\theta}_N,\hat{\Lambda}_N)$ is the solution to $\mathbb{P}_N^{\textup{H}} \dot{\ell}_{\theta,\Lambda}=0$ and $\mathbb{P}_N^{\textup{H}} B_{\theta,\Lambda}(h) = 0$. 
Conditions \ref{cond:wlenonreg1}-\ref{cond:wlenonreg4} can be verified
along the same line as in \cite{MR3059418} by replacing their weighted empirical process results by our H-empirical process results.
In particular, our U-LLN (Theorem \ref{thm:gc}) is used for consistency, and Theorem \ref{thm:rate} establishes the rate of convergence of $\hat{\Lambda}_N$ as $N^{1/3}$
 in view of $(\hat{\theta}_N,\hat{\Lambda}_N)$ as the maximizer of
 $\mathbb{P}_N^{\textup{H}} \ell(\theta,\Lambda)$. This rate agrees
 with the one in the i.i.d.\@ setting \cite{MR1394975}.
Then Theorem \ref{thm:zthm2} yields
\begin{equation*}
\sqrt{N}(\hat{\theta}_N-\theta_0) \rightarrow_d 
\mathbb{G}^{\textup{H}}\tilde{\ell}_0\sim N\left(0,I_0^{-1} + \sum_{j=1}^J \nu^{(j)}\frac{1-p^{(j)}}{p^{(j)}} \textup{Var}_{0}^{(j)}(\rho^{(j)}\tilde{\ell}_{0})\right).
\end{equation*}
Here $\tilde{\ell}_0=I_0^{-1}\ell_0^*$ with
\begin{equation*}
\ell^*_{0} 
\equiv e^{\theta_0^Tz}Q(c,\delta,z;\theta_0,\Lambda_0)\Lambda_0(c)
\left\{z-\frac{E[Ze^{2\theta_0^TZ}O(C|Z)|C=c]}{E[e^{2\theta_0^TZ}O(C|Z)|C=c]}\right\}  
\end{equation*}
and $I_{0} = P_0(\ell^*_{0})^{\otimes 2}$
where $Q(c,\delta,z;\theta,\Lambda) =\delta r(c,z;\theta,\Lambda)-(1-\delta)$ and 
$O(c|z)= \{1-F(c)\}^{\exp(\theta_0^Tz)}/[1-\{1-F(c)\}^{\exp(\theta_0^Tz)}]$.

\end{ex}

\section{Numerical Results}
\label{sec:6}
\subsection{Simulation Studies}

\begin{table}[h]
\scriptsize

\begin{tabular}{c|cc|c|cc|cc|c}
& $\mathcal{V}^{(1)}$& $\mathcal{V}^{(2)}$&$N$& $N^{(1)}$& $N^{(2)}$& $n^{(1)}$& $n^{(2)}$& Duplication\\\hline
Scenario 1&$Z_2\geq -1$&$Z_2\leq 1$&500&421&421&85&127&21\\
& &&10000&8413&8414&1683&2525&410\\\hline
Scenario 2&$\mathcal{V}$&$Z_2\leq 1$&500&500&421&100&127&25\\
& &&10000&10000&8413&2000&2524&505\\\hline
Scenario 3&$\mathcal{V}$&$\Delta=1$&500&500& 76&100&76&15\\
& &&10000 &10000&1529&2000&1529&305\\
\end{tabular}
\begin{tabular}{c|c|ccc|ccc|cc}
&&&&&&&&\multicolumn{2}{c}{Duplication}\\
& $N$& $N^{(1)}$& $N^{(2)}$& $N^{(3)}$&$n^{(1)}$& $n^{(2)}$&$n^{(3)}$&                                                                       twice&                                                                       
thrice\\\hline
Scenario 4&500&76&423&278&76&43&28&13&1\\
& 10000&8475&5564&1529&848&556&1529&258&9\\
\end{tabular}

\caption{Sample sizes and the numbers of duplications based on 2000
  simulated 
  datasets.\label{tab:cox1}}

\end{table}

\begin{table}
\tiny

\begin{tabular}{cc|cc|cc|cc||cc|cc|cc}
$(\alpha,\beta)$ & &\multicolumn{6}{|c|}{$(0.2,0.5)$}&
 \multicolumn{6}{|c}{$(0.2,5.0)$}\\
$\theta$&& \multicolumn{2}{|c|}{$\log (2)$}&\multicolumn{2}{|c|}{$\log (1.2)$}&\multicolumn{2}{|c|}{$0$}&\multicolumn{2}{|c|}{$\log (2)$}&\multicolumn{2}{|c|}{$\log (1.2)$}&\multicolumn{2}{|c}{$0$}\\
$N$& &500& 10000 & 500 & 10000 & 500&\multicolumn{1}{c|}{10000}& 500 &
                                                                     10000 & 500& 10000& 500&10000 \\\hline\hline
\multicolumn{14}{l}{Complete data (MLE)}\\\hline
$\theta_1$  &  Bias &.004&.0031&.004&.0002&.001&.0003&.017&.0000&.001&.0016&.001&.0024\\
&SD &.246&.0534&.241&.0531&.236&.0518&.244&.0530&.236&.0522&.234&.0518\\\hline
 $\theta_2$ &  Bias &.004&.0004&.006&.0005&.001&.0008&.011&.0020&.004&.0002&.004&.0004 \\
&SD &.121&.0270&.119&.0259&.120&.0259&.129&.0274&.117&.0260&.122&.0255\\\hline\hline

\multicolumn{14}{l}{Scenario 1}\\\hline
$\theta_1$  &  Bias &.024&.0061&.014&.0020&.011&.0017&.022&.0031&.002&.0026&.005&.0004\\
&SD &.482&.0985&.432&.0914 &.429&.0887 &.477&.0977 &.435&.0905&.423&.0889\\
&SEE &.467&.0989&.425&.0908&.419&.0899&.471&.0973&.427&.0892&.425&.0891\\\hline
 $\theta_2$ &  Bias &.005&.0031&.005&.0031&.011&.0011&.050&.0000&.004&0005&.008&.0010\\
&SD &.251&.0526&.242&.0486 &.234&.0495&.277&.0544 &.248&.0496&.249&.0505\\
&SEE &.260&.0524&.248&.0509&.244&.0507&.285&.0550&.254&.0504&.254&.0503\\\hline\hline

\multicolumn{14}{l}{Scenario 2}\\\hline
$\theta_1$  &  Bias &.062&.0005&.017&.0014&.009&.0010&.034&.0012&.019&.0030&.008&.0054\\
&SD &.479&.0967&.421&.0894&.416&.0876&.469&.0941&.425&.0889&.404&.0899\\
&SEE &.467&.0981&.421&.0901&.412&.0871&.459&.0952&.423&.0888&.415&.0883\\\hline
 $\theta_2$ &  Bias &.016&.0000&.003&.0001&.015&.0001&.026&.0027&.002&.0005&.010&.0015\\
&SD &.250&.0526&.226&.0499&.222&.0493&.259&.0510&.238&.0486&.238&.0501\\
&SEE &.252&.0510&.238&.0493&.232&.0480&.267&.0527&.241&.0489&.236&.0487\\\hline\hline

\multicolumn{14}{l}{Scenario 3}\\\hline
$\theta_1$  &  Bias &.005&.0009&.008&.0028&.006&.0011&.025&.0008&.008&.0004&.006&.0004\\
&SD &.330&.0733&.309&.0660&.301&.0676&.399&.0860&.391&.0840&.395&.0841\\
&SEE &.330&.0728&.308&.0674&.305&.0668&.375&.0856&.365&.0826&.366&.0828\\\hline
 $\theta_2$ &  Bias &.023&.0003&.018&.0010&.001&.0007&.029&.0027&.016&.0001&.001&.0014\\
&SD &.181&.0378&.157&.0341&.163&.0342&.193&.0437&.201&.0422&.202&.0414\\
&SEE &.171&.0381&.156&.0339&.156&.0334&.183&.0427&.181&.0413&.183&.0414\\\hline\hline
\multicolumn{14}{l}{Scenario 4}\\\hline 
$\theta_1$  &  Bias &.010&.0019&.003&.0001&.005&.0003&.060&.0023&.016&.0023&.002&.0001\\
&SD &.368&.0789&.354&.0760&.372&.0775&.432&.0970&.466&.1031 &.481&.1022\\
&SEE &.355&.0789&.343&.0758&.347&.0765&.401&.0965&.417&.0996&.418&.1010\\\hline
 $\theta_2$ &  Bias &.023&.0018&.006&.0007&.012&.0016&.011&.0038&.013&.0011&.037&.0019\\
&SD &.192&.0407&.179&.0363&.185&.0367&.222&.0477&.235&.0489&.239&.0499\\
&SEE &.181&.0407&.169&.0366&.169&.0367&.198&.0470&.203&.0484&.202&.0492\\\hline
\multicolumn{14}{l}{\tiny Note: MLE, maximum likelihood estimator
  based on $N$ items, Bias, an absolute Monte Carlo sample
  bias; }\\
 \multicolumn{14}{l}{\tiny SD, a Monte Carlo sample standard
  deviation; SEE, average of a plug-in estimator of a standard error.}\\
\end{tabular}
\caption{Results of Monte Carlo simulations with different $\theta$,
  $(\alpha,\beta)$, and scenarios.\label{tab:cox2}}

\end{table}

\begin{figure}
\includegraphics[width=11cm,height=6cm]{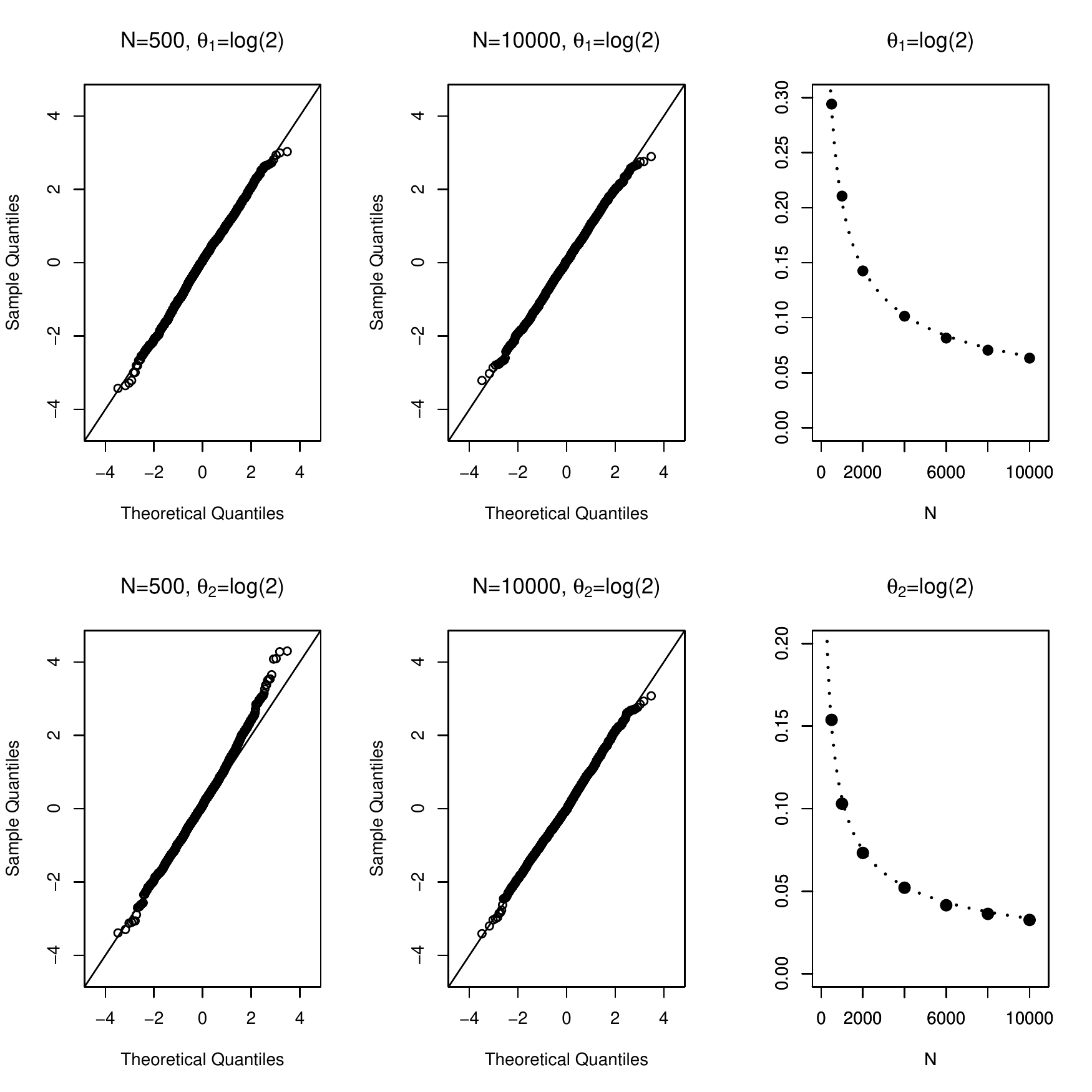}
\caption{Q-Q plots of
  $\sqrt{N}(\hat{\theta}-\theta_0)/\widehat{SE}(\hat{\theta})$ superimposed by $y=x$ and plots of averages of absolute
  differences $\lVert \hat{\theta}_{N}-\theta_{0}\rVert$
  against $N$ superimposed by $y=c/x^{1/2},c=6.5,3.4$ in Scenario 4
  where $\widehat{SE}(\hat{\theta})$ is a plug-in estimator of a standard
  error of $\sqrt{N}(\hat{\theta}-\theta_0)$.}
\label{fig:2}
\end{figure}

\begin{table}
\tiny

\begin{tabular}{cccccccccc}
 $(\alpha,\beta)=(.2,.5)$&\multicolumn{4}{c}{$N=500$}&\multicolumn{4}{c}{$N=10000$}\\ 
$\theta_1=\log (2)$& w/o& SC & C & DC & &w/o& SC & C & DC \\\hline
 MLE &  .246&&&&&.0534&&&\\
  S &.368  &.333&.370&.371&&.0789&.0720&.0789&.0789\\
  SF &.375  &.341&.376&.376&&.0809&.0740&.0809&.0804\\
  B & .497 &.474&.497&.497&&.1060&.1005&.1060&.1060\\
&  &&&&&&&&\\
$\theta_2=\log (2)$& w/o& SC & C & DC & &w/o& SC & C & DC \\\hline
 MLE & .121 &&&&&.0270 &&&\\
  S &.192  &.188&.193&.193&&.0407&.0395&.0405&.0403\\
  SF   &.197&.192&.197&.196&&.0414&.0401&.0412&.0409\\
  B   &.258&.253&.258&.258&&.0530&.0517&.0530&.0530\\
  \multicolumn{10}{l}{Note:  S, the proposed weights; SF,
  $\rho$ for a single-frame estimator; B, a balanced weights;}\\
\multicolumn{10}{l}{
 w/o, non-calibration; SC, the proposed
  calibration; C, the standard calibration; DC, the }\\
\multicolumn{10}{l}{ data-source-specific calibration. All calibrations use $U$ and $Y$.}\\
\end{tabular}

\caption{Comparison of calibrations
  and $\rho$ by standard deviations in Scenario 4.\label{tab:cox3}}

\end{table}

We conducted simulation studies to evaluate finite sample properties
of our proposed estimator in the Cox model with right censoring discussed in
Example \ref{ex:coxr}.
Linear and logistic regression models are treated in the Appendix \ref{sec:data}. 
Data were generated from the Cox model with two independent covariates
$Z_1\sim \textup{Bernoulli}(.5)$ and $Z_2\sim N(0,1)$. The failure
time $T$ follows
$\textup{Weibull}(\alpha,\beta),\alpha=.2,\beta\in\{.5,5\}$ at the
baseline, and is subject to independent censoring by $C\sim
\textup{Uniform}(0,c)$ where $c$ was chosen to yield approximately 85\%
censoring.
The regression coefficients are $\theta=(\theta_1,\theta_2)$ with
$\theta_1=\theta_2 \in \{0,\log (1.2),\log(2)\}$.
The auxiliary binary variable $U$ is correlated with $Z_1$ through sensitivity $P(U=1|Z_1=1)=.9$  and specificity $P(U=0|Z_1=0)=.9$. 

We considered four scenarios based on the formation of data sources. 
Data sources are
$\mathcal{V}^{(1)} =\{V:Z_2\geq -1\}$ and $\mathcal{V}^{(2)}
=\{V:Z_2\leq 1\}$ in Scenario 1 and 
$\mathcal{V}^{(1)}=\mathcal{V}$ and $\mathcal{V}^{(2)}=\{V:Z_2\leq 1\}$
in Scenario 2.
Sampling fractions in both scenarios were 20\% and 30\%.
In Scenario 3, data sources are $\mathcal{V}^{(1)}
=\mathcal{V}$ and $\mathcal{V}^{(2)}
=\{V:\Delta= 1\}$ with sampling fractions 20\% and 100\%.
Scenario 4 has three sources where membership in the first
two were determined via multinomial logistic
regression with $Z_2$ as a covariate and the third source is $\mathcal{V}^{(3)}
=\{V:\Delta= 1\}$.
Sampling fractions were 10\%, 10\% and 100\%, respectively. Average
sample sizes and numbers of duplications over 2000 datasets in each
scenario are shown in Table \ref{tab:cox1}.

The Monte Carlo sample bias and standard deviation of the proposed
estimator with $\rho $ from Proposition \ref{prop:optBer} are reported
in Table \ref{tab:cox2}.
The results show that bias is small and standard
deviations are close to averages of plug-in estimators of standard
errors in each setting.
In Figure \ref{fig:2}, right panels show that averages of absolute deviations
$\lVert \hat{\theta}_N-\theta_0\rVert $ are proportional to
$1/N^{1/2}$, and Q-Q plots of the scaled estimators indicate their
distributions are approximately the standard normal distribution.
Table \ref{tab:cox3} displays comparison of three different calibration methods in
Section \ref{sec:4} and two other choices of $\rho$ (the extension of the
single-frame estimator of \cite{Kalton1986} studied by
\cite{MR2324141} (SF), and balanced weights of the inverse of the
number of sources to which an item belongs (B)) in Scenario 4.
Results show that the estimator with the proposed weights $\rho$ and
calibration achieved the smallest standard deviations in all cases. 
All of the above results provide numerical support for our theory.
Discussion of additional results and a plug-in variance estimator is
provided in the Appendix \ref{sec:data}.
Note that our estimator did not lose much efficiency compared to the MLE
for complete data if we base comparison on the number of items used
for estimation. For example, $2933$ items with duplication were used for our estimator
on average when $N=10000$ in Scenario 4 and its standard deviations are .0789 and .0407 with $(\alpha,\beta)=(.2,.5)$ and $\theta_1=\theta_2=\log(2)$.
In this case, standard deviations of the MLE based on 2933 items are
expected to be about 
.0986 and .0499. 



\subsection{Real Data Example}
We illustrate our methods  with the national Wilms tumor study (NWTS)
\cite{pmid2544249} where 
3915 patients with Wilms tumor were followed until the disease
progression.
Data contain complete information of all subjects, and was
used to compare different stratified designs in
\cite{BreslowChatterjee1999}. 
To compare our methods with the MLE with the full cohort and the
weighted likelihood estimator with stratified sampling
\cite{MR2325244}, we
randomly divided the dataset into two, applied 
three methods with different designs to training data, and computed the
partial likelihood based on testing data.  
Data sources are deceased subjects,
subjects with unfavorable histology measured at hospitals subject to misclassification, and the entire cohort with
sampling fractions 100\%, 50\% and 10\% resulting in selecting 506
subjects with duplications (438 without duplication).
Strata for stratified sampling are deceased subjects, living
subjects with unfavorable histology and the rest with sampling
fractions 100\%, 50\% and 14\% yielding 502 selected subjects.   
We fitted data to the Cox model to identify predictors of the relapse of Wilms tumor. Results of the MLE is considered to be most reliable.
Estimates from merged and stratified data were all
similar to the MLE except the one for cancer stage. 
Estimated standard errors of the proposed estimator were smaller than
those of 
the estimator with balanced $\rho$ but larger than those from
stratified data because stratified sampling effectively used
information by avoiding duplication at the design stage. 
These differences, however, were rather small relative to the
magnitudes of estimates even when making comparison
with the MLE (except cancer stage). 
The partial likelihood at the proposed estimator shows better fit than
in stratified sampling though the estimator with balanced $\rho$
yielded a larger value.
Overall, good performance of the proposed estimator illustrates the
potential of data integration as an alternative to stratified sampling.

\begin{table}[h]\scriptsize
\label{tab:cox4}
\begin{tabular}{lcccccccc}
&\multicolumn{2}{c}{Full cohort}& \multicolumn{4}{c}{Data integration}&\multicolumn{2}{c}{Stratified sampling}\\
$\rho$& &&\multicolumn{2}{c}{Proposed}&  \multicolumn{2}{c}{Balanced} & \\\hline
\# subjects&1957 &&  \multicolumn{4}{l}{438 (506 with duplication) } &502&\\
Duplication&  0&& \multicolumn{2}{l}{64 (twice)}&\multicolumn{2}{l}{ 2
                                                  (thrice)}&0&\\\hline
Partial likelihood& -2458.8& &-2464.7 & & -2463.2& &-2467.2& \\\hline
Variable& $\hat{\theta}$& SE & $\hat{\theta}$& SE &$\hat{\theta}$& SE &  $\hat{\theta}$& SE \\\hline
Histology&1.430&0.125&1.243 &0.236 &1.383&0.268 &1.419&0.190\\
Age&0.084&0.021&0.045&0.043& 0.043&0.047&0.110&0.035\\
Stage (III/IV)&1.506&0.356&2.680&0.761&2.589&0.848&2.157&0.705\\
Tumor&0.064&0.020&0.082 &0.046&0.076&0.052&0.106&0.041\\
Stage
  $\times$Tumor&-0.079&0.029&-0.156&0.061&-0.079&0.068&-0.134&0.055\\\hline
\multicolumn{8}{l}{Note: Histology is measured at a central
  laboratory.}
\end{tabular}
\caption{Point estimates and estimated standard errors in the analysis of the NWTS study with different
  sampling schemes.}
\end{table}

\section{Discussion}
\label{sec:7}
We developed large sample theory for merged data from multiple sources.
We proved two limit theorems for the H-empirical process, 
and applied them to study asymptotic properties of infinite-dimensional $M$-estimation.
Our theory is a non-trivial extension of empirical process theory to a
dependent and biased sample with duplication.

We adopted Hartley's estimator as a building block for our theory.
This estimator and its variants have been extensively studied under multiple-frame surveys in sampling theory.
To conclude this paper, we briefly describe two approaches in sampling
theory to illustrate differences from ours.

A primary difference lies in probabilistic frameworks. 
Sampling theory 
adopts a finite-population framework where randomness arises only from selection of units. 
Parameters are finite-population parameters such as
sample averages, and statistical models
are outside the scope.
Asymptotic results usually assume the existence of CLT 
a priori and asymptotic variance is defined as limits of 
deterministic sequences (see e.g.\@
\cite{MR2324141,Metcalf2009,MR2796566,MR1147105,MR1394091}).
This difference leads to different optimal $\rho$ and calibration as seen above. 

Another less common approach called
the super-population framework \cite{MR0386084, MR648029, MR2253102} adopts
a similar two-stage formulation \cite{MR2253102} but two qualitatively distinct sets of conditions are assumed for different stages of sampling. 
These conditions concern specific random and non-random vectors
instead of treating a class of functions in a systematic way.
Applications 
are thus
limited to (generalized) linear models \cite{Lu2012,Metcalf2009}
where variance estimators
 (p. 4690 of \cite{Lu2012}, p. 1514 of
\cite{Metcalf2009}) are our variance estimator for the first
stage only. 
This seeming discrepancy reflects a distinction in probabilistic frameworks.

\par\noindent
{\bf Acknowledgements:}
We owe thanks to the Associate Editor and two anonymous referees for
their constructive suggestions, which significantly improved the paper. 
We also thank Jon Wellner for helpful discussions of empirical process theory.


\appendix

\section{U-LLN and U-CLT}
\label{sec:uclt}
We prove Theorems \ref{thm:gc} and \ref{thm:D} with rigor.
Define the finite sampling empirical measure and process for each source by 
\begin{eqnarray*}
&&\mathbb{P}_{N^{(j)}}^{R(j)} \equiv \frac{1}{N^{(j)}} \sum_{i=1}^{N^{(j)}}R_{(j),i}^{(j)} \delta_{(X_{(j),i},V_{(j),i})} ,\\
&&\mathbb{G}^{R(j)}_{N^{(j)}} \equiv\sqrt{N^{(j)}}\left(\mathbb{P}_{N^{(j)}}^{R(j)} -\frac{n^{(j)}}{N^{(j)}}\mathbb{P}_{N^{(j)}}^{(j)} \right), \quad j=1,\ldots,J,
\end{eqnarray*} 
where $\mathbb{P}_{N^{(j)}}^{(j)} = (1/N^{(j)})\sum_{i=1}^{N^{(j)}}
\delta_{(X_{(j),i,},V_{(j),i})}$ is the empirical measure restricted
to source $j$.
Note that
\begin{equation*}
\mathbb{P}_{N^{(j)}}^{R(j)}=\frac{n^{(j)}}{N^{(j)}}\hat{\mathbb{P}}_{n^{(j)}}^{(j)},
\end{equation*}
and that $\hat{\mathbb{P}}_{n^{(j)}}^{(j)}$ and
$\mathbb{P}_{N^{(j)}}^{(j)}$ have been defined in (\ref{eqn:hempj}).
Each finite sampling empirical process is an exchangeably weighted bootstrap empirical process  with $R^{(j)}$s viewed as the bootstrap weights.
Note that $\rho$ is not included in these definitions.
Also define the index set as
\begin{equation*}
\tilde{\mathcal{F}}_j\equiv\{\tilde{f}^{(j)}:\tilde{f}^{(j)}(x,v) =\rho^{(j)}(v) f(x), f\in\mathcal{F}\},\quad j=1,\ldots, J.
\end{equation*}
With this notation, we obtain the following decomposition
\begin{eqnarray}
\mathbb{G}_N^{\textup{H}} f
&=&\mathbb{G}_N  f+ \sum_{j=1}^J\sqrt{\frac{N^{(j)}}{N}} \frac{N^{(j)}}{n^{(j)}}\mathbb{G}_{N^{(j)}}^{R (j)}\tilde{f}^{(j)} \label{eqn:decompmult}
\end{eqnarray}
where $\tilde{f}^{(j)} \equiv \rho^{(j)}f\in\tilde{\mathcal{F}}_j$.
Recall that $\mathbb{G}_N=\sqrt{N}(\mathbb{P}_N-\tilde{P}_0)$ is the
empirical process in Section \ref{sec:3}.

\begin{proof}[Proof of Theorem \ref{thm:gc}]
Using the decomposition (\ref{eqn:decompmult}), the triangle inequality yields
\begin{eqnarray*}
\lVert \mathbb{P}_N^{\textup{H}} - \tilde{P}_0\rVert_{\mathcal{F}}
\leq \lVert \mathbb{P}_N - \tilde{P}_0\rVert_{\mathcal{F}}
+ \sum_{j=1}^J \frac{N^{(j)}}{N} \frac{N^{(j)}}{n^{(j)}}\left\lVert \mathbb{P}_{N^{(j)}}^{R (j)} -\frac{n^{(j)}}{N^{(j)}} \mathbb{P}_{N^{(j)}}\right\rVert_{\tilde{\mathcal{F}}_j}
\end{eqnarray*}
The first term is $o_{P^*}(1)$ by the Glivenko-Cantelli theorem. For the second term, note that $N^{(j)}/N\rightarrow_{P^*} \nu^{(j)}$ by the weak law of large numbers and $n^{(j)}/N^{(j)}\rightarrow p^{(j)}>0$ by assumption. 
To show $\lVert \mathbb{P}_{N^{(j)}}^{R (j)} -(n^{(j)}/N^{(j)})
\mathbb{P}^{(j)}_{N^{(j)}}\rVert_{\tilde{\mathcal{F}}_j}=o_{P^*}(1)$,
we apply the bootstrap Glivenko-Cantelli theorem (Lemma 3.6.16 of
\cite{MR1385671}) in view of $\mathbb{P}_{N^{(j)}}^{R (j)}$ as the exchangeably weighted bootstrap empirical process.
It is easy to see that sampling indicators $(R^{(j)}_{(j),1},\ldots,R^{(j)}_{(j),N^{(j)}})$ satisfy the condition (3.6.8) of \cite{MR1385671} for the exchangeable bootstrap weights. 
We use the unconditional version of the theorem which replaces the probability measure for bootstrap weights by the joint probability measure of data and bootstrap weights in Lemma 3.6.16 of \cite{MR1385671}.
This result is easily obtained by replacing the conditional multiplier inequality by the unconditional multiplier inequality in Lemma 3.6.7 of \cite{MR1385671} in the proof of Lemma 3.6.16 of \cite{MR1385671}.
Now, it suffices to show that $\tilde{\mathcal{F}}_j$ are Glivenko-Cantelli classes.
Since $\mathcal{F}$ is $\tilde{P}_0$-Glivenko-Cantelli 
and $\rho^{(j)}$ are bounded,
the desired result follows from the Glivenko-Cantelli preservation theorem (Proposition 2 of \cite{MR1857319}).
\end{proof}

\begin{proof}[Proof of Theorem \ref{thm:D}]
The first term $\mathbb{G}_N$ in (\ref{eqn:decompmult})  weakly converges to the $P_0$-Brownian bridge process $\mathbb{G}$ by the usual Donsker theorem. 
For finite sampling empirical processes in the second term, note that
the classes $\tilde{\mathcal{F}}_j$ are Donsker classes since
$\rho^{(j)}$ are measurable and bounded, and $\mathcal{F}$ is Donsker (see Example 2.10.10 of \cite{MR1385671}).
We apply the unconditional version of Theorem 3.6.13 of \cite{MR1385671} to obtain $\mathbb{G}^{R(j)}_{N^{(j)}}\rightsquigarrow\sqrt{p^{(j)}(1-p^{(j)})}\mathbb{G}^{(j)}$ in $\ell^{\infty}(\tilde{\mathcal{F}}_j)$ where $\mathbb{G}^{(j)}$ are the $P_0^{(j)}$-Brownian bridge processes.
To prove the unconditional result, first obtain the same finite dimensional convergence by verifying the Lindeberg-Feller condition in the proof of Lemma 3.6.15 of \cite{MR1385671} with sample average replaced by expectation. Then, use the unconditional multiplier inequality in Lemma 3.6.7 of \cite{MR1385671} in the proof of Theorem 3.6.13 of \cite{MR1385671}. 

These limiting processes can be viewed as the stochastic processes indexed by $\mathcal{F}$ in $\ell^{\infty}(\mathcal{F}) $ because elements of $\tilde{\mathcal{F}}_j$ can be identified uniquely by $\mathcal{F}$ and $\rho$ is bounded.
Since $N^{(j)}/N\rightarrow_{P^*} \nu^{(j)}$ by the LLN and $n^{(j)}/N^{(j)}\rightarrow p^{(j)}$ by assumption, we obtain the first expression in the theorem. 

For the covariance function, it suffices to show that all limiting processes, $\mathbb{G}$, $\mathbb{G}^{(j)}$, $j=1,\ldots,J,$ are independent.
Since convergence in $\ell^{\infty}(\mathcal{F})$ implies marginal convergence, this reduces to computing the limit of covariance among $\mathbb{G}_N$ and  $\mathbb{G}_{N^{(j)}}^{R(j)},j=1,\ldots,J,$ evaluated at arbitrary functions  $f,g\in\mathcal{F}$.
Let $\underline{X}_N=(X_1,\ldots,X_N)$, and $\underline{V}_N=(V_1,\ldots,V_N)$. We have
\begin{eqnarray*}
\textup{Cov}(\mathbb{G}_{N^{(j)}}^{R(j)}f,\mathbb{G}_{N^{(j')}}^{R(j')}g) 
&=& \textup{Cov}(E[\mathbb{G}_{N^{(j)}}^{R(j)}f|\underline{X}_N,\underline{V}_N],E[\mathbb{G}_{N^{(j')}}^{R(j')}g|\underline{X}_N,\underline{V}_N])\\
&&+E[\textup{Cov}(\mathbb{G}_{N^{(j)}}^{R(j)}f,\mathbb{G}_{N^{(j')}}^{R(j')}g|\underline{X}_N,\underline{V}_N)].
\end{eqnarray*}
Because $E[R_{(j),i}^{(j)}|\underline{X}_N,\underline{V}_N]=n^{(j)}/N^{(j)}$, $E[\mathbb{G}_{N^{(j)}}^{R(j)}f|\underline{X}_N,\underline{V}_N]=0$. 
Independence of $R^{(j)}_{(j),i}$ and $R^{(j')}_{(j'),i}$ yields $\textup{Cov}(\mathbb{G}_{N^{(j)}}^{R(j)}f,\mathbb{G}_{N^{(j')}}^{R(j')}g|\underline{X}_N,\underline{V}_N)=0$. 
Since limiting processes $\mathbb{G}^{(j)}$ and $\mathbb{G}^{(j)'}$ are Gaussian, they are independent.
Independence of $\mathbb{G}$ and $\mathbb{G}^{(j)}$ can be similarly proved.
\end{proof}

\begin{proof}[Proof of Corollary \ref{cor:fpD}]
Apply Theorem 3.6.13 of \cite{MR1385671} conditionally to the second terms of the decomposition (\ref{eqn:decompmult}). 
\end{proof}

\begin{proof}[Proof of Theorem \ref{thm:BerD}]
We apply the usual Donsker theorem. 
The class of measurable functions
$\{(x,v,r^{(j)})\mapsto\sum_{j=1}^Jr^{(j)}\rho^{(j)}(v)/p^{(j)}f(x):r^{(j)}\in\{0,1\},f\in\mathcal{F}\}$
is a Donsker class by Theorem 2.10.6 of \cite{MR1385671} since
$\mathcal{F}$ is Donsker and $\rho^{(j)}$ are measurable and bounded. 
The limiting process $\mathbb{G}^{\textup{H,Ber}}$ of $\mathbb{G}^{\textup{H,Ber}}_N$ is a Brownian bridge process. 
Simplify its covariance function to $v^{\textup{Ber}}(f,g)$ by the law of total covariance.
\end{proof}

\begin{proof}[Proof of Proposition \ref{prop:optBer}]
It follows from Theorem \ref{thm:BerD} that the term that involves with $(c^{(1)},c^{(2)})$ in the asymptotic variance is
\begin{equation*}
\left\{od(p^{(1)})\left(c^{(1)}\right)^2+od(p^{(2)})\left(c^{(2)}\right)^2\right\}P_0f^{\otimes 2}I\{V\in \mathcal{V}^{(1)}\cap \mathcal{V}^{(2)}\}.
\end{equation*}
Since the matrix in the last display is positive definite, it suffices to minimize the quantity in the parenthesis with a constraint $c^{(1)}+c^{(2)}=1$. 
The case for $J>2$ is similar.
This completes the proof.
\end{proof}

\begin{proof}[Proof of Proposition \ref{prop:optWR}]
The proof is similar to the proof of Proposition
\ref{prop:optBer} above and is omitted.
\end{proof}

\section{Calibration}
\label{sec:cal}
\begin{prop}
\label{prop:alphaCal}
Under Condition \ref{cond:cal}, $\hat{\alpha}_N^{\#}\rightarrow_{P^*} 0$ with $\#\in\{\textup{c,sc}\}$ and 
\begin{eqnarray*}
&&\sqrt{N}\hat{\alpha}_N^{\textup{c}} \rightsquigarrow -\dot{G}(0)^{-1}\left\{P_0 V^{\otimes 2}\right\}^{-1}
         \sum_{j=1}^J\sqrt{\nu^{(j)}}\sqrt{\frac{1-p^{(j)}}{p^{(j)}}}\mathbb{G}^{(j)}\rho^{(j)}V,\\
&&\sqrt{N^{(j)}}\hat{\alpha}_N^{\textup{sc},(j)} \\
&&\rightsquigarrow -\dot{G}(0)^{-1}\left\{\textup{Var}_{0}^{(j)}\left( \rho^{(j)}V\right) \right\}^{-1}
         \sqrt{\frac{1-p^{(j)}}{p^{(j)}}}\mathbb{G}^{(j)}\left(\rho^{(j)}V-P_0^{(j)}\rho^{(j)}V\right).
\end{eqnarray*}
Here $\rho^{(j)}V$ is understood to be $\rho^{(j)}(V)V$.
\end{prop}

\begin{proof}
For $\hat{\alpha}_N^{\textup{c}}$, define $\Phi_{N,\textup{c}}(\alpha)\equiv \mathbb{P}_N^{\textup{H}} G(V^T\alpha)V -\mathbb{P}_NV$ 
and $\Phi_{\textup{c}}(\alpha)\equiv P_0[\{G(V^T\alpha)-1\}V]$.
Note that $\Phi_{N,\textup{c}}(\hat{\alpha}_N^{\textup{c}})=0$ and $\Phi_{\textup{c}}(0)=0$.
We apply Theorem 5.7 of \cite{MR1652247} for a consistency proof. 
The first condition of the theorem is the supremum of $\left | \Phi_{N,\textup{c}}(\alpha)-\Phi_{\textup{c}}(\alpha)\right| $ over $\alpha\in\mathbb{R}^k$ is $o_{P^*}(1)$. 
Here $|\cdot|$ is the Euclidean norm.
The triangle inequality yields
\begin{eqnarray*}
\sup_{\alpha\in\mathbb{R}^k}\left| \Phi_{N,\textup{c}}(\alpha)-\Phi_{\textup{c}}(\alpha)\right| 
&\leq&\sup_{\alpha\in\mathbb{R}^k}\left|(\mathbb{P}_N^{\textup{H}}-P_0)G(V^T\alpha)V\right| \nonumber
 +\left| (\mathbb{P}_N-P_0)V\right |.
\end{eqnarray*}
The  class $\left\{v\mapsto G(v^T\alpha)v:  \alpha \in \mathbb{R}^k\right\}$ of functions on $\mathcal{V}$ is a $P_0$-Glivenko-Cantelli class (see the proof of Proposition A.1 of \cite{SW2013supp}). 
Thus, the first term is $o_{P^*}(1)$ by Theorem \ref{thm:gc}.
The second term is $o_{P^*}(1)$ by the weak law of large numbers.
This verifies the first condition.
The second condition of the theorem is that for any $\epsilon>0$,
$\inf_{|\alpha|>\epsilon}\left| \Phi_{\textup{c}}(\alpha)\right|>0$, which was established in the proof of Proposition A.1 of \cite{SW2013supp}. 
This proves that $\hat{\alpha}_N^{\textup{c}}\rightarrow_{P^*}0$. 

We apply Theorem 3.3.1 of \cite{MR1385671} to show the asymptotic normality of $\hat{\alpha}_N^{\textup{c}}$. 
For the asymptotic equicontinuity condition, Taylor's theorem yields
\begin{eqnarray*}
&&\sqrt{N}(\Phi_{N,\textup{c}}-\Phi_{\textup{c}})(\hat{\alpha}_N^{\textup{c}})-\sqrt{N}(\Phi_{N,\textup{c}}-\Phi_{c})(0)\\
&&=(\mathbb{P}_N^{\textup{H}}-P_0)\dot{G}(V^T\alpha^*)V^{\otimes 2}\sqrt{N}(\hat{\alpha}_N^{\textup{c}})
\end{eqnarray*}
for some $\alpha^*$ with $|\alpha^*-0|\leq |\hat{\alpha}_N^{\textup{c}}-0|$.
This term is $o_{P^*}(1+\sqrt{N}|\hat{\alpha}_N^{\textup{c}}|)$ because $(\mathbb{P}_N^{\textup{H}}-P_0)V^{\otimes 2}\dot{G}(V^T\alpha) \rightarrow_{P^*} 0$, uniformly in $\alpha$.
To see this, note that we need to show 
$\{v\mapsto v^{\otimes 2}\dot{G}(v^T\alpha):\alpha\in\mathbb{R}^k\}$ is a  $P_0$-Glivenko-Cantelli class in order to apply our Glivenko-Cantelli theorem (Theorem \ref{thm:gc}), but this was proved in the proof of Proposition A.1 of \cite{SW2013supp}). 
Hence, this verifies the asymptotic equicontinuity condition.
We show the weak convergence of the process $\sqrt{N}(\Phi_{N,\textup{c}}-\Phi_{\textup{c}})(\alpha)$ at $\alpha=0$.
Corollary \ref{cor:fpD} yields
\begin{eqnarray*}
&&\sqrt{N}(\Phi_{N,\textup{c}}-\Phi_{\textup{c}})(0)
=\sqrt{N}(\mathbb{P}_N^{\textup{H}}-\mathbb{P}_N) V
\rightsquigarrow 
\sum_{j=1}^J\sqrt{\nu^{(j)}}\sqrt{\frac{1-p^{(j)}}{p^{(j)}}}\mathbb{G}^{(j)}\rho^{(j)}V.
\end{eqnarray*}
The Fr\'{e}chet derivative $\dot{\Phi}_{\textup{c}}(0)$ of $\Phi_{\textup{c}}$ at $0$ is $\dot{G}(0)P_0V^{\otimes 2}$. 
It follows by Theorem 3.3.1 of \cite{MR1385671} that 
\begin{eqnarray*}
\sqrt{N}\hat{\alpha}_N^{\textup{c}}
&=& -\dot{\Phi}_{\textup{c}}(0)^{-1}\sqrt{N}(\Phi_{N,\textup{c}}-\Phi_{\textup{c}})(0)+o_{P*}(1)\\
&\rightsquigarrow& -\dot{G}(0)^{-1}\left\{P_0 V^{\otimes 2}\right\}^{-1}
\sum_{j=1}^J\sqrt{\nu^{(j)}}\sqrt{\frac{1-p^{(j)}}{p^{(j)}}}\mathbb{G}^{(j)}\rho^{(j)}V.
\end{eqnarray*}

For $\hat{\alpha}^{\textup{sc},(j)}_N$, note that this can be viewed
as the solution to the centered calibration with a single stratum with
probability measure $P_0^{(j)}$ for the Horvitz-Thompson estimator of
$\rho^{(j)}V$ in a two-phase stratified sample studied by
\cite{MR3059418}. Desired consistency and asymptotic normality follow
from this observation. 
\end{proof}

\begin{proof}[Proof of Theorem \ref{thm:calD}]
First, we consider $\mathbb{G}_N^{\textup{H,c}}$. 
Note that $\mathbb{G}_N^{\textup{H,c}} f= \sqrt{N}(\mathbb{P}_N^{\textup{H}} G(V^T\hat{\alpha}^{\textup{c}})f-P_0f)$.
For finite-dimensional convergence, we have for a fixed function $f\in\mathcal{F}$ that
\begin{eqnarray*}
\mathbb{G}_N^{\textup{H,c}}f &=&  \mathbb{G}_N^{\textup{H}}f+(\mathbb{G}_N^{\textup{H,c}}-\mathbb{G}_N^{\textup{H}}) f\\
&=& \mathbb{G}_N^{\textup{H}}f+\mathbb{G}_N^{\textup{H}}\{G(V^T\hat{\alpha}_N^{\textup{c}})-1\}f+\sqrt{N}P_0\{G(V^T\hat{\alpha}_N^{\textup{c}})-1\}f \\
&=&\mathbb{G}_N^{\textup{H}}f+\{(\mathbb{P}_N^{\textup{H}}-P_0)+P_0\}\dot{G}(V^T\alpha^*)fV^T\sqrt{N}\hat{\alpha}^{\textup{c}}_N,
\end{eqnarray*}
for some $\alpha^*$ with $|\alpha^*-0|\leq |\hat{\alpha}_N^{\textup{c}}-0|$.
Because of boundedness of $\dot{G}$, integrability of $fV$ (by
Cauchy-Schwarz inequality; here and other place) and $\sqrt{N}\hat{\alpha}_N^{\textup{c}} = O_{P^*}(1)$ by Proposition \ref{prop:alphaCal}, the second term in the display is $\dot{G}(0)P_0f(X)V^T \sqrt{N}\hat{\alpha}_N^{\textup{c}}+o_{P^*}(1)$ by our Glivenko-Cantelli theorem (Theorem \ref{thm:gc}) and dominated convergence theorem. 
It follows from Theorem \ref{thm:D} and Proposition \ref{prop:alphaCal} that $\mathbb{G}_N^{\textup{H,c}}f$ converges in distribution to
\begin{eqnarray*}
\mathbb{G}f +  \sum_{j=1}^J\sqrt{\nu^{(j)}}\sqrt{\frac{1-p^{(j)}}{p^{(j)}}} \mathbb{G}^{(j)} \left[\rho^{(j)}f-P_0(fV^T)\{P_0V^{\otimes 2}\}^{-1}\rho^{(j)}V\right].
\end{eqnarray*}
For asymptotic equicontinuity of $\mathbb{G}_N^{\textup{H,c}} =
\mathbb{G}_N^{\textup{H}}+(\mathbb{G}_N^{\textup{H,c}}-\mathbb{G}_N^{\textup{H}})$,
first note that $\mathbb{G}_N^{\textup{H}}$ is asymptotically
equicontinuous with respect to the metric $d^2(f,g) \equiv \upsilon(f-g,f-g)$
by Theorem \ref{thm:D}.
Recall from Theorem \ref{thm:D} that 
\begin{eqnarray*}
d^2(f,g) &=& \upsilon(f-g,f-g) \\
&=& \textup{Var}_0\{f(X)-g(X)\} + \sum_{j=1}^J \nu^{(j)}\frac{1-p^{(j)}}{p^{(j)}} \textup{Var}_{0}^{(j)}\{\rho^{(j)}(V)\{f(X)-g(X)\}.
\end{eqnarray*}
Because $\textup{Var}_0(f-g)$ in $d^2(f,g)$ corresponds to $\mathbb{G}_N$ in the decomposition (\ref{eqn:decompmult}), $\mathbb{G}_N^{\textup{H}}$ is still asymptotically equicontinuous with respect to the metric 
\begin{equation*}
d_{\textup{c}}^2(f,g) = P_0(f-g)^2  + \sum_{j=1}^{J}\sqrt{\nu^{(j)}}\sqrt{\frac{1-p^{(j)}}{p^{(j)}}}\textup{Var}_{0}^{(j)}(f-g),
\end{equation*}
which replaces $\textup{Var}_0(f-g)$ by $P_0(f-g)^2$ in $d^2(f,g)$, in view of 2.1.2 of \cite{MR1385671} and our assumption $\lVert P_0\rVert_{\mathcal{F}}<\infty$.
Define the class $\mathcal{F}_{\delta_N}\equiv \{f-g:f,g\in\mathcal{F},d_{\textup{c}}(f,g)\leq \delta_N\}$ of functions for an arbitrary sequence $\delta_N\downarrow 0$.
Proceeding in the same way as above using Taylor's theorem, $\lVert \mathbb{G}_N^{\textup{H,c}} - \mathbb{G}_N^{\textup{H}}\rVert_{\mathcal{F}_{\delta_N}}$ is bounded by
\begin{eqnarray*}
\lVert \mathbb{P}_N^{\textup{H}}-P_0\rVert_{\mathcal{F}_{\delta_N}} O_{P^*}(1) + \sup_{h\in\mathcal{F}_{\delta_N}}| P_0 \dot{G}(V^T\alpha^*)hV^T|O_{P^*}(1).
\end{eqnarray*}
for some $\alpha^*$ with $|\alpha^*-0|\leq |\hat{\alpha}_N^{\textup{c}}-0|$.
Since $\mathcal{F}_{\delta_N}$ is contained in a
$P_0$-Glivenko-Cantelli class for $N$ sufficiently large
(e.g. $\mathcal{F}_1$ when $\delta_N<1$), the fist term in the last
display is $o_{P^*}(1)$ by our Glivenko-Cantelli theorem (Theorem
\ref{thm:gc}). Applying the Cauchy-Schwarz inequality and then
the dominated convergence theorem, the second term is bounded by
$\dot{G}(0)P_0V^{\otimes 2}\lVert
P_0h^2\rVert_{\mathcal{F}_{\delta_N}}O_{P^*}(1)$.
Since $h=f-g\in\mathcal{F}_{\delta_N}$ and $d_{\textup{c}}(f,g)\rightarrow 0$, $P_0h^2 =P_0(f-g)^2\rightarrow 0$. This established asymptotic equicontinuity of $\mathbb{G}_N^{\textup{H,c}}$ with respect to the metric $d_{\textup{c}}$.

Next, we consider $\mathbb{G}_N^{\textup{H,sc}}$. 
Note that 
\begin{equation*}
\mathbb{P}_N^{\textup{H,sc}}f= \sum_{j=1}^J\frac{N^{(j)}}{N} \frac{N^{(j)}}{n^{(j)}}\mathbb{P}_{N^{(j)}}^{R(j)}G^{(j)}_{\hat{\alpha}^{\textup{sc},(j)}}\tilde{f}^{(j)}.
\end{equation*}
For finite-dimensional convergence, the decomposition (\ref{eqn:decompmult}) yields
\begin{eqnarray*}
\mathbb{G}_N^{\textup{H,sc}}f
&=& \mathbb{G}_N^{\textup{H}} f +\mathbb{G}_N^{\textup{H,sc}} f-\mathbb{G}_N^{\textup{H}} f\\
&=&\mathbb{G}_N^{\textup{H}} f+\sqrt{N}\sum_{j=1}^J\frac{N^{(j)}}{N}\frac{N^{(j)}}{n^{(j)}}\mathbb{P}_{N^{(j)}}^{R(j)}(G^{(j)}_{\hat{\alpha}_N^{\textup{sc},(j)}}-1)\tilde{f}^{(j)}.
\end{eqnarray*}
The first term $\mathbb{G}_N^{\textup{H}}$ weakly converges to $\mathbb{G}^{\textup{H}}$ by Theorem \ref{thm:D}.
The second term can be written as
\begin{eqnarray*}
\sqrt{N}\sum_{j=1}^J\frac{N^{(j)}}{N}\left\{\frac{N^{(j)}}{n^{(j)}}\left(\mathbb{P}_{N^{(j)}}^{R(j)}-\frac{n^{(j)}}{N^{(j)}}\mathbb{P}_{N^{(j)}}^{(j)}\right)+\mathbb{P}_{N^{(j)}}^{(j)}\right\}\left(G^{(j)}_{\hat{\alpha}_N^{\textup{sc},(j)}}-1\right)\tilde{f}^{(j)}
\end{eqnarray*}
Apply Taylor's theorem to obtain  
\begin{eqnarray*}
&&\sum_{j=1}^J\left\{
\sqrt{\frac{N^{(j)}}{N}}\frac{N^{(j)}}{n^{(j)}}\left\{\mathbb{P}_{N^{(j)}}^{R(j)}-\frac{n^{(j)}}{N^{(j)}}\mathbb{P}_{N^{(j)}}^{(j)}\right\}\dot{G}^{(j)}_{\alpha^{*,(j)}}\tilde{f}^{(j)}\left(\tilde{V}^{(j)}-\mathbb{P}_{N^{(j)}}^{(j)}\tilde{V}^{(j)}\right)^T
\right.\\
&&+\left.
\sqrt{\frac{N^{(j)}}{N}} \mathbb{P}_{N^{(j)}}^{(j)}\dot{G}^{(j)}_{\alpha^{*,(j)}}\tilde{f}^{(j)}\left(\tilde{V}^{(j)}-\mathbb{P}_{N^{(j)}}^{(j)}\tilde{V}^{(j)}\right)^T
\right\}\sqrt{N^{(j)}}\hat{\alpha}_N^{\textup{sc},(j)}
\end{eqnarray*}
for some $\alpha^{*,(j)}$ with $|\alpha^{*,(j)}-0|\leq |\hat{\alpha}_N^{\textup{sc},(j)}-0|$.
The first term in the summation is $o_{P^*}(1)$ because we can apply the bootstrap Glivenko-Cantelli theorem to $\mathbb{P}_{N^{(j)}}^{R(j)}-(n^{(j)}/N^{(j)})\mathbb{P}_{N^{(j)}}^{(j)}$ as in the proof of Theorem \ref{thm:gc}.
For the second term, note that $\mathbb{P}_{N^{(j)}}^{(j)}$ is the empirical measure conditional on $V\in\mathcal{V}^{(j)}$. Thus, we can apply the usual Glivenko-Cantelli theorem and then the dominated convergence theorem together with  $N^{(j)}/N\rightarrow \nu^{(j)}$ to see the last display is
\begin{equation*}
 \dot{G}(0)\sum_{j=1}^J\sqrt{\nu^{(j)}}P_0^{(j)}\tilde{f}^{(j)}\left(\tilde{V}^{(j)}-P_0^{(j)}\tilde{V}^{(j)}\right)^T\sqrt{N^{(j)}}\hat{\alpha}_N^{\textup{sc},(j)} + o_{P^*}(1). 
\end{equation*} 
It follows from Proposition \ref{prop:alphaCal} that this is
\begin{eqnarray*} 
&&-\sum_{j=1}^J\sqrt{\nu^{(j)}}\sqrt{\frac{1-p^{(j)}}{p^{(j)}}}P_0^{(j)}\left\{\tilde{f}^{(j)}\left(\tilde{V}^{(j)}-P_0^{(j)}\tilde{V}^{(j)}\right)^T\right\}\left\{\textup{Var}_{0}^{(j)}\left(\tilde{V}^{(j)}\right)\right\}^{-1}\\
&&\qquad \qquad\times\mathbb{G}^{(j)}\left(\tilde{V}^{(j)}-P_0^{(j)}\tilde{V}^{(j)}\right) + o_{P^*}(1) \\
&=& -\sum_{j=1}^J\sqrt{\nu^{(j)}}\sqrt{\frac{1-p^{(j)}}{p^{(j)}}}\mathbb{G}^{(j)}Q_{\textup{sc}}^{(j)}f+ o_{P^*}(1),
\end{eqnarray*}
Combine this with $\mathbb{G}_N^{\textup{H}} f= \mathbb{G}^{\textup{H}} f+o_{P^*}(1)$ and apply Theorem \ref{thm:D} to conclude $\mathbb{G}_N^{\textup{H,sc}}f\rightsquigarrow \mathbb{G}^{\textup{H,sc}}f$.
The asymptotic equicontinuity of $\mathbb{G}_N^{\textup{H,sc}}$ with respect to $d_{\textup{c}}$ can be proved in a similar way to that of $\mathbb{G}_N^{\textup{H,c}}$.
\end{proof}

\section{Infinite Dimensional $M$-Estimation}
\begin{proof}[Proof of Theorem \ref{thm:zthm1}]
We prove the claim for $\hat{\theta}_{N,\textup{sc}}$.
First, Theorem \ref{thm:calD} together with Condition \ref{cond:zthm1-1} yields
\begin{equation*}
\mathbb{G}^{\textup{H,sc}}_NB_{\theta_0}
\rightsquigarrow
\mathbb{G}^{\textup{H,sc}}B_{\theta_0}, \quad \mbox{in }\ell^\infty(\mathcal{H}).
\end{equation*}

For a fixed arbitrary sequence $\{\delta_N\}$ with $\delta_N\rightarrow 0$, let
\begin{eqnarray*}
\mathcal{D}_N
&\equiv&\left\{\frac{B_\theta(h)-B_{\theta_0}(h)}{1+\sqrt{N}\lVert\theta-\theta_0\rVert}:h\in\mathcal{H},\lVert\theta-\theta_0\rVert \leq \delta_N\right\}\\
&\equiv&\left\{B_N(\theta,\theta_0)[h]:h\in\mathcal{H},\lVert\theta-\theta_0\rVert \leq \delta_N\right\}.
\end{eqnarray*}
Condition \ref{cond:zthm1-3} yields $\lVert
\mathbb{G}_N\rVert_{\mathcal{D}_N}=o_{P^*}(1)$, which implies $E\lVert \mathbb{G}_N\rVert_{\mathcal{D}_N}=o(1)$ as $N\rightarrow \infty$ by Lemma \ref{lemma:vw239} since  $E^*\lVert \delta_X-P_0\rVert_{\mathcal{D}_N} \leq 2E^*\sup_{\theta\in\Theta,h\in\mathcal{H}} |(\delta_X-P_0)B_\theta(h)|<\infty$ by Condition \ref{cond:zthm1-1} regarding integrable envelope.
It follows by Lemma \ref{lemma:gpleqg} that $E\lVert \mathbb{G}_N^{\textup{H}}\rVert_{\mathcal{D}_N}=o(1)$ and hence
$\lVert \mathbb{G}_N^{\textup{H}}\rVert_{\mathcal{D}_N}=o_{P^*}(1)$ by Markov's inequality.
Since $\mathcal{D}_N$ is $P_0$-Glivenko-Cantelli, apply Taylor's theorem as in the proof of Theorem \ref{thm:calD} and then the dominated convergence theorem to obtain $\lVert \mathbb{G}_N^{\textup{H,sc}}f-\mathbb{G}_N^{\textup{H}}\rVert_{\mathcal{D}_N}=o_{P^*}(1)$.
Thus, consistency of $\hat{\theta}_{N,\textup{sc}}$ to $\theta_0$ and Condition \ref{cond:zthm1-3} imply that
\begin{eqnarray}
\label{eqn:zthmaec}
\lVert \mathbb{G}_N^{\textup{H,sc}} (B_{\hat{\theta}_{N,\textup{sc}}}-B_{\theta_0})\rVert_{\mathcal{H}} = o_{P^*}(1+\sqrt{N}\lVert \hat{\theta}_{N,\textup{sc}}-\theta_0\rVert).
\end{eqnarray}

We prove $\sqrt{N}\lVert \hat{\theta}_{N,\textup{sc}}-\theta_0\rVert =O_{P^*}(1)$.
Because  $\lVert \mathbb{P}_N^{\textup{H,sc}}B_{\hat{\theta}_{N,\textup{sc}}}\rVert_{\mathcal{H}}=o_{P^*}(N^{-1/2})$ and $P_0B_{\theta_0} =0$, it follows from (\ref{eqn:zthmaec}) that
\begin{eqnarray}
&&\sqrt{N}(\Psi(\hat{\theta}_{N,\textup{sc}})-\Psi(\theta_0))\nonumber\\
&&=\sqrt{N} (\Psi(\hat{\theta}_{N,\textup{sc}}) -\Psi^{\textup{H,sc}}_N (\hat{\theta}_{N,\textup{sc}}) ) + o_{P^*}(1) \nonumber\\
&&=-\sqrt{N} ( \Psi^{\textup{H,sc}}_N (\theta_{0})-\Psi(\theta_{0}) ) + o_{P^*}(1+\sqrt{N}\lVert \hat{\theta}_{N,\textup{sc}}-\theta_0\rVert). \label{eqn:zthmaec2}
\end{eqnarray}
Since the continuous invertibility of $\dot{\Psi}_0$ at $\theta_0$ implies that there is some constant $c>0$ such that
$(c+o_P(1))\lVert \hat{\theta}_{N,\textup{sc}}-\theta_0\rVert \leq \lVert \Psi(\hat{\theta}_{N,\textup{sc}})-\Psi(\theta_0)\rVert_{\mathcal{H}}$,
we have
\begin{eqnarray*}
&&(c+o_P(1))\sqrt{N}\lVert \hat{\theta}_{N,\textup{sc}}-\theta_0\rVert
\leq \lVert \sqrt{N}(\Psi(\hat{\theta}_{N,\textup{sc}})-\Psi(\theta_0))\rVert_{\mathcal{H}}\\
&&\leq \lVert \mathbb{G}_N^{\textup{H,sc}} B_{\theta_0}\rVert_{\mathcal{H}}
+o_{P^*}(1+\sqrt{N}\lVert \hat{\theta}_{N,\textup{sc}}-\theta_0\rVert).
\end{eqnarray*}
Since $\lVert \mathbb{G}_N^{\textup{H,sc}} B_{\theta_0}\rVert_{\mathcal{H}}=O_{P^*}(1)$ by Condition \ref{cond:zthm1-1} and Theorem \ref{thm:D},
the claim $\sqrt{N}\lVert \hat{\theta}_{N,\textup{sc}}-\theta_0\rVert =O_{P^*}(1)$ follows.

We prove the asymptotic normality of $\hat{\theta}_{N,\textup{sc}}$.
It follows from the Fr\'{e}chet differentiability of $\Psi$ and $\sqrt{N}$-consistency of $\hat{\theta}_{N,\textup{sc}}$ that (\ref{eqn:zthmaec2}) becomes
\begin{equation}
\label{eqn:wz414dc}
\sqrt{N}\dot{\Psi}_0(\hat{\theta}_{N,\textup{sc}}-\theta_{0})
=-\mathbb{G}_N^{\textup{H,sc}} B_{\theta_0} +o_{P^*}(1).
\end{equation}
Continuity of the inverse $\dot{\Psi}_0^{-1}$ of $\dot{\Psi}_0$, the continuous mapping theorem and weak convergence of $\mathbb{G}_N^{\textup{H,sc}}$ by Theorem \ref{thm:calD} yield the weak convergence of $\sqrt{N}(\hat{\theta}_{N,\textup{sc}}-\theta_0)$:
$\sqrt{N}(\hat{\theta}_{N,\textup{sc}}-\theta_{0})
=-\dot{\Psi}_0^{-1}\mathbb{G}^{\textup{H,sc}} B_{\theta_0}+o_{P^*_1}(1)$.
This establishes the theorem for $\hat{\theta}_{N,\textup{sc}}$. 
Proofs for other cases are similar and omitted.
\end{proof}

\begin{proof}[Proof of Theorem \ref{thm:zthm2}]
For $\hat{\theta}_{N,\textup{sc}}$, we have
\begin{eqnarray*}
&&\sqrt{N}\mathbb{P}_N^{\textup{H,sc}}\dot{\ell}_{0} 
+\sqrt{N}P_0\dot{\ell}_{\hat{\theta}_{N,\textup{sc}},\hat{\eta}_{N,\textup{sc}}}=o_{P^*}(1),\\
&&\sqrt{N}\mathbb{P}_N^{\textup{H,sc}}B_{0}\left[\underline{h}_0\right]
+\sqrt{N}P_0B_{\hat{\theta}_{N,\textup{sc}},\hat{\eta}_{N,\textup{sc}}}[\underline{h}_0]=o_{P^*}(1).
\end{eqnarray*}
To see this, note that $\sqrt{N}\mathbb{P}_N^{\textup{H,sc}}\dot{\ell}_{0} 
  +\sqrt{N}P_0\dot{\ell}_{\hat{\theta}_{N,\textup{sc}},\hat{\eta}_{N,\textup{sc}}}
=-\mathbb{G}_N^{\textup{H,sc}}( \dot{\ell}_{\hat{\theta}_{N,\textup{sc}},\hat{\eta}_{N,\textup{sc}}}-\dot{\ell}_{0}) +o_{P^*}(1)$ because $\mathbb{P}_N^{\textup{H,sc}}\dot{\ell}_{\hat{\theta}_{N,\textup{sc}},\hat{\eta}_{N,\textup{sc}}} = o_{P^*}(N^{-1/2})$ 
by assumption and $P_0\dot{\ell}_{0}=0$.
Let $\delta_N\downarrow 0$ be arbitrary and define 
$\mathcal{F}_N \equiv \{\dot{\ell}_{\theta,\eta}-\dot{\ell}_{0}:
    |\theta-\theta_0|\leq \delta_N,\lVert\eta-\eta_0\rVert\leq N^{-\beta}\}$. 
Apply Lemma \ref{lemma:wza5} with Condition \ref{cond:wlenonreg3} to obtain 
$\lVert \mathbb{G}_N^{\textup{H,sc}}\rVert_{\mathcal{F}_N}=o_{P^*}(1)$. 

It follows from Condition \ref{cond:wlenonreg4} that
\begin{eqnarray}
\label{eqn:huanga}
&&P_0\left[-\dot{\ell}_{0}\left\{\dot{\ell}_{0}^T(\hat{\theta}_{N,\textup{sc}}-\theta_0)
        +B_{0}(\hat{\eta}_{N,\textup{sc}}-\eta_0)\right\}\right]\nonumber\\
&&\quad + \ o\left(|\hat{\theta}_{N,\textup{sc}}-\theta_0|\right)+O\left(\lVert \hat{\eta}_{N,\textup{sc}}-\eta_0\rVert^{\alpha}\right) 
          +\mathbb{P}_N^{\textup{H,sc}}\dot{\ell}_{0}\nonumber\\
&&=P_0\left[-\dot{\ell}_{0}\left\{\dot{\ell}_{0}^T(\hat{\theta}_{N,\textup{sc}}-\theta_0)
          +B_{0}(\hat{\eta}_{N,\textup{sc}}-\eta_0)\right\} -\dot{\ell}_{\hat{\theta}_{N,\textup{sc}},\hat{\eta}_{N,\textup{sc}}}
            +\dot{\ell}_{0}\right]\nonumber\\
&&\quad +\ o\left(|\hat{\theta}_{N,\textup{sc}}-\theta_0|\right)+O\left(\lVert \hat{\eta}_{N,\textup{sc}}-\eta_0\rVert^{\alpha}\right) 
        +P_0\dot{\ell}_{\hat{\theta}_{N,\textup{sc}},\hat{\eta}_{N,\textup{sc}}}+\mathbb{P}_N^{\textup{H,sc}}\dot{\ell}_{0}\nonumber\\
&& =o_{P^*}(N^{-1/2}), 
\end{eqnarray}
and, furthermore, that 
\begin{eqnarray}
\label{eqn:huangb}
&&P_0\left[-B_{0}\left[\underline{h}_0\right]\left\{\dot{\ell}_{0}^T(\hat{\theta}_{N,\textup{sc}}-\theta_0)
       +B_{0}(\hat{\eta}_{N,\textup{sc}}-\eta_0)\right\}\right]\nonumber \\
&& \quad + \ o\left(|\hat{\theta}_{N,\textup{sc}}-\theta_0|\right)+O\left(\lVert \hat{\eta}_{N,\textup{sc}}-\eta_0\rVert^{\alpha}\right)
        +\mathbb{P}_N^{\textup{H,sc}}B_{0}\left[\underline{h}_0\right]\nonumber\\
&&=o_{P^*}(N^{-1/2}).
\end{eqnarray}
Taking the difference of (\ref{eqn:huanga}) and (\ref{eqn:huangb}) yields
\begin{eqnarray*}
&&-P_0\left(\left\{\dot{\ell}_{0}-B_{0}
       \left[\underline{h}_0\right]\right\}\dot{\ell}_{0}^T\right)\left(\hat{\theta}_{N,\textup{sc}}-\theta_0\right)
+ \ o\left(|\hat{\theta}_{N,\textup{sc}}-\theta_0|\right) \\
&& \quad+ o_{P^*}(N^{-1/2})- o_{P^*}(N^{-1/2})
       +\mathbb{P}_N^{\textup{H,sc}}\left(\dot{\ell}_{0}-B_{0}\left[\underline{h}_0\right]\right)\\ &&
=o_{P^*}(N^{-1/2})-o_{P^*}(N^{-1/2}),
\end{eqnarray*}
or
\begin{equation*}
-I_0 (\hat{\theta}_{N,\textup{sc}}-\theta_0) 
=\mathbb{P}_N^{\textup{H,sc}}\left(\dot{\ell}_{0}
   -B_{0}\left[\underline{h}_0\right]\right) 
    +o_{P^*}(N^{-1/2}).
\end{equation*}
Here we used Condition \ref{cond:wlenonreg2}  and the fact that $\sqrt{N}O_{P^*}\left(\lVert \hat{\eta}_N-\eta_0\rVert^\alpha\right)=o_{P^*}(1)$ (Condition \ref{cond:wlenonreg1} and $\alpha\beta >1/2$). 

It follows from the invertibility of $I_0$ and the definition of the efficient influence function that 
\begin{eqnarray*}
\sqrt{N}\left(\hat{\theta}_{N,\textup{sc}}-\theta_0\right)
=-\sqrt{N}\mathbb{P}_N^{\textup{H,sc}}I_0^{-1}\tilde{\ell}_{0}+o_{P^*}(1).
\end{eqnarray*}
Apply Theorem \ref{thm:calD} to complete the proof.
\end{proof}
\begin{proof}[Proof of Theorem \ref{thm:const}]
This follows from Corollary 3.2.3 of \cite{MR1385671}.
\end{proof}

\begin{proof}[Proof of Theorem \ref{thm:rate}]
Apply Lemma \ref{lemma:gpleqg} to obtain
$
E^*\left\lVert \mathbb{G}_N^{\textup{H}}\right\rVert_{\mathcal{M}_\delta}
\lesssim
E^*\left\lVert \mathbb{G}_N\right\rVert_{\mathcal{M}_\delta}
\leq \phi_N(\delta),
$
and then apply Theorem 3.2.5 of \cite{MR1385671}.
\end{proof}

\section{Stratified Sampling}
\label{sec:st}
In this section, we consider stratified sampling at the second stage where each dataset from a source is obtained by stratified sampling without replacement.
Each source $\mathcal{V}^{(j)}$ is partitioned into disjoint strata. For source $j$, there are $K_j$ strata $\mathcal{S}_1^{(j)},\ldots, \mathcal{S}_{K_j}^{(j)}$ with $\sum_{k=1}^{K_j} \mathcal{S}_k^{(j)} = \mathcal{V}_j$.
The $k$th stratum $S^{(j)}_k$ in source $j$ consists of $N_{k}^{(j)}$ units with $\sum_{k=1}^{K_j} N_{k}^{(j)} = N^{(j)}$.
A subsample of size  $n_k^{(j)}$ is drawn without replacement from the stratum $\mathcal{S}_k^{(j)}$.
With the sampling indicator $R_i^{(j)}$ for source $j$, sampling probability for unit $i$ is $\pi^{(j)}(V_i)=n_k^{(j)}/N_k^{(j)}$ if $V_i \in \mathcal{S}_k^{(j)}$.
We assume $\pi^{(j)}(v) \geq c>0$ for some constant $c$ and $\pi^{(j)}(v)\rightarrow p_k^{(j)}$ for $v\in\mathcal{S}_k^{(j)}$ as $N\rightarrow \infty$ for $k=1,\ldots,K_j$ with $j=1,\ldots,J$. 
We write $\nu_k^{(j)}\equiv P(V\in \mathcal{S}_k^{(j)}|V\in \mathcal{V}_j)$, and $P_{0|k}^{(j)} (\cdot) \equiv \tilde{P}_{0}(\cdot|V \in \mathcal{S}_k^{(j)})$.
The H-empirical measure $\mathbb{P}_N^{\textup{H}}$ and process $\mathbb{G}_N^{\textup{H}}$ are defined in the same way. 

The following are uniform LLN and CLT for merged data with stratified sampling. Proofs are similar to those for simple random sampling with the help of asymptotic results in \cite{MR2325244, MR3059418}  and omitted.
\begin{thm}
\label{thm:stGC}
Suppose $\mathcal{F}$ is  $P_0$-Glivenko-Cantelli.
Then 
\begin{equation*}
\lVert \mathbb{P}_N^{\textup{H}} - \tilde{P}_0\rVert_{\mathcal{F}} = \sup_{f\in\mathcal{F}}|(\mathbb{P}_N^{\textup{H}} -\tilde{P}_0)f|\rightarrow_{P^*} 0.
\end{equation*}
\end{thm}

\begin{thm}
\label{thm:stD}
Suppose $\mathcal{F}$ is $P_0$-Donsker.
Then 
\begin{equation*}
\mathbb{G}_N^{\textup{H}}(\cdot) \rightsquigarrow \mathbb{G}^{\textup{H}}(\cdot) \equiv \mathbb{G}(\cdot) + \sum_{j=1}^J \sum_{k=1}^{K_j}\sqrt{\nu_k^{(j)}} \sqrt{\frac{1-p_k^{(j)}}{p_k^{(j)}}}\mathbb{G}_{k}^{(j)}(\rho^{(j)}\cdot)
\end{equation*}
in $\ell^{\infty}(\mathcal{F})$  where the $P_0$-Brownian bridge process $\mathbb{G}$ and the $P_{0|k}^{(j)}$-Brownian bridge processes $\mathbb{G}_{k}^{(j)}$ are independent.
The covariance function $\upsilon^{\textup{st}}(\cdot,\cdot) = \textup{Cov}(\mathbb{G}^{\textup{H}} \cdot,\mathbb{G}^{\textup{H}} \cdot)$ on $\mathcal{F}\times \mathcal{F}$ is given by
\begin{eqnarray*}
&&\upsilon^{\textup{st}}(f,g)
=P_0(f-P_0f)(g-P_0g)^T\\
&&+ \sum_{j=1}^J \sum_{k=1}^{K_j}\nu_k^{(j)}\frac{1-p_k^{(j)}}{p_k^{(j)}}P_{0|k}^{(j)}\left(\rho^{(j)} f-P_{0|k}^{(j)}\rho^{(j)}f\right)\left(\rho^{(j)}g-P_{0|k}^{(j)}\rho^{(j)}g\right)^T.
\end{eqnarray*}
In particular, the asymptotic variance of $\mathbb{G}_N^{\textup{H}}f$ is
\begin{equation*}
\upsilon^{\textup{st}}(f,f) = \textup{Var}_{0}(f) + \sum_{j=1}^J \sum_{k=1}^{K_j}\nu_k^{(j)}\frac{1-p_k^{(j)}}{p_k^{(j)}}\textup{Var}_{0|k}^{(j)}\left(\rho^{(j)} f\right).
\end{equation*}
where $\textup{Var}_{0|k}^{(j)}$ is the variance under  $P_{0|k}^{(j)}$.
\end{thm}

\section{Auxiliary results}
\begin{lemma}
\label{lemma:gpleqg}
For an arbitrary set $\mathcal{F}$ of integrable functions,
\begin{equation*}
E^*\left\lVert \mathbb{G}_N^{\textup{H}}\right\rVert_{\mathcal{F}}
\lesssim E^*\left\lVert \mathbb{G}_N\right\rVert_{\mathcal{F}},
\end{equation*}
where $a\lesssim b$ means $a \leq K b$ for some constant $K \in (0,\infty)$.
\end{lemma}
\begin{proof}
Note that the finite sampling empirical process for each source is equivalent to the finite sampling empirical process for each stratum ((11) of \cite{MR2325244}) for stratified sampling. 
Because the H-empirical process admits a similar decomposition
(compare (\ref{eqn:decompmult}) and (10) of \cite{MR2325244}), this
lemma can be proved in the same way as in the proof of Lemma A.2 of \cite{SW2013supp} if $\lVert \mathbb{G}_{j,N_j}\rVert_{\tilde{\mathcal{F}}_j}\lesssim \lVert \mathbb{G}_{j,N_j}\rVert_{\mathcal{F}}$ where $\mathbb{G}_{j,N_j}\equiv \sqrt{N_j}(\mathbb{P}_{N^{(j)}}^{(j)}-P_0^{(j)})$. This inequality is easily proved by Jensen's inequality with $E\left\{1-\rho^{(j)}(V)\right\}(f(X)-P_0^{(j)}f)I\{V\in\mathcal{V}^{(j)}\}=0$.
\end{proof}

\begin{lemma}
\label{lemma:wza5}
Let $\mathcal{F}_N$ be a sequence of decreasing classes of functions such that 
$\lVert \mathbb{G}_N\rVert_{\mathcal{F}_N}=o_{P^*}(1)$. 
Assume that there exists an integrable envelope for $\mathcal{F}_{N_0}$ for some $N_0$. 
Then $E\lVert\mathbb{G}_N\rVert_{\mathcal{F}_N}\rightarrow 0$ as $N\rightarrow \infty$.
As a consequence, $\lVert \mathbb{G}_N^{\textup{H}}\rVert_{\mathcal{F}_N}=o_{P^*}(1)$. 

Suppose, moreover, that $\mathcal{F}_N$ is $P_0$-Glivenko-Cantelli  with 
$\| P_0 \|_{{\cal F}_{N_1}} < \infty$ for some $N_1$,
and that every $f_N \in\mathcal{F}_N$ converges to zero either pointwise or in 
$L_1 (P_0)$ as $N\rightarrow \infty$. 
Then $\lVert \mathbb{G}_N^{\textup{H},\#}\rVert_{\mathcal{F}_N}=o_{P^*}(1)$ with $\#\in\{\textup{c,sc}\}$,
assuming Condition \ref{cond:cal}.  
\end{lemma}
\begin{proof}
We apply Lemma \ref{lemma:vw239} with $\ZZ_{i}$ and $\mathcal{F}_N$ in 
Lemma~\ref{lemma:vw239} replaced by $\delta_{(X_i,V_i)}-P_0$.
The uniform boundedness condition of Lemma~\ref{lemma:vw239} is satisfied, 
because $E^*\lVert \delta_{(X_1,V_1)}-P_0\rVert_{\mathcal{F}_N}<\infty$ for $N\geq N_0$, 
and this expectation is decreasing in $N\geq N_0$.
Thus, 
$E^*\lVert \mathbb{G}_N\rVert_{\mathcal{F}_N}\rightarrow 0$. 
Apply Lemma \ref{lemma:gpleqg}, and Markov's inequality to obtain 
$\lVert \mathbb{G}_N^{\textup{H}}\rVert_{\mathcal{F}_N}=o_{P^*}(1)$.

For $\mathbb{G}_N^{\textup{H},\#}$ with $\#\in\{\textup{c,sc}\}$, apply Taylor's theorem as in the proof of Theorem \ref{thm:calD} and then the dominated convergence theorem to conclude $\lVert \mathbb{G}_N^{\textup{H},\#}-\mathbb{G}_N^{\textup{H}}\rVert_{\mathcal{F}_N}=o_{P^*}(1)$.
The triangle inequality yields 
$\left\lVert\mathbb{G}_N^{\textup{H},\#}\right\rVert_{\mathcal{F}_{\delta_N}}=o_{P^*}(1)$.
\end{proof}

The following is Lemma A.4 of \cite{SW2013supp} with correction that $\mathbb{S}_N=N^{-1/2}\sum_{i=1}^N\mathbb{Z}_{i}$ instead of $\mathbb{S}_N=\sum_{i=1}^N\mathbb{Z}_{i}$.
\begin{lemma}
\label{lemma:vw239}
Let $\mathbb{Z}_{1},\mathbb{Z}_{2},\ldots$ be i.i.d.\@ stochastic processes indexed
by $\mathcal{F}_N$ with $E^*\lVert \mathbb{Z}_{1}\rVert_{\mathcal{F}_N}$ uniformly bounded in $N$.
Suppose that  $\lVert \mathbb{S}_N\rVert_{\mathcal{F}_N} \ $\\
$ \ \equiv\lVert N^{-1/2}\sum_{i=1}^N\mathbb{Z}_{i}\rVert_{\mathcal{F}_N}=o_{P^*}(1)$.
Then
$E^*\lVert \mathbb{S}_N\rVert_{\mathcal{F}_N}\rightarrow 0$, as $N\rightarrow \infty$.
\end{lemma}

\section{Unknown Sample Size}
\label{sec:N}
In this section, we briefly discuss the case where $N$ is unknown but
$N^{(j)},j=1,\ldots,J,$ are known. In this case, we can estimate $N$
by
\begin{equation*}
\hat{N} = \sum_{i=1}^N\sum_{j=1}^J \frac{R_i^{(j)}
  \rho^{(j)}(V_i)}{\pi^{(j)}(V_i)} = N\mathbb{P}_N^{\textup{H}}1.
\end{equation*} 
This estimator is unbiased for $N$ and is getting closer to $N$ as
$N\rightarrow\infty$ in the sense that $N/\hat{N}\rightarrow_P1$.
We consider two situations regarding estimation.
In the first case, we are
interested in estimating the mean of $f(X)$ without knowing $N$. This can be done by
$\mathbb{P}_N^{\textup{H}}f$ if $N$ is known. 
We can estimate $\theta\equiv P_0f$ without known $N$ by replacing $N$ in
$\mathbb{P}_N^{\textup{H}}f$ by $\hat{N}$:
\begin{equation*}
\hat{\theta}_N = \frac{1}{\hat{N}}\sum_{i=1}^N\sum_{j=1}^J
  \frac{R^{(j)}_i\rho^{(j)}(V_i)}{\pi^{(j)}(V_i)} f(X_i) = \left(\mathbb{P}_N^{\textup{H}}1\right)^{-1}\mathbb{P}_N^{\textup{H}}f.
\end{equation*}
This modified estimator is
consistent for $P_0f$. To see this note that
\begin{eqnarray*}
&&\hat{\theta}_N
=\mathbb{P}_N^{\textup{H}}f+\left\{\left(\mathbb{P}_N^{\textup{H}}1\right)^{-1}-1\right\}\mathbb{P}_N^{\textup{H}}f\rightarrow_p P_0f,
\end{eqnarray*}
since $\mathbb{P}_N^{\textup{H}}f\rightarrow_P P_0f$ and
$\mathbb{P}_N^{\textup{H}}1\rightarrow_P1$ by Theorem \ref{thm:gc}.
For asymptotic normality, the delta method and
$\mathbb{P}_N^{\textup{H}}1\rightarrow_P1$ yield
\begin{eqnarray*}
&&\sqrt{N}(\hat{\theta}_N-\theta)\\
&&=\sqrt{N}\left(\mathbb{P}_N^{\textup{H}}1\right)^{-1}\left(\mathbb{P}_N^{\textup{H}}f-P_0f\right)
+\sqrt{N}\left\{\left(\mathbb{P}_N^{\textup{H}}1\right)^{-1}-1\right\}P_0f\\
&&=\sqrt{N}\left(\mathbb{P}_N^{\textup{H}}f-P_0f\right)
-\sqrt{N}(\mathbb{P}_N^{\textup{H}}1-1)P_0f + o_P(1)\\
&&=\mathbb{G}_N^{\textup{H}}(f-P_0f)  + o_P(1)\rightarrow_d \mathbb{G}^{\textup{H}}(f-P_0f).
\end{eqnarray*}
The limiting variable is normally distributed with mean zero and
variance
\begin{equation*}
\textup{AV}\left(\sqrt{N}(\hat{\theta}_N-\theta)\right)=
\textup{Var}_0(f) + \sum_{j=1}^J \nu^{(j)} \frac{1-p^{(j)}}{p^{(j)}}
\textup{Var}_0^{(j)} \left\{\rho^{(j)}(V)(f-P_0f)\right\}.
\end{equation*}
This asymptotic variance is estimated by a plug-in estimator presented
below. For the population
variance $\textup{Var}_0(f)$, we use
\begin{equation*}
\widehat{\textup{Var}}_0(f) =
\frac{1}{\hat{N}}\sum_{i=1}^N\sum_{j=1}^J
\frac{R^{(j)}_i\rho^{(j)}(V_i)}{\pi^{(j)}(V_i)} \{f(X_i)\}^2
-\left\{\frac{1}{\hat{N}}\sum_{i=1}^N\sum_{j=1}^J
\frac{R^{(j)}_i\rho^{(j)}(V_i)}{\pi^{(j)}(V_i)} f(X_i)\right\}^2.
\end{equation*}
For design variances, 
 we estimate $\nu^{(j)}$ and $p^{(j)}$ by 
\begin{equation*}
\hat{\nu}^{(j)} =\frac{N^{(j)}}{\hat{N}} , \quad \hat{p}^{(j)} = \frac{n^{(j)}}{N^{(j)}}.
\end{equation*}
The conditional variance is estimated by
\begin{eqnarray*}
&&  \widehat{\textup{Var}}_0^{(j)}\left(\rho^{(j)}(f-P_0f)\right) \\
&&= \frac{1}{N^{(j)}}
\sum_{i=1}^{N^{(j)}}
\frac{R_{(j),i}^{(j)}}{\pi^{(j)}(V_{(j),i})}\left\{\rho^{(j)}(V_{(j),i})f(X_{(j),i})-\hat{\theta}_N)\right\}^{\otimes
  2}\\
&&\quad -\left\{\frac{1}{N^{(j)}} \sum_{i=1}^{N^{(j)}}
\frac{R_{(j),i}^{(j)}}{\pi^{(j)}(V_{(j),i})}\rho^{(j)}(V_{(j),i})(f(X_{(j),i})-\hat{\theta}_N)\right\}^{\otimes
  2}.
\end{eqnarray*}
In practice, an estimated variance of $\hat{\theta}_N$ is often
reported as an estimated
$\textup{AV}\left(\sqrt{N}(\hat{\theta}_N-\theta)\right)$ divided by
$N$ if $N$ is known. For unknown $N$, we can report estimated
$\textup{AV}\left(\sqrt{N}(\hat{\theta}_N-\theta)\right)$ shown above divided by
$\hat{N}$.

The second case is infinite-dimensional $M$-estimation.
Note that both $M$- and $Z$-estimators can be obtained without known
$N$ since $N$ is a multiplicative factor for a criterion function and
estimating equations. Hence results of consistency, rates of convergence and
asymptotic normality follow without additional changes. The sample size
is needed in variance estimation as above
 but we can simply replace $N$ by $\hat{N}$. The fact
 $1/N-1/\hat{N}\rightarrow_P0$ justifies this replacement.

\section{Numerical Study}
\label{sec:data}
\subsection{Linear regression}
\label{subsec:lin}
We consider the following regression model 
\begin{equation*}
Y=Z^T\theta+e, \quad E[e|Z]=0.
\end{equation*}
where $Y$ is an outcome variable, $Z$ is a vector of covariates, $e$
is an error term and $\theta$ is the vector of regression coefficients.
If we further assume the normality of $e$, this estimation problem
reduces to $M$-estimation for a parametric model discussed in Example
\ref{ex:para}.
Here we consider weighted least squares estimation without assuming
normality and derive asymptotic properties for illustration of our methodology.
The weighted least squares estimator 
\begin{equation*}
\hat{\theta}_N = \left\{\mathbb{P}_N^{\textup{H}}Z^{\otimes 2}\right\}^{-1}\mathbb{P}_N^{\textup{H}}ZY
\end{equation*}
minimizes a criterion function
\begin{equation*}
\mathbb{P}_N^{\textup{H}}m_\theta(Y,Z) = \mathbb{P}_N^{\textup{H}}|Y-Z^T\theta|^2.
\end{equation*}
This is  the same estimator as a solution to the Hartley-type likelihood
equation if we assume the normality of $e$. 
If $P_0Y^2<\infty$ and $P_0|Z^{\otimes 2}|<\infty$, it follows from
Theorem \ref{thm:gc} that 
\begin{equation*}
\hat{\theta}_N = \left\{P_0Z^{\otimes 2}\right\}^{-1}P_0ZY+o_{P}(1) = \theta_0+o_P(1).
\end{equation*} 
For asymptotic normality, we apply Theorem \ref{thm:D} to obtain
\begin{equation*}
\sqrt{N}(\hat{\theta}_N-\theta_0) = \left\{P_0Z^{\otimes
    2}\right\}^{-1} \mathbb{G}_N^{\textup{H}}Ze \rightarrow_d\left\{
P_0Z^{\otimes 2}\right\}^{-1}\mathbb{G}^{\textup{H}} Ze.
\end{equation*}
This limiting variable is a mean-zero normal random vector with
variance 
\begin{equation*}
\textup{Var}_0\left(\left\{P_0Z^{\otimes 2}\right\}^{-1}Ze\right) + \sum_{i=1}^J \nu^{(j)}\frac{1-p^{(j)}}{p^{(j)}} \textup{Var}_0^{(j)}\left(\left\{P_0Z^{\otimes 2}\right\}^{-1}Ze\right)
\end{equation*}
For variance estimation, we use a plug-in estimate of the asymptotic
variance. We estimate the function $\tilde{\ell}_0(x)=
\left\{P_0Z^{\otimes 2}\right\}^{-1}ze$ by
\begin{equation*}
\hat{\tilde{\ell}}_0(X_i) \equiv \left\{\mathbb{P}_N^{\textup{H}}Z^{\otimes
    2}\right\}^{-1}Z_i (Y_i-Z_i^T\hat{\theta}_N),
\end{equation*}
at each selected observation.
The population variance $\textup{Var}_0(\tilde{\ell}_0)$ is estimated
by
\begin{equation*}
\widehat{\textup{Var}_0(\tilde{\ell}_0)} =
\mathbb{P}_N^{\textup{H}}\left\{\hat{\tilde{\ell}}_0(X)\right\}^{\otimes 2}
\end{equation*}
noting that $P_0\tilde{\ell}_0=0$.
For the design variance we estimate $\nu^{(j)}$ and $p^{(j)}$ by 
\begin{equation*}
\hat{\nu}^{(j)} =\frac{N^{(j)}}{N} , \quad \hat{p}^{(j)} = \frac{n^{(j)}}{N^{(j)}}.
\end{equation*}
The conditional variance is estimated by
\begin{eqnarray*}
\widehat{\textup{Var}_0^{(j)}(\rho^{(j)}\tilde{\ell}_0)} &=& \frac{1}{N^{(j)}}
\sum_{i=1}^{N^{(j)}}
\frac{R_{(j),i}^{(j)}}{\pi^{(j)}(V_{(j),i})}\left\{\rho^{(j)}(V_{(j),i})\hat{\tilde{\ell}}_0(X_{(j),i})\right\}^{\otimes
  2}\\
&&-\left\{\frac{1}{N^{(j)}} \sum_{i=1}^{N^{(j)}}
\frac{R_{(j),i}^{(j)}}{\pi^{(j)}(V_{(j),i})}\rho^{(j)}(V_{(j),i})\hat{\tilde{\ell}}_0(X_{(j),i})\right\}^{\otimes
  2}.
\end{eqnarray*}
We use a similar variance estimator for logistic regression and Cox regression
models with appropriate changes of the estimate of $\tilde{\ell}_0$.

For a simulation study, data were generated with a covariate
$Z\sim N(0,1)$ and a normal error $e\sim N(0,1)$. 
The intercept is $\theta_1=1$ and the slope is $\theta_2\in\{0,1/2,1\}$.
The variable $V\in\mathcal{V}=\mathbb{R}$ observed for every item is $V=Z$ which determines
data source membership. For
selected items we observe $X=(Y,Z)$.
We consider two scenarios. In the first scenario, two data sources are
$\mathcal{V}^{(1)} =\{V:Z\geq -1\}$ and $\mathcal{V}^{(2)}
=\{V:Z\leq 1\}$.
In the second scenario, data sources are $\mathcal{V}^{(1)}
=\mathcal{V}$ and $\mathcal{V}^{(2)}
=\{V:Z\leq 1\}$. In either case, we selected 20 percent of items in
$\mathcal{V}^{(1)}$ and 30 percent of items in $\mathcal{V}^{(2)}$. 
In the first scenario, sizes of two data sources are almost
identical. The items in the intersection $\mathcal{V}^{(1)}\cap
\mathcal{V}^{(2)}$ constitute about 68 percent of the entire population.
On average, about 38 percent of items were selected, and among them 12
percent were selected twice. 
In the second scenario, the second data source consists of about 84
percent of the entire population. On average about 41 percent were
selected among whom about 12 percent were selected twice.

Table \ref{tab:ln1} shows Monte Carlo sample bias and standard deviations of our
estimator with $\rho$ from Proposition \ref{prop:optBer}.
Clearly, our estimator has little bias in each setting. The standard
deviations are similar to the average of the plug-in estimator of the
standard error. 
In Figures \ref{fig:lin1} and \ref{fig:lin2}, the average of the absolute deviations are proportional
to $1/N^{1/2}$, which indicates the $\sqrt{N}$-convergence rate of our
estimator. 
Q-Q plots of the scaled estimator
$\sqrt{N}(\hat{\theta}_N-\theta_0)/\widehat{SE}(\hat{\theta})$ show that most
points are concentrated on the line $y=x$, suggesting that the scaled
estimator approximately has the standard normal distribution.

\begin{table}\tiny
\label{tab:ln2}
\begin{tabular}{c|cccccc}
& $N$& $N^{(1)}$& $N^{(2)}$& $n^{(1)}$& $n^{(2)}$& Duplication\\\hline
Scenario 1&500&421&421&85&127&21\\
& 10000&8413&8416&1683&2525&410\\\hline
Scenario 2&500&500&420&100&127&25\\
& 10000&10000&8414&2000&2524&505\\
\end{tabular}
\caption{Sample sizes for the linear regression model based on 2000
  data sets}
\end{table}

\begin{table}\tiny

\begin{tabular}{cc|cc|cc|cc}
$\theta$&& \multicolumn{2}{|c|}{$(1,1)$}&\multicolumn{2}{|c|}{$(1,0.5)$}&\multicolumn{2}{|c|}{$(1,0)$}\\
$N$& &500& 10000 & 500 & 10000 & 500&\multicolumn{1}{c|}{10000}\\\hline\hline

\multicolumn{8}{l}{Complete data (MLE)}\\\hline
$\theta_1$  &  Bias &.0002&.0001&.0002&.0001&.0002&.0001\\
&SD &.0439&.0111&.0439&.0100&.0439&.0100\\\hline
 $\theta_2$ &  Bias & .0009&.0003&.0009&.0002&.0009&.0002\\
&SD &.0438&.0111&.0438&.0100&.0438&.0100\\\hline\hline

\multicolumn{8}{l}{Scenario 1}\\\hline
$\theta_1$  &  Bias &.0014&.0005&.0014&.0005&.0014&.0005\\
&SD &.0781&.0195&.0781&.0178&.0781&.0178\\
&SEE &.0765&.0194&.0765&.0174&.0765&.0174\\\hline
 $\theta_2$ &  Bias &.0033&.0001&,0033&.0002&.0033&.0002\\
&SD &.0883&.0215&.0883&.0196&.0883&.0196\\
&SEE &.0853&.0218&.0853&.0196&.0853&.0196\\\hline\hline

\multicolumn{8}{l}{Scenario 2}\\\hline
$\theta_1$  &  Bias &.0004&.0002&.0004&.0002&.0004&.0002\\
&SD &.0731&.0171&.0731&.0171&.0731&.0171\\
&SEE &.0749&.0169&.0749&.0169&.0749&.0169\\\hline
 $\theta_2$ &  Bias &.0005&.0003&.0005&.0003&.0005&.0003\\
&SD &.0847&.0190&.0847&.0190&.0847&.0190\\
&SEE &.0812&.0186&.0812&.0186&.0812&.0186\\\hline\hline
\multicolumn{8}{l}{Note: Bias, an absolute Monte Carlo sample
  bias; }\\
\multicolumn{8}{l}{SD, a Monte Carlo sample standard deviation; }\\
\multicolumn{8}{l}{SEE, average of a plug-in
  estimator of a standard error.}\\
\end{tabular}
\caption{Performance of $\hat{\theta}$ with different $\theta$ and
  scenarios for the linear regression model. }
\label{tab:ln1}
\end{table}

\begin{figure}[htbp]
\includegraphics[width=12cm,height=7cm]{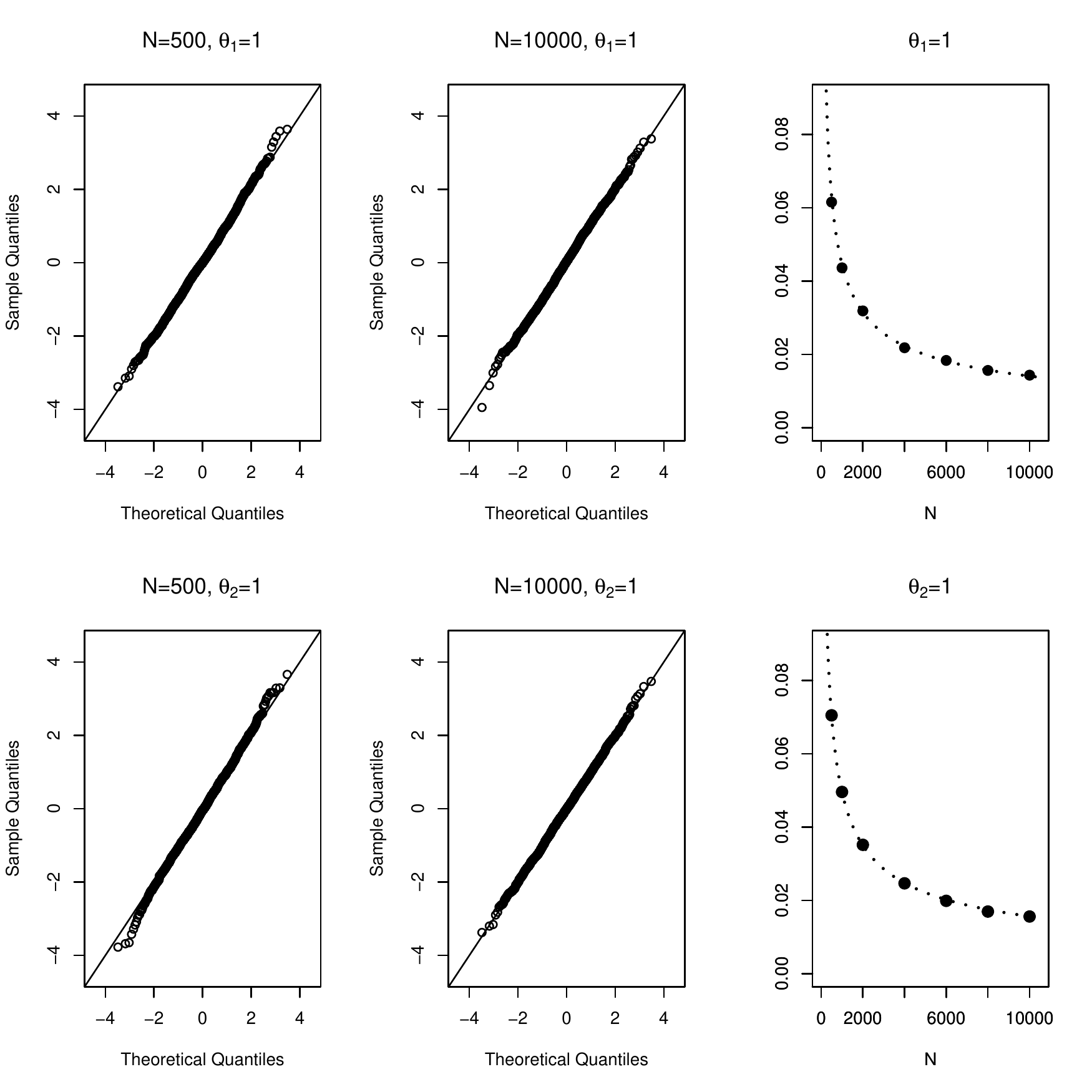}
\caption{Q-Q plots of
  $\sqrt{N}(\hat{\theta}-\theta_0)/\widehat{SE}(\hat{\theta})$ superimposed by
  $y=x$ and plots of average absolute
  differences against $N$ superimposed by $y=c/x^{1/2},c=1.4, 1.6$ for
  linear regression in Scenario 1.}
\label{fig:lin1}
\end{figure}

\begin{figure}[htbp]
\includegraphics[width=12cm,height=7cm]{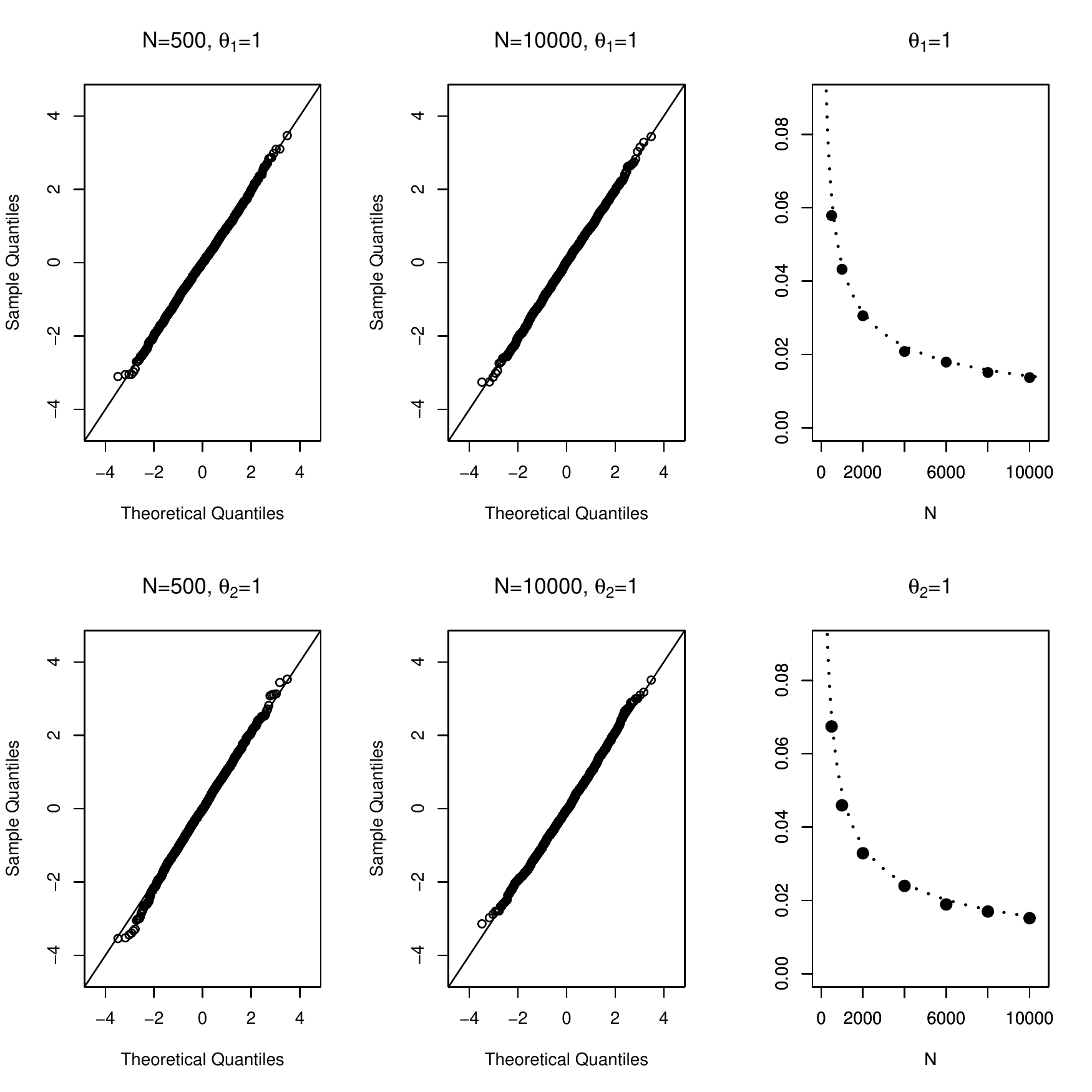}
\caption{Q-Q plots of
  $\sqrt{N}(\hat{\theta}-\theta_0)/\widehat{SE}(\hat{\theta})$  superimposed by
  $y=x$ and plots of average absolute
  differences against $N$ superimposed by $y=c/x^{1/2},c=1.4, 1.6$ for
  linear regression in Scenario 2.}
\label{fig:lin2}
\end{figure}

\newpage
\subsection{Logistic regression}
Next, we consider logistic regression model 
\begin{equation*}
E[Y|Z] = g(Z^T\theta), \quad g(x) = 1/(1+e^{-x}),
\end{equation*}
where $Y$ is a binary random variable,  $Z$ is a vector of
covariates and $\theta$ is the vector of regression coefficients.
Our estimator $\hat{\theta}_N$ solves the Hartley-type likelihood
equation 
\begin{equation*}
\mathbb{P}_N^{\textup{H}} Z\{Y-g(Z^T\theta)\}=0.
\end{equation*}
For consistency of $\hat{\theta}_n$ we introduce a variant of Theorem \ref{thm:const}. 
This theorem is useful when an estimator is formulated as a $Z$-estimator.
\begin{thm}
Suppose that $\mathcal{F}=\{f_\theta(x):\theta\in\Theta\}$ is $P_0$-Glivenko-Cantelli, and that for every $\epsilon>0$
\begin{equation*}
\inf_{\lVert \theta-\theta_0\rVert\geq \epsilon}|P_0 f_\theta|>0=P_0f_{\theta_0}.
\end{equation*}
Then $\lVert \hat{\theta}_N-\theta_0\rVert\rightarrow_P0$ if
$\mathbb{P}_N^{\textup{H}}f_{\hat{\theta}_N} =o_P(1)$.
\end{thm}
The proof follows from a slightly more general form of Theorem
\ref{thm:const}. See Theorems 5.7 and 5.9 of \cite{MR1652247}.
The class of functions $\mathcal{F}=\{f_\theta:\theta\in\Theta\}$ with
$f_\theta(x) = z\{y-g(z^T\theta)\}$ is Glivenko-Cantelli assuming that
$P_0|Z|<\infty$. The second condition in the theorem is satisfied if
$P_0Z^{\otimes 2}$ is positive definite. 
Other conditions of Theorem \ref{thm:zthm1} are easily verified. 
It follows that
\begin{equation*}
\sqrt{N}(\hat{\theta}_N-\theta_0) \rightarrow_d
\mathbb{G}^{\textup{H}} \tilde{\ell}_0
\end{equation*}
where 
\begin{equation*}
\tilde{\ell}_0(x) =\left\{P_0 g(Z^T\theta_0)\{1-g(Z^T\theta_0)\}Z^{\otimes 2}\right\}^{-1}\{y-g(z^T\theta_0)\}.
\end{equation*}
The function $\tilde{\ell}_0(x)$ is estimated by
\begin{equation*}
\hat{\tilde{\ell}}_0(x) =\left\{\mathbb{P}_N^{\textup{H}} g(Z^T\hat{\theta}_N)\{1-g(Z^T\hat{\theta}_N)\}Z^{\otimes 2}\right\}^{-1}\{y-g(z^T\hat{\theta}_N)\}.
\end{equation*}
This is used for our plug-in estimator of the asymptotic variance of $\hat{\theta}_N$.

For a simulation study, data were generated with a covariate
$Z\sim N(0,1)$ and a normal error $e\sim N(0,1)$. 
The regression coefficients are $\theta=(\theta_1,\theta_2)$ where $\theta_1$ was chosen so that the overall
prevalence was approximately 15\%, and $\theta_2\in\{0,\log(1.2),\log(2)\}$.
We considered three scenarios. The first two are the same as scenarios
considered for linear regression in Section \ref{subsec:lin}.
The third scenario has data sources $\mathcal{V}^{(1)}=\mathcal{V}=\{0,1\}$
and $\mathcal{V}^{(2)}=\{V:Y=1\}$ where $V=Y$ and $X=(Y,Z)$. 
Results are summarized in Table \ref{tab:lg1} and Figures \ref{fig:lg1}-\ref{fig:lg3}.
These agree with what are expected from our theory as in linear regression.

\begin{table}\tiny
\label{tab:lg2}
\begin{tabular}{c|cccccc}
& $N$& $N^{(1)}$& $N^{(2)}$& $n^{(1)}$& $n^{(2)}$& Duplication\\\hline
Scenario 1&500&421&421&85&127&21\\
& 10000&8413&8413&1683&2524&410\\\hline
Scenario 2&500&500&421&100&127&26\\
& 10000&10000&8414&2000&2525&505\\\hline
Scenario 3& 500&500&102&100&102&20\\
&10000&10000&2030&2000&2030&407\\
\end{tabular}
\caption{Sample sizes for the logistic regression model based on 2000
  data sets}
\end{table}

\begin{table}\tiny

\begin{tabular}{cc|cc|cc|cc}
$\theta$&& \multicolumn{2}{|c|}{$(-1.5,\log (2))$}&\multicolumn{2}{|c|}{$(-1.38,\log (1.2))$}&\multicolumn{2}{|c|}{$(-1.36,0)$}\\
$N$& &500& 10000 & 500 & 10000 & 500&\multicolumn{1}{c|}{10000}\\\hline\hline

\multicolumn{8}{l}{Complete data (MLE)}\\\hline
$\theta_1$  &  Bias &.007&.0007& .003&.0006&.004&.0008\\
&SD &.125&.0287&.113&.0253&.112&.0250\\\hline
 $\theta_2$ &  Bias & .007&.0003&.002&.0003&.001&.0004\\
&SD &.127&.0279&.114&.0247&.113&.0241\\\hline\hline

\multicolumn{8}{l}{Scenario 1}\\\hline
$\theta_1$  &  Bias &.010&.0012&.016&.0016&.017&.0017\\
&SD & .211&.0465&.198&.0441&.199&.0445\\
&SEE &.209&.0463&.195&.0432&.194&.0431\\\hline
 $\theta_2$ &  Bias &.005&.0010&.007&.0009&.005&.0002\\
&SD &.246&.0531&.221&.0479&.218&.0472\\
&SEE &.245&.0547&.218&.0494&.215&.0486\\\hline\hline

\multicolumn{8}{l}{Scenario 2}\\\hline
$\theta_1$  &  Bias &.029&.0031&.024&.0017&.022&.0021\\
&SD &.209&.0458&.196&.0427&.195&.0430\\
&SEE &.205&.0450&.191&.0421&.190&.0421 \\\hline
 $\theta_2$ &  Bias &.014&.0021&.004&.0014&.001&.0015\\
&SD &.250&.0535&.216&.0455&.214&.0451\\
&SEE &.240&.0531&.212&.0473&.207&.0463\\\hline\hline

\multicolumn{8}{l}{Scenario 3}\\\hline
$\theta_1$  &  Bias &.004&.0004&.005&.0006&.009&.0097\\
&SD &.134&.0296&.121&.0275&.121&.0267\\
&SEE &.136&.0298&.123&.0271&.121&.0268\\\hline
 $\theta_2$ &  Bias &.020&.0015&.009&.0009&.004&.0010\\
&SD &.183&.0397&.156&.0339&.155&.0332\\
&SEE &.180&.0401&.156&.0340&.153&.0335\\\hline\hline
\multicolumn{8}{l}{Note: Bias, an absolute Monte Carlo sample
  bias; }\\
\multicolumn{8}{l}{SD, a Monte Carlo sample standard deviation;}\\
\multicolumn{8}{l}{SEE, average of a plug-in
  estimator of a standard error.}\\
\end{tabular}
\caption{Performance of $\hat{\theta}$ with different $\theta$ and
  scenarios  for the logistic regression model.  }
\label{tab:lg1}
\end{table}

\begin{figure}[htbp]
\includegraphics[width=12cm,height=7cm]{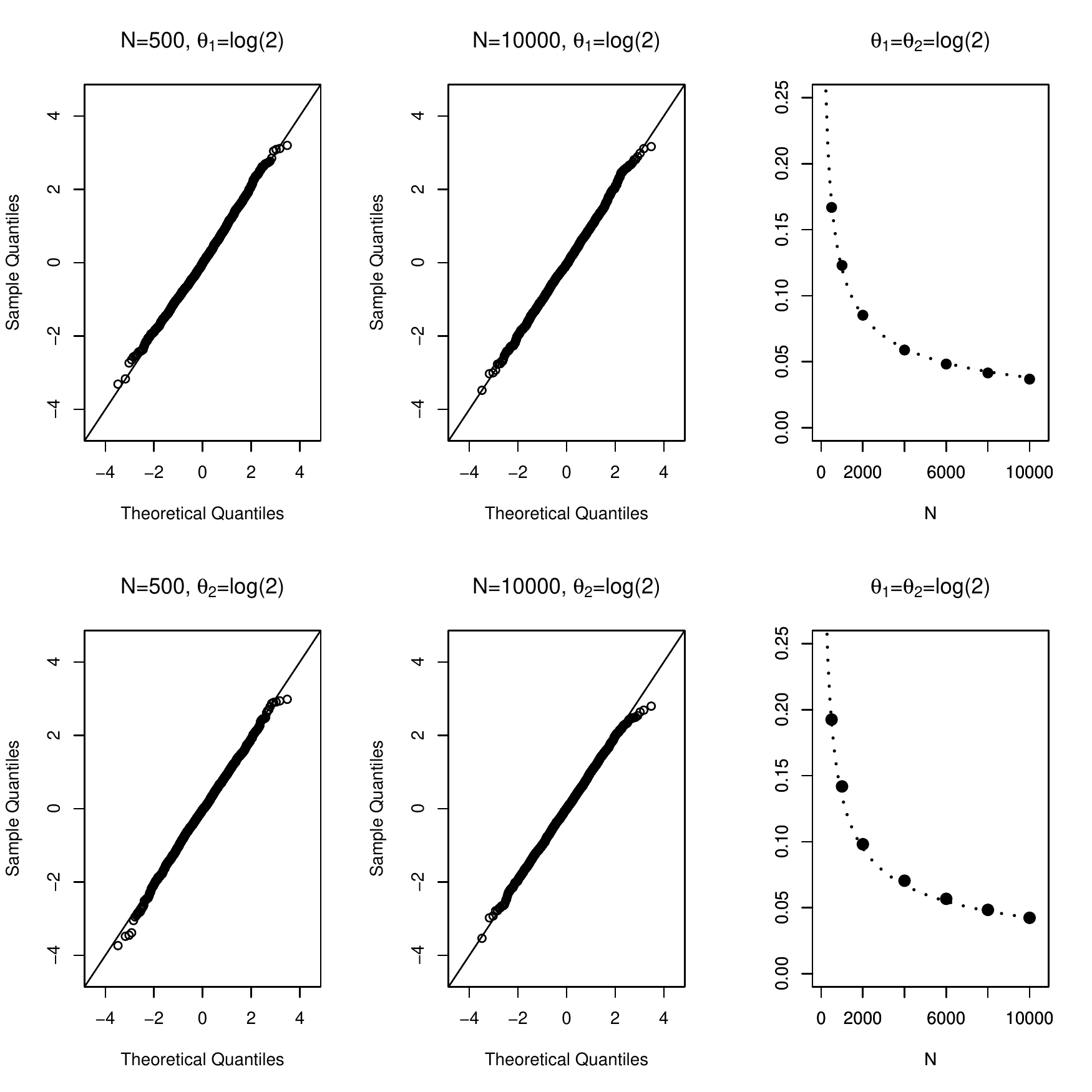}
\caption{Q-Q plots of
  $\sqrt{N}(\hat{\theta}-\theta_0)/\widehat{SE}(\hat{\theta})$ superimposed by
  $y=x$ and plots of average absolute
  differences against $N$ superimposed by $y=c/x^{1/2},c=3.8,4.2$ for
  logistic regression in Scenario 1.}
\label{fig:lg1}
\end{figure}

\begin{figure}[htbp]
\includegraphics[width=12cm,height=7cm]{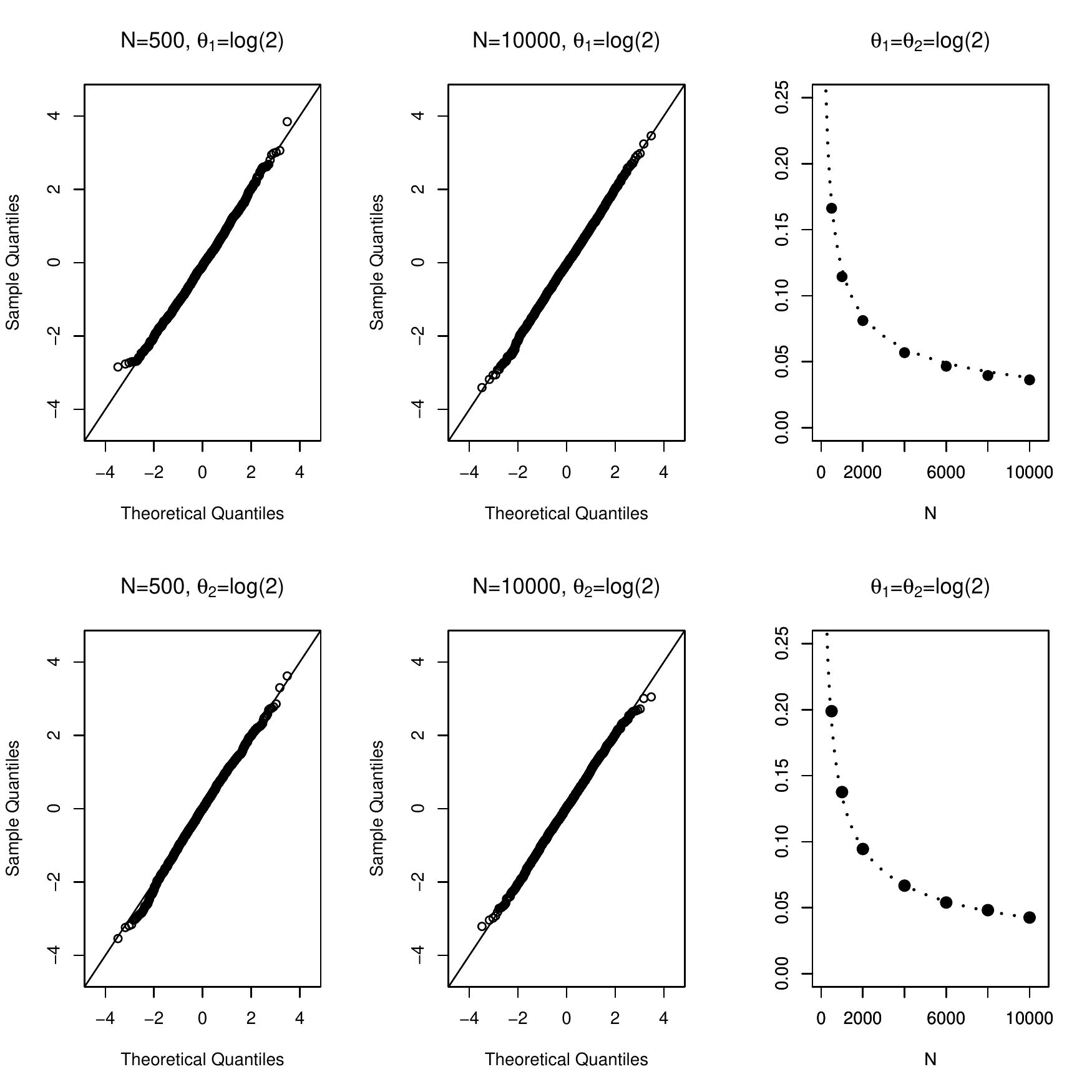}
\caption{Q-Q plots of
  $\sqrt{N}(\hat{\theta}-\theta_0)/\widehat{SE}(\hat{\theta})$ superimposed by
  $y=x$ and plots of average absolute
  differences against $N$ superimposed by $y=c/x^{1/2},c=3.8,4.2$ for
  logistic regression in Scenario 2.}
\label{fig:lg2}
\end{figure}
\begin{figure}[htbp]
\includegraphics[width=12cm,height=7cm]{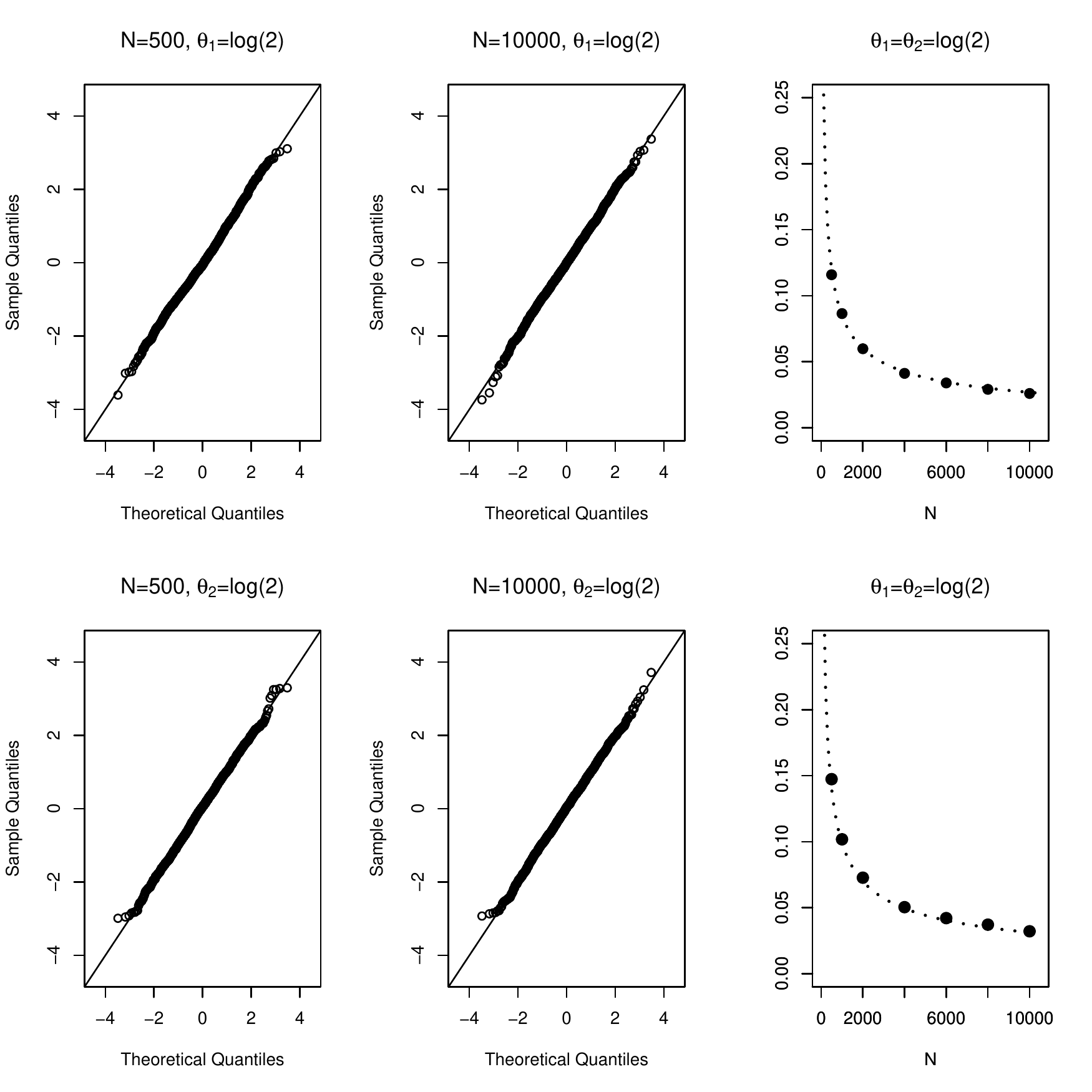}
\caption{Q-Q plots of
  $\sqrt{N}(\hat{\theta}-\theta_0)/\widehat{SE}(\hat{\theta})$ superimposed by
  $y=x$ and plots of average absolute
  differences against $N$ superimposed by $y=c/x^{1/2},c=3.8,4.2$ for
  logistic regression in Scenario 3.}
\label{fig:lg3}
\end{figure}

\newpage
\subsection{Cox proportional hazards model}
We present a plug-in estimator of the asymptotic variance used in the
simulation study. The asymptotic variance of our
uncalibrated estimator $\hat{\theta}_N$ is 
\begin{equation*}
I_0^{-1} + \sum_{j=1}^J \nu^{(j)}\frac{1-p^{(j)}}{p^{(j)}} \textup{Var}_{0}^{(j)}(I^{-1}_0\rho^{(j)}(V)\ell^*_{0}(X)).
\end{equation*}
We estimate $\nu^{(j)}$ and $p^{(j)}$ by
\begin{equation*}
\hat{\nu}^{(j)} = N^{(j)}/N, \quad \hat{p}^{(j)} = n^{(j)}/N^{(j)}.
\end{equation*}
The efficient score in the i.i.d.\@ setting is 
\begin{equation*}
\ell^*_{0} (y, \delta , z)
=\delta(z-(M_1/M_0)(y))
         -e^{\theta_0^Tz}\int_{[0,y]}\left(z-(M_1/M_0)(t)\right)d\Lambda_0(t),
\end{equation*}
where $M_k(s) = P_{\theta_0,\Lambda_0}[Z^ke^{\theta_0^TZ}I(Y\geq s)]$, $k= 0, 1$.
We estimate $M_k(s)$ by
\begin{equation*}
\widehat{M}_k(s) =
\mathbb{P}_N^{\textup{H}}Z^ke^{\hat{\theta}_N^TZ}I(Y\geq s).
\end{equation*}
The estimator $\hat{\Lambda}_N$ of $\Lambda$ is the weighted Breslow estimator of $\Lambda$ given by
\begin{equation*}
\hat{\Lambda}_N(t)=\mathbb{P}_N^{\textup{H}}\frac{\Delta I(Y\leq t)}{\widehat{M}_0(Y)},
\end{equation*}
Thus, we estimate $\ell^*_0$ by
\begin{equation*}
\widehat{\ell_0^*}(y, \delta , z)
=\delta(z-(\widehat{M}_1/\widehat{M}_0)(y))
         -e^{\hat{\theta}_N^Tz}\int_{[0,y]}\left(z-(\widehat{M}_1/\widehat{M}_0)(t)\right)d\widehat{\Lambda}_N(t).
\end{equation*}
The efficient information
$I_0=P_0\left(\ell_0^*\right)^{\otimes 2}$ in the i.i.d.\@ setting
is estimated by
\begin{equation*}
\hat{I}_N = \mathbb{P}_N^{\textup{H}} \left\{\widehat{\ell^*_0} \right\}^{\otimes 2},
\end{equation*}
and the efficient influence function $\tilde{\ell}_0 = I_0^{-1}\ell^*$
in the i.i.d.\@ setting is
estimated by
\begin{equation*}
\widehat{\tilde{\ell}_0}(y, \delta , z) = \{\hat{I}_N\}^{-1}\widehat{\ell_0^*}(y, \delta , z).
\end{equation*}
Now the population variance $I_0^{-1}$ is estimated by
$\{\hat{I}_N\}^{-1}$, and the conditional variance
$\textup{Var}_0^{(j)}(\rho^{(j)}\tilde{\ell}_0)$ is estimated by
\begin{eqnarray}
\widehat{\textup{Var}_0^{(j)}(\rho^{(j)}\tilde{\ell}_0)} &=& \frac{1}{N^{(j)}}
\sum_{i=1}^{N^{(j)}}
\frac{R_{(j),i}^{(j)}}{\pi^{(j)}(V_{(j),i})}\left\{\rho^{(j)}(V_{(j),i})\widehat{\tilde{\ell}}_0(X_{(j),i})\right\}^{\otimes
  2}\nonumber\\
&&-\left\{\frac{1}{N^{(j)}} \sum_{i=1}^{N^{(j)}}
\frac{R_{(j),i}^{(j)}}{\pi^{(j)}(V_{(j),i})}\rho^{(j)}(V_{(j),i})\widehat{\tilde{\ell}}_0(X_{(j),i})\right\}^{\otimes
  2}\label{eqn:estvarj}.
\end{eqnarray}
Combining these pieces we obtain our plug-in variance estimator.

For the calibrated estimator $\hat{\theta}_{N,\textup{c}}$, we  estimate $Q_{\textup{c}}^{(j)}(\tilde{\ell}_0)[v]$ by
\begin{eqnarray*}
\widehat{Q}_{\textup{c}}^{(j)}(\widehat{\tilde{\ell}}_0)[v]
  =\mathbb{P}_N^{\textup{H}}\widehat{\tilde{\ell}}_0V \left\{\mathbb{P}_N^{\textup{H}}V\right\}\rho^{(j)}(v)v.
\end{eqnarray*}
We estimate
$\textup{Var}_0^{(j)}(\rho^{(j)}\tilde{\ell}_0-Q_{\textup{c}}^{(j)}(\tilde{\ell}_0))$
by replacing $\rho^{(j)}(V_{(j),i})\widehat{\tilde{\ell}}_0(X_{(j),i})$ in
(\ref{eqn:estvarj}) by $\rho^{(j)}(V_{(j),i})\widehat{\tilde{\ell}}_0(X_{(j),i}) -
\widehat{Q}_{\textup{c}}^{(j)}(\widehat{\tilde{\ell}}_0)[V_{(j),i}]$.
The rest is the same as the variance estimator for the uncalibrated
estimator. 
The variance estimator for the proposed calibrated estimator
$\hat{\theta}_{N,\textup{sc}}$ is similarly obtained.

We provide a further detail on data source membership in Scenario 4.
There are three data sources in Scenario 4 with $\mathcal{V}^{(3)} = \{V:\Delta=1\}$. The first two data sources were
determined by multinomial logistic regression with a parameter
$\beta=(\beta_1,\beta_2,\beta_3)=(-1,5,3)$.
The probabilities of memberships in 
$\mathcal{V}^{(1)}\cap \{\mathcal{V}^{(2)}\}^c$,
$\mathcal{V}^{(1)}\cap \mathcal{V}^{(2)}$, and
$\{\mathcal{V}^{(1)}\}^{c}\cap \mathcal{V}^{(2)}$ given $V$ are 
\begin{eqnarray*}
&&P\left(\left.V\in \mathcal{V}^{(1)}\cap
\{\mathcal{V}^{(2)}\}^c\right|V\right) = \frac{\exp(\beta_1Z_2)}{\sum_{j=1}^3\exp(\beta_jZ_2)},\\
&&P\left(\left.V\in \mathcal{V}^{(1)}\cap
\mathcal{V}^{(2)}\right|V\right) = \frac{\exp(\beta_2Z_2)}{\sum_{j=1}^3\exp(\beta_jZ_2)},\\
&&P\left(\left.V\in \{\mathcal{V}^{(1)}\}^c\cap
\mathcal{V}^{(2)}\right|V\right) = \frac{\exp(\beta_3Z_2)}{\sum_{j=1}^3\exp(\beta_jZ_2)}.
\end{eqnarray*}

Additional results are summarized in Tables \ref{tab:supp1}-\ref{tab:supp3}, and Figures \ref{fig:3}-\ref{fig:5}.
In Section \ref{sec:6} we present comparison of calibration methods and
choice of $\rho$ for Scenario 4. Table \ref{tab:supp1} summarizes
comparison for other scenarios. Results for Scenario 3 are similar to
results for scenario 4 although efficiency gain through our
calibration method is negligible for estimation of
$\theta_2$.
In Scenarios 1 and 2, differences in Monte Carlo sample standard deviations by choice of $\rho$ is
small and our proposed calibration does not improve efficiency much
for estimation of $\theta_1$. Still, our proposed weights tended to
produce the smallest standard deviation and our proposed calibration reduced
standard deviation most in estimation of $\theta_2$.
Tables \ref{tab:supp2}-\ref{tab:supp3} compare variables used for calibration in Scenarios 3 and 4
respectively. Our proposed calibration gains efficiency most when
using $U$ and $Y$ for estimation of both $\theta_1$ and
$\theta_2$. This indicates that our method is expected to perform
better with more variables to be calibrated. Other calibration methods
produce small improvement and hence we do not see a clear pattern of
choice of variables for reducing standard deviation. 
The percent reduction in design variances is larger in estimation of
$\theta_1$ than in estimation of $\theta_2$.
This is because the auxiliary variable $U$ correlated with $Z_1$ is
used for calibration while $U$ and $Y$ are not strongly correlated
with $Z_2$.
Figures \ref{fig:3}-\ref{fig:5} are Q-Q plots and plots of the average absolute deviations
against $N$ in Scenarios 1-3. Results show that our scaled estimator follows the
standard normal distribution with $\sqrt{N}$-convergence rate.

\begin{table}\tiny

\begin{tabular}{cccccccccc}
\multicolumn{10}{l}{Scenario 1}\\
 $(\alpha,\beta)=(.2,.5)$&\multicolumn{4}{c}{$N=500$}&\multicolumn{4}{c}{$N=10000$}\\ 
$\theta_1=\log (2)$& w/o& SC & C & DC & &w/o& SC & C & DC \\\hline
 MLE & .245&&&&&.0540&&&\\
  S &.482&.488&.482&.482&&.0985&.0985&.0985&.0985\\
  SF &.483&489&.483&.483&&.0985&.0985&.0985&.0985\\
  B & 490&.496&.490&.490&&.0992&.0993&.0992&.0992\\

$\theta_2=\log (2)$& w/o& SC & C & DC & &w/o& SC & C & DC \\\hline
 MLE &.121  &&&&&.0257 &&&\\
  S &.251&.247&.252&.251&&.0526&.0510&.0526&.0523\\
  SF &.252&.247&.252&.251&&.0526&.0511&.0527&.0523\\
  B   &.254&.250&.254&.254&&.0528&.0515&.0528&.0528\\

\multicolumn{10}{l}{Scenario 2}\\
 $(\alpha,\beta)=(.2,.5)$&\multicolumn{4}{c}{$N=500$}&\multicolumn{4}{c}{$N=10000$}\\ 
$\theta_1=\log (2)$& w/o& SC & C & DC & &w/o& SC & C & DC \\\hline
 MLE & .244&&&&&.0537&&&\\
  S &.479  &.484&.478&.478&&.0967&.0967&.0966&.0967\\
  SF &.480&.485&.479&.479&&.0968&.0968&.0968&.0968\\
  B & .489&.491&.489&.489&&.0981&.0981&.0981&.0981\\
$\theta_2=\log (2)$& w/o& SC & C & DC & &w/o& SC & C & DC \\\hline
 MLE &.120 &&&&&.0267&&&\\
  S &.250&.248&.251&.250&&.0526&.0507&.0526&.0524\\
  SF &.250&.249&.251&.250&&.0526&.0508&.0526&.0524\\
  B   &.252&.252&.252&.252&&.0528&.0513&.0528&.0528\\

\multicolumn{10}{l}{Scenario 3}\\
 $(\alpha,\beta)=(.2,.5)$&\multicolumn{4}{c}{$N=500$}&\multicolumn{4}{c}{$N=10000$}\\ 
$\theta_1=\log (2)$& w/o& SC & C & DC & &w/o& SC & C & DC \\\hline
 MLE & .241&&&&&.0534&&&\\
  S &  .330&.307&.331&.331&&.0733&.0666&.0733&.0733\\
  SF &.337&.315&.338&.338&&.0748&.0682&.0748&.0748\\
  B & .396&.379&.396&.396&&.0866&.0810&.0866&.0866\\

$\theta_2=\log (2)$&& &  & & && & & \\\hline
 MLE &.122  &&&&&.0259 &&&\\
  S &.181&.182&.182&.182&&.0378&.0376&.0378&.0378\\
  SF &.185&.185&.185&.185&&.0386&.0383&.0386&.0386\\
  B   &.211&.211&.211&.211&&.0441&.0438&.0441&.0441\\
\end{tabular}
\caption{Comparison of estimators with different
  calibrations and $\rho$ by standard deviations. }
\label{tab:supp1}
\end{table}

\begin{table}\tiny

\begin{tabular}{ccc|cccc|ccc}
Scenario 3&\multicolumn{2}{l}{$(\alpha,\beta)=(.2,.5)$}&&&\\ 
$\theta_1=\log (2)$&&&\multicolumn{3}{c}{$N=500$}&&\multicolumn{3}{c}{$N=10000$}  \\\hline
& MLE & Bias &&.003&&&&.0004&\\
&&SD && .241&&&&.0534\\
&&SEE&& .243&&&&.0536\\\cline{2-10}
& w/o & Bias &&.005&&&&.0009\\
&&SD && .330&&&&.0733 \\
&&SEE& &.330&&&&.0728 \\\cline{2-10}
\\
&\multicolumn{2}{c|}{Calibrated variables}& $U$& $Y$ & $U,Y$  &&$U$& $Y$ & $U,Y$  \\\cline{2-10}
\multirow{12}{*}{Methods} & SC & Bias &.003&.004&.004&&.0002&.0009&.0002\\
&&SD &.308&.331&.307&&.0667&.0734&.0666\\
&&SEE &.304&.329&.303&&.0673&.0728&.0670\\
&&\% Reduction &24.1&$-1.2$&25.5&& 33.1&$-.5$&33.7\\\cline{2-10}
 & C & Bias &.003&.005&.005&&.0020&.0010&.0009\\
&&SD &.333&.332&.331&&.0722&.0722&.0733\\
&&SEE&.330&.331&.328&&.0729&.0728&.0726\\
&&\% Reduction &$-3.5$&$-3.3$&$-1.2$&&5.5&5.5&0.0\\\cline{2-10}
 & DC & Bias &.003&.005&.005&&.0020&.0010&.0009\\
&&SD &.333&.332&.331&&.0722&.0722&.0733\\ 
&&SEE &.330&.331&.328&&.0729&.0728&.0726\\
&&\% Reduction &$-3.5$&$-3.3$&$-1.2$&&5.5&5.5&0.0\\\cline{2-10}
\multicolumn{6}{c}{}\\
$\theta_2=\log (2)$&\multicolumn{6}{c}{} \\\hline

& MLE & Bias &&.001&&&&.0010\\
&&SD && .122&&&&.0259\\
&&SEE&& .120&&&&.0264 \\\cline{2-10}
& w/o & Bias &&.023&&&&.0003\\
&&SD && .181&&&&.0378\\
&&SEE&& .171&&&&.0381 \\\hline
\multicolumn{6}{c}{}\\
&\multicolumn{2}{c|}{Calibrated variables}& $U$& $Y$ & $U,Y$  &&$U$& $Y$ & $U,Y$  \\\cline{2-10}
\multirow{12}{*}{Methods} & SC & Bias &.023&.023&.023&&.0003&.0003&.0004\\
&&SD &.181&.182&.181&&.0381&.0378&.0376\\
&&SEE &.170&.171&.170&&.0381&.0381&.0381\\
&&\% Reduction &0.0&$-.7$&$-.7$&&$-2.5$&0.0&1.7\\\cline{2-10}
 & C & Bias &.021&.022&.023&&.0008&.0004&.0003\\
&&SD &.180&.179&.182&&.0376&.0383&.0378\\
&&SEE &.171&.172&.171&&.0381&.0381&.0381\\
&&\% Reduction &2.1&4.6&$-.3$&&1.7&$-3.4$&0.0\\\cline{2-10}
 & DC & Bias &.021&.022&.023&&.0008&.0004&.0003\\
&&SD &.180&.179&.182&&.0376&.0383&.0378\\ 
&&SEE &.171&.172&.171&&.0381&.0381&.0381\\
&&\% Reduction &2.2&4.6&$-.3$&&1.7&$-3.4$&0.0\\\hline
\\\\
\multicolumn{10}{l}{Note: \% Reduction in design variance is with respect to an uncalibrated estimator with}\\
\multicolumn{10}{l}{our proposed weights.}\\
\end{tabular}
\caption{Comparison of choice of variables for
  calibration by standard deviations in Scenario 3. }
\label{tab:supp2}
\end{table}

\begin{table}\tiny

\begin{tabular}{ccc|cccc|ccc}
Scenario 4&\multicolumn{2}{l}{$(\alpha,\beta)=(.2,.5)$}&&&\\ 
$\theta_1=\log (2)$&&&\multicolumn{3}{c}{$N=500$}&&\multicolumn{3}{c}{$N=10000$}  \\\hline
& MLE & Bias &&.004&&&&.0031&\\
&&SD & &.246&&&&.0534\\
&&SEE& &.243&&&&.0536 \\\cline{2-10}
& w/o & Bias &&.010&&&&.0019\\
&&SD && .368&&&&.0789\\
&&SEE&& .355&&&&.0789 \\\cline{2-10}
\multicolumn{6}{c}{}\\
&\multicolumn{2}{c|}{Calibrated variables}& $U$& $Y$ & $U,Y$  &&$U$&
                                                                     $Y$ & $U,Y$  \\ \cline{2-10}
\multirow{12}{*}{Methods} & SC & Bias &.002&.010&.004&&.0014&.0020&.0010\\
&&SD &.336&.369&.332&&.0723&.0789&.0720\\
&&SEE &.324&.353&.320&&.0719&.0789&.0711\\
&&\% Reduction &26.5&$-1.1$&29.0&&25.9&0.0&27.1\\\cline{2-10}
 & C & Bias &.003&.014&.010&&.0007&.0015&.0020\\
&&SD &.375&.370&.370&&.0784&.0792&.0789\\
&&SEE& .356&.358&.354&&.0789&.0789&.0786\\
&&\% Reduction &$-5.5$&$-1.9$&$-1.2$&&2.0&$-1.2$&0.0\\\cline{2-10}
 & DC & Bias &.015&.009&.017&&.0022&.0019&.0023\\
&&SD &.370&.370&.371&&.0789&.0789&.0789\\ 
&&SEE &.354&.354&.351&&.0789&.0789&.0785\\
&&\% Reduction &$-1.4$&$-1.3$&$-2.5$&&0.0&0.0&0.0\\\cline{2-10}
\multicolumn{6}{c}{}\\
$\theta_2=\log (2)$&   \multicolumn{6}{c}{}\\\hline

& MLE & Bias &&.004&&&&.0004\\
&&SD && .121&&&&.0270\\
&&SEE&& .120&&&&.0264 \\\cline{2-10}
& w/o & Bias& &.023&&&&.0018\\
&&SD && .192&&&&.0407\\
&&SEE&& .181&&&&.0407 \\\hline
 \multicolumn{6}{c}{}\\
&\multicolumn{2}{c|}{Calibrated variables}& $U$& $Y$ & $U,Y$  &&$U$& $Y$ & $U,Y$  \\\cline{2-10}
\multirow{12}{*}{Methods} & SC & Bias &.016&.018&.013&&.0016&.0012&.0011\\
&&SD &.190&.189&.188&&.0401&.0398&.0395\\
&&SEE &.178&.175&.174&&.0402&.0397&.0394\\
&&\% Reduction &3.0&5.2&6.4&&4.4&6.6&8.8\\\cline{2-10}
 & C & Bias &.019&.037&.020&&.0020&.0021&.0016\\
&&SD &.187&.191&.193&&.0405&.0406&.0405\\
&&SEE &.181&.181&.180&&.0407&.0407&.0406\\
&&\% Reduction &7.3&2.1&$-.6$&&1.5&.7&1.5\\\cline{2-10}
 & DC & Bias &.020&.020&.019&&.0017&.0017&.0016\\
&&SD &.194&.192&.193&&.0407&.0403&.0403\\ 
&&SEE &.181&.176&.178&.&.0407&.0399&.0401\\
&&\% Reduction &$-1.5$&.8&$-.6$&&0.0&2.9&2.9\\\hline
\multicolumn{6}{c}{}\\
\multicolumn{10}{l}{Note: \% Reduction in design variance is with respect to an uncalibrated estimator with}\\
\multicolumn{10}{l}{our proposed weights.}\\
\end{tabular}
\caption{Comparison of choice of variables for
  calibration by standard deviations in Scenario 4. }
\label{tab:supp3}
\end{table}

\begin{figure}[htbp]
\includegraphics[width=12cm,height=7cm]{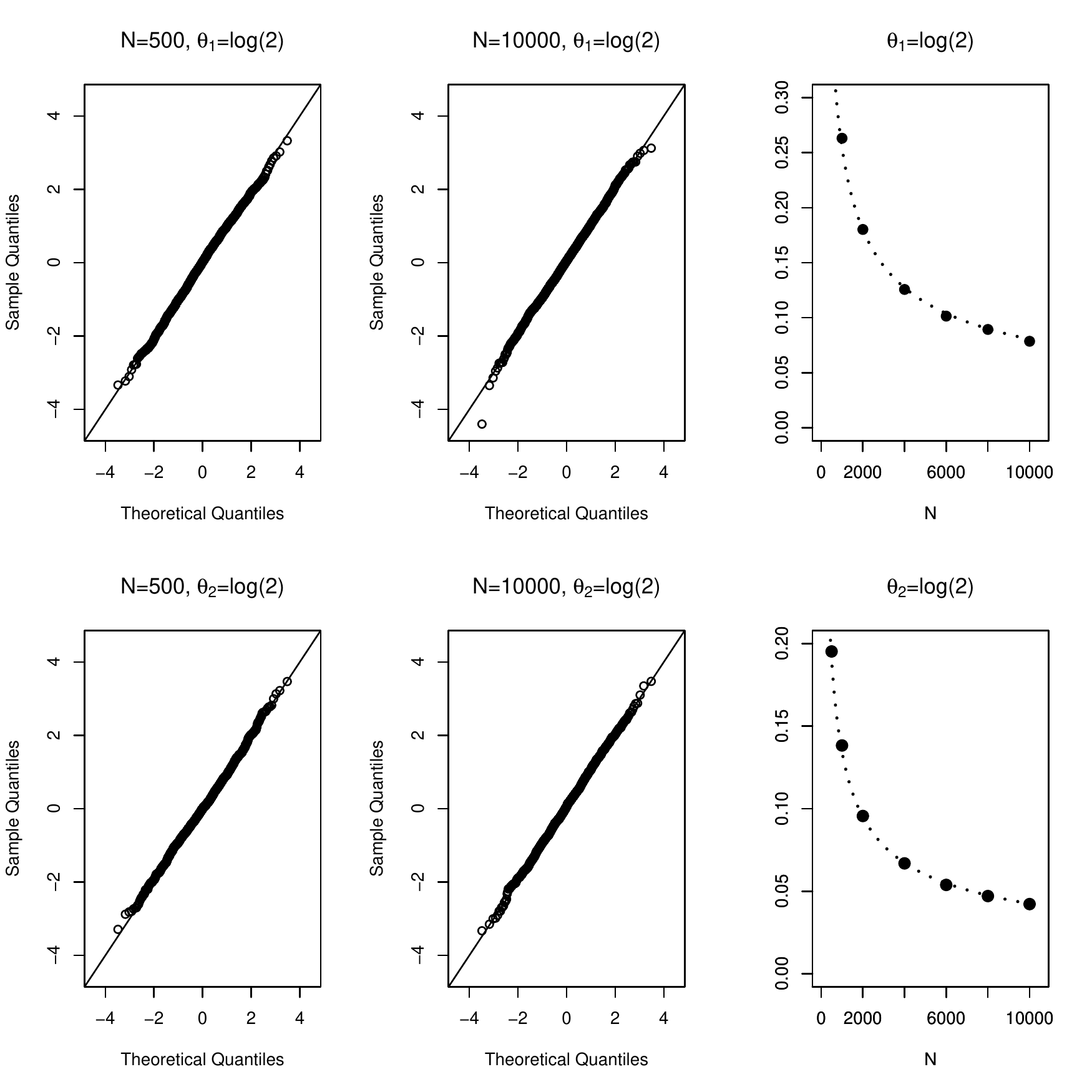}
\caption{Q-Q plots of
  $\sqrt{N}(\hat{\theta}-\theta_0)/\widehat{SE}(\hat{\theta})$ superimposed by
  $y=x$ and plots of average absolute
  differences against $N$ superimposed by $y=c/x^{1/2},c=8.0,4.2$ for
  the Cox model in Scenario 1.}
\label{fig:3}
\end{figure}
\begin{figure}[htbp]
\includegraphics[width=12cm,height=7cm]{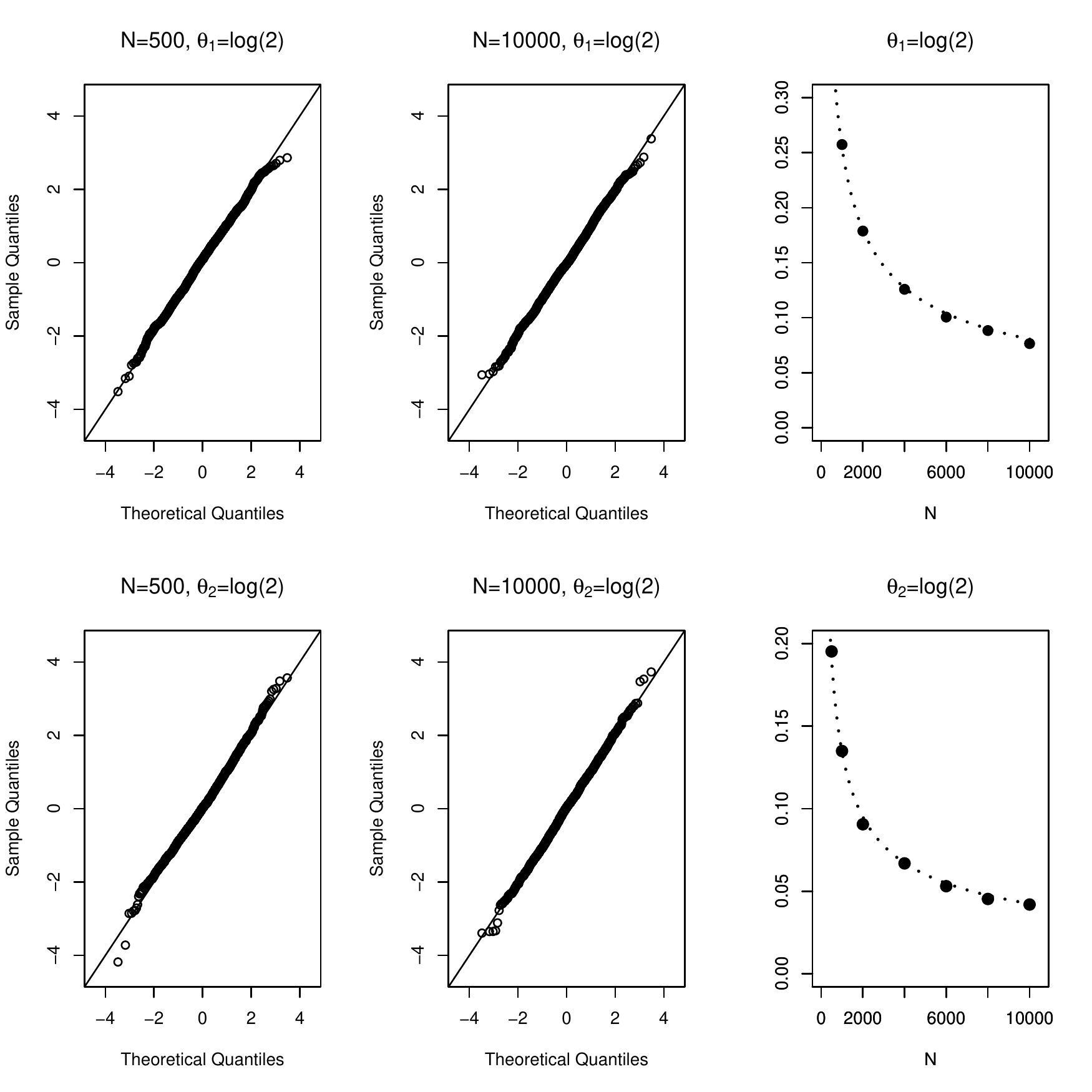}
\caption{Q-Q plots of
  $\sqrt{N}(\hat{\theta}-\theta_0)/\widehat{SE}(\hat{\theta})$ superimposed by
  $y=x$ and plots of average absolute
  differences against $N$ superimposed by $y=c/x^{1/2},c=8.0,4.2$ for
  the Cox model in Scenario 2.}
\label{fig:4}
\end{figure}
\begin{figure}[htbp]
\includegraphics[width=12cm,height=7cm]{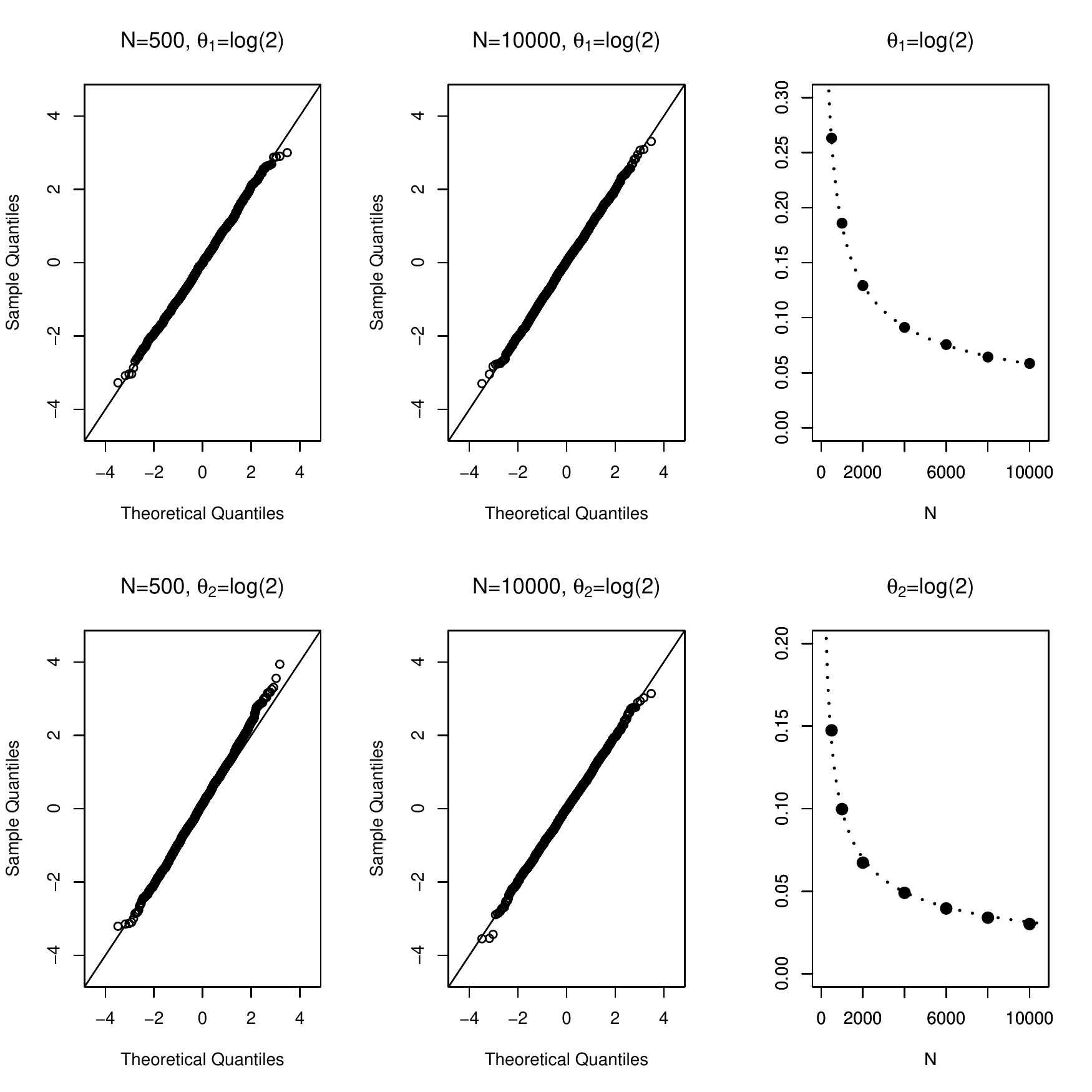}
\caption{Q-Q plots of
  $\sqrt{N}(\hat{\theta}-\theta_0)/\widehat{SE}(\hat{\theta})$ superimposed by
  $y=x$ and plots of average absolute
  differences against $N$ superimposed by $y=c/x^{1/2},c=5.8,3.1$ for
  the Cox model in Scenario 3.}
\label{fig:5}
\end{figure}

\bibliographystyle{imsart-number}
\bibliography{wlebootv1}

\begin{thebibliography}{62}

\bibitem{MR822047}
\begin{binproceedings}[author]
\bauthor{\bsnm{Alexander},~\bfnm{Kenneth~S.}\binits{K.~S.}}
(\byear{1985}).
\btitle{Rates of growth for weighted empirical processes}.
In \bbooktitle{Proceedings of the {B}erkeley conference in honor of {J}erzy
  {N}eyman and {J}ack {K}iefer, {V}ol.\ {II} ({B}erkeley, {C}alif., 1983)}.
\bseries{Wadsworth Statist./Probab. Ser.}
\bpages{475--493}.
\bpublisher{Wadsworth, Belmont, CA}.
\bmrnumber{822047 (87e:60057)}
\end{binproceedings}
\endbibitem

\bibitem{MR1340827}
\begin{barticle}[author]
\bauthor{\bsnm{Bae},~\bfnm{Jongsig}\binits{J.}} \AND
  \bauthor{\bsnm{Levental},~\bfnm{Shlomo}\binits{S.}}
(\byear{1995}).
\btitle{Uniform {CLT} for {M}arkov chains and its invariance principle: a
  martingale approach}.
\bjournal{J. Theoret. Probab.}
\bvolume{8}
\bpages{549--570}.
\bmrnumber{1340827}
\end{barticle}
\endbibitem

\bibitem{MR0464344}
\begin{barticle}[author]
\bauthor{\bsnm{Berkes},~\bfnm{Istv\'an}\binits{I.}} \AND
  \bauthor{\bsnm{Philipp},~\bfnm{Walter}\binits{W.}}
(\byear{1977/78}).
\btitle{An almost sure invariance principle for the empirical distribution
  function of mixing random variables}.
\bjournal{Z. Wahrscheinlichkeitstheorie und Verw. Gebiete}
\bvolume{41}
\bpages{115--137}.
\bmrnumber{0464344}
\end{barticle}
\endbibitem

\bibitem{MR3619696}
\begin{barticle}[author]
\bauthor{\bsnm{Bertail},~\bfnm{Patrice}\binits{P.}},
  \bauthor{\bsnm{Chautru},~\bfnm{Emilie}\binits{E.}} \AND
  \bauthor{\bsnm{Cl\'emen\c{c}on},~\bfnm{Stephan}\binits{S.}}
(\byear{2017}).
\btitle{Empirical processes in survey sampling with (conditional) {P}oisson
  designs}.
\bjournal{Scand. J. Stat.}
\bvolume{44}
\bpages{97--111}.
\bmrnumber{3619696}
\end{barticle}
\endbibitem

\bibitem{MR3670194}
\begin{barticle}[author]
\bauthor{\bsnm{Boistard},~\bfnm{H\'el\`ene}\binits{H.}},
  \bauthor{\bsnm{Lopuha\"a},~\bfnm{Hendrik~P.}\binits{H.~P.}} \AND
  \bauthor{\bsnm{Ruiz-Gazen},~\bfnm{Anne}\binits{A.}}
(\byear{2017}).
\btitle{Functional central limit theorems for single-stage sampling designs}.
\bjournal{Ann. Statist.}
\bvolume{45}
\bpages{1728--1758}.
\bmrnumber{3670194}
\end{barticle}
\endbibitem

\bibitem{BreslowChatterjee1999}
\begin{barticle}[author]
\bauthor{\bsnm{Breslow},~\bfnm{N.~E.}\binits{N.~E.}} \AND
  \bauthor{\bsnm{Chatterjee},~\bfnm{N.}\binits{N.}}
(\byear{1999}).
\btitle{Design and analysis of two-phase studies with binary outcome applied to
  Wilms tumour prognosis}.
\bjournal{Journal of the Royal Statistical Society: Series C (Applied
  Statistics)}
\bvolume{48}
\bpages{457--468}.
\bdoi{10.1111/1467-9876.00165}
\end{barticle}
\endbibitem

\bibitem{MR2325244}
\begin{barticle}[author]
\bauthor{\bsnm{Breslow},~\bfnm{Norman~E.}\binits{N.~E.}} \AND
  \bauthor{\bsnm{Wellner},~\bfnm{Jon~A.}\binits{J.~A.}}
(\byear{2007}).
\btitle{Weighted likelihood for semiparametric models and two-phase stratified
  samples, with application to {C}ox regression}.
\bjournal{Scand. J. Statist.}
\bvolume{34}
\bpages{86--102}.
\bdoi{10.1111/j.1467-9469.2006.00523.x}.
\bmrnumber{2325244}
\end{barticle}
\endbibitem

\bibitem{MR2391566}
\begin{barticle}[author]
\bauthor{\bsnm{Breslow},~\bfnm{Norman~E.}\binits{N.~E.}} \AND
  \bauthor{\bsnm{Wellner},~\bfnm{Jon~A.}\binits{J.~A.}}
(\byear{2008}).
\btitle{A {$Z$}-theorem with estimated nuisance parameters and correction note
  for: ``{W}eighted likelihood for semiparametric models and two-phase
  stratified samples, with application to {C}ox regression'' [{S}cand. {J}.
  {S}tatist. {\bf 34} (2007), no. 1, 86--102; MR2325244]}.
\bjournal{Scand. J. Statist.}
\bvolume{35}
\bpages{186--192}.
\bdoi{10.1111/j.1467-9469.2007.00574.x}.
\bmrnumber{2391566}
\end{barticle}
\endbibitem

\bibitem{Brick2006}
\begin{barticle}[author]
\bauthor{\bsnm{Brick},~\bfnm{J.~Michael}\binits{J.~M.}},
  \bauthor{\bsnm{Dipko},~\bfnm{Sarah}\binits{S.}},
  \bauthor{\bsnm{Presser},~\bfnm{Stanley}\binits{S.}},
  \bauthor{\bsnm{Tucker},~\bfnm{Clyde}\binits{C.}} \AND
  \bauthor{\bsnm{Yuan},~\bfnm{Yangyang}\binits{Y.}}
(\byear{2006}).
\btitle{Nonresponse Bias in a Dual Frame Sample of Cell and Landline Numbers}.
\bjournal{The Public Opinion Quarterly}
\bvolume{70}
\bpages{pp. 780-793}.
\end{barticle}
\endbibitem

\bibitem{Cantelli1933}
\begin{barticle}[author]
\bauthor{\bsnm{Cantelli},~\bfnm{F.~P.}\binits{F.~P.}}
(\byear{1933}).
\btitle{Sulla determinazione empirica delle leggi di probabilita}.
\bjournal{Giorn. Ist. Ital. Attuari}
\bvolume{4}
\bpages{421--424}.
\end{barticle}
\endbibitem

\bibitem{Cervantes2006}
\begin{binproceedings}[author]
\bauthor{\bsnm{Cervantes},~\bfnm{I.~F.}\binits{I.~F.}},
  \bauthor{\bsnm{Jones},~\bfnm{M.~E.}\binits{M.~E.}},
  \bauthor{\bsnm{Rojas},~\bfnm{L.~A.}\binits{L.~A.}},
  \bauthor{\bsnm{Brick},~\bfnm{J.~M.}\binits{J.~M.}},
  \bauthor{\bsnm{Kurata},~\bfnm{J.}\binits{J.}} \AND
  \bauthor{\bsnm{Grant},~\bfnm{D.}\binits{D.}}
(\byear{2006}).
\btitle{A Review of the Sample Design for the California Health Interview
  Survey}.
In \bbooktitle{Proceedings of the Social Statistics Section, American
  Statistical Association}
\bpages{3023--3030}.
\end{binproceedings}
\endbibitem

\bibitem{MR3494641}
\begin{barticle}[author]
\bauthor{\bsnm{Chatterjee},~\bfnm{Nilanjan}\binits{N.}},
  \bauthor{\bsnm{Chen},~\bfnm{Yi-Hau}\binits{Y.-H.}},
  \bauthor{\bsnm{Maas},~\bfnm{Paige}\binits{P.}} \AND
  \bauthor{\bsnm{Carroll},~\bfnm{Raymond~J.}\binits{R.~J.}}
(\byear{2016}).
\btitle{Constrained maximum likelihood estimation for model calibration using
  summary-level information from external big data sources}.
\bjournal{J. Amer. Statist. Assoc.}
\bvolume{111}
\bpages{107--117}.
\bdoi{10.1080/01621459.2015.1123157}.
\bmrnumber{3494641}
\end{barticle}
\endbibitem

\bibitem{MR0341758}
\begin{barticle}[author]
\bauthor{\bsnm{Cox},~\bfnm{D.~R.}\binits{D.~R.}}
(\byear{1972}).
\btitle{Regression models and life-tables}.
\bjournal{J. Roy. Statist. Soc. Ser. B}
\bvolume{34}
\bpages{187--220}.
\bmrnumber{0341758}
\end{barticle}
\endbibitem

\bibitem{pmid2544249}
\begin{barticle}[author]
\bauthor{\bsnm{D'Angio},~\bfnm{G.~J.}\binits{G.~J.}},
  \bauthor{\bsnm{Breslow},~\bfnm{Norman}\binits{N.}},
  \bauthor{\bsnm{Beckwith},~\bfnm{J.~B.}\binits{J.~B.}},
  \bauthor{\bsnm{Evans},~\bfnm{A.}\binits{A.}},
  \bauthor{\bsnm{Baum},~\bfnm{H.}\binits{H.}},
  \bauthor{\bsnm{deLorimier},~\bfnm{A.}\binits{A.}},
  \bauthor{\bsnm{Fernbach},~\bfnm{D.}\binits{D.}},
  \bauthor{\bsnm{Hrabovsky},~\bfnm{E.}\binits{E.}},
  \bauthor{\bsnm{Jones},~\bfnm{B.}\binits{B.}} \AND
  \bauthor{\bsnm{Kelalis},~\bfnm{P.}\binits{P.}}
(\byear{1989}).
\btitle{{{T}reatment of {W}ilms' tumor. {R}esults of the {T}hird {N}ational
  {W}ilms' {T}umor {S}tudy}}.
\bjournal{Cancer}
\bvolume{64}
\bpages{349--360}.
\end{barticle}
\endbibitem

\bibitem{DeLeeuw2005}
\begin{barticle}[author]
\bauthor{\bparticle{de} \bsnm{Leeuw},~\bfnm{Edith~D.}\binits{E.~D.}}
(\byear{2005}).
\btitle{To Mix or Not to Mix Data Collection Modes in Surveys}.
\bjournal{J Off Stat}
\bvolume{21}
\bpages{233--255}.
\end{barticle}
\endbibitem

\bibitem{MR1173804}
\begin{barticle}[author]
\bauthor{\bsnm{Deville},~\bfnm{Jean-Claude}\binits{J.-C.}} \AND
  \bauthor{\bsnm{S{\"a}rndal},~\bfnm{Carl-Erik}\binits{C.-E.}}
(\byear{1992}).
\btitle{Calibration estimators in survey sampling}.
\bjournal{J. Amer. Statist. Assoc.}
\bvolume{87}
\bpages{376--382}.
\bmrnumber{1173804}
\end{barticle}
\endbibitem

\bibitem{Dillman:2014}
\begin{bbook}[author]
\bauthor{\bsnm{Dillman},~\bfnm{Don~A.}\binits{D.~A.}},
  \bauthor{\bsnm{Smyth},~\bfnm{Jolene~D.}\binits{J.~D.}} \AND
  \bauthor{\bsnm{Christian},~\bfnm{Leah~Melani}\binits{L.~M.}}
(\byear{2014}).
\btitle{Internet, Phone, Mail, and Mixed-Mode Surveys: The Tailored Design
  Method}, \bedition{4th} ed.
\bpublisher{Wiley Publishing}.
\end{bbook}
\endbibitem

\bibitem{MR3012400}
\begin{barticle}[author]
\bauthor{\bsnm{Ding},~\bfnm{Ying}\binits{Y.}} \AND
  \bauthor{\bsnm{Nan},~\bfnm{Bin}\binits{B.}}
(\byear{2011}).
\btitle{A sieve {$M$}-theorem for bundled parameters in semiparametric models,
  with application to the efficient estimation in a linear model for censored
  data}.
\bjournal{Ann. Statist.}
\bvolume{39}
\bpages{3032--3061}.
\bdoi{10.1214/11-AOS934}.
\bmrnumber{3012400}
\end{barticle}
\endbibitem

\bibitem{MR0047288}
\begin{barticle}[author]
\bauthor{\bsnm{Donsker},~\bfnm{Monroe~D.}\binits{M.~D.}}
(\byear{1952}).
\btitle{Justification and extension of {D}oob's heuristic approach to the
  {K}omogorov-{S}mirnov theorems}.
\bjournal{Ann. Math. Statistics}
\bvolume{23}
\bpages{277--281}.
\bmrnumber{0047288}
\end{barticle}
\endbibitem

\bibitem{MR665285}
\begin{bincollection}[author]
\bauthor{\bsnm{Dudley},~\bfnm{R.~M.}\binits{R.~M.}}
(\byear{1981}).
\btitle{Donsker classes of functions}.
In \bbooktitle{Statistics and related topics ({O}ttawa, {O}nt., 1980)}
\bpages{341--352}.
\bpublisher{North-Holland, Amsterdam-New York}.
\bmrnumber{665285}
\end{bincollection}
\endbibitem

\bibitem{doi:10.1080/01621459.1969.10501049}
\begin{barticle}[author]
\bauthor{\bsnm{Fellegi},~\bfnm{Ivan~P.}\binits{I.~P.}} \AND
  \bauthor{\bsnm{Sunter},~\bfnm{Alan~B.}\binits{A.~B.}}
(\byear{1969}).
\btitle{A Theory for Record Linkage}.
\bjournal{Journal of the American Statistical Association}
\bvolume{64}
\bpages{1183-1210}.
\bdoi{10.1080/01621459.1969.10501049}
\end{barticle}
\endbibitem

\bibitem{MR757767}
\begin{barticle}[author]
\bauthor{\bsnm{Gin\'e},~\bfnm{Evarist}\binits{E.}} \AND
  \bauthor{\bsnm{Zinn},~\bfnm{Joel}\binits{J.}}
(\byear{1984}).
\btitle{Some limit theorems for empirical processes}.
\bjournal{Ann. Probab.}
\bvolume{12}
\bpages{929--998}.
\bnote{With discussion}.
\bmrnumber{757767}
\end{barticle}
\endbibitem

\bibitem{Glivenko1933}
\begin{barticle}[author]
\bauthor{\bsnm{Glivenko},~\bfnm{V.}\binits{V.}}
(\byear{1933}).
\btitle{Sulla determinazione empirica della legge di probabilita}.
\bjournal{Giorn. Ist. Ital. Attuari}
\bvolume{4}
\bpages{92--99}.
\end{barticle}
\endbibitem

\bibitem{hajek1960}
\begin{barticle}[author]
\bauthor{\bsnm{H{\'a}jek},~\bfnm{J.}\binits{J.}}
(\byear{1960}).
\btitle{Limiting distributions in simple random sampling from a finite
  population}.
\bjournal{Pub. Math. Inst. Hungar. Acad. Sci.}
\bvolume{5}
\bpages{361--374}.
\end{barticle}
\endbibitem

\bibitem{Hartley1962}
\begin{binproceedings}[author]
\bauthor{\bsnm{Hartley},~\bfnm{H.~O.}\binits{H.~O.}}
(\byear{1962}).
\btitle{Multiple frame surveys}.
In \bbooktitle{Proceedings of the Social Statistics Section, American
  Statistical Association}
\bpages{203--206}.
\end{binproceedings}
\endbibitem

\bibitem{MR0228137}
\begin{barticle}[author]
\bauthor{\bsnm{Hartley},~\bfnm{H.~O.}\binits{H.~O.}}
(\byear{1974}).
\btitle{Multiple frame methodology and selected applications}.
\bjournal{Sankhy\=a Ser. C}
\bvolume{36}
\bpages{99--118}.
\end{barticle}
\endbibitem

\bibitem{MR0386084}
\begin{barticle}[author]
\bauthor{\bsnm{Hartley},~\bfnm{H.~O.}\binits{H.~O.}} \AND
  \bauthor{\bsnm{Sielken},~\bfnm{R.~L.}\binits{R.~L.} \bsuffix{Jr.}}
(\byear{1975}).
\btitle{A ``super-population viewpoint'' for finite population sampling}.
\bjournal{Biometrics}
\bvolume{31}
\bpages{411--422}.
\bmrnumber{0386084 (52 \#\#6943)}
\end{barticle}
\endbibitem

\bibitem{Herzog:2007:DQR:1534235}
\begin{bbook}[author]
\bauthor{\bsnm{Herzog},~\bfnm{Thomas~N.}\binits{T.~N.}},
  \bauthor{\bsnm{Scheuren},~\bfnm{Fritz~J.}\binits{F.~J.}} \AND
  \bauthor{\bsnm{Winkler},~\bfnm{William~E.}\binits{W.~E.}}
(\byear{2007}).
\btitle{Data Quality and Record Linkage Techniques}, \bedition{1st} ed.
\bpublisher{Springer Publishing Company, Incorporated}.
\end{bbook}
\endbibitem

\bibitem{MR0053460}
\begin{barticle}[author]
\bauthor{\bsnm{Horvitz},~\bfnm{D.~G.}\binits{D.~G.}} \AND
  \bauthor{\bsnm{Thompson},~\bfnm{D.~J.}\binits{D.~J.}}
(\byear{1952}).
\btitle{A generalization of sampling without replacement from a finite
  universe}.
\bjournal{J. Amer. Statist. Assoc.}
\bvolume{47}
\bpages{663--685}.
\bmrnumber{0053460}
\end{barticle}
\endbibitem

\bibitem{Hu15032011}
\begin{barticle}[author]
\bauthor{\bsnm{Hu},~\bfnm{S.~Sean}\binits{S.~S.}},
  \bauthor{\bsnm{Balluz},~\bfnm{Lina}\binits{L.}},
  \bauthor{\bsnm{Battaglia},~\bfnm{Michael~P.}\binits{M.~P.}} \AND
  \bauthor{\bsnm{Frankel},~\bfnm{Martin~R.}\binits{M.~R.}}
(\byear{2011}).
\btitle{Improving Public Health Surveillance Using a Dual-Frame Survey of
  Landline and Cell Phone Numbers}.
\bjournal{American Journal of Epidemiology}
\bvolume{173}
\bpages{703-711}.
\bdoi{10.1093/aje/kwq442}
\end{barticle}
\endbibitem

\bibitem{MR1394975}
\begin{barticle}[author]
\bauthor{\bsnm{Huang},~\bfnm{Jian}\binits{J.}}
(\byear{1996}).
\btitle{Efficient estimation for the proportional hazards model with interval
  censoring}.
\bjournal{Ann. Statist.}
\bvolume{24}
\bpages{540--568}.
\bdoi{10.1214/aos/1032894452}.
\bmrnumber{1394975}
\end{barticle}
\endbibitem

\bibitem{HuangWellner1997}
\begin{bincollection}[author]
\bauthor{\bsnm{{Huang}},~\bfnm{Jian}\binits{J.}} \AND
  \bauthor{\bsnm{{Wellner}},~\bfnm{Jon~A.}\binits{J.~A.}}
(\byear{1997}).
\btitle{{Interval censored survival data: A review of recent progress.}}
In \bbooktitle{{Proceedings of the first Seattle symposium in biostatistics:
  survival analysis, Seattle, WA, USA, November 20--21, 1995}}
\bpages{123--169}.
\bpublisher{Berlin: Springer}.
\end{bincollection}
\endbibitem

\bibitem{Iachan1993}
\begin{barticle}[author]
\bauthor{\bsnm{Iachan},~\bfnm{R.}\binits{R.}} \AND
  \bauthor{\bsnm{Dennis},~\bfnm{M.~L.}\binits{M.~L.}}
(\byear{1993}).
\btitle{A Multiple Frame Approach to Sampling the Homeless and Transient
  Population}.
\bjournal{J Off Stat}
\bvolume{9}
\bpages{747--764}.
\end{barticle}
\endbibitem

\bibitem{MR648029}
\begin{barticle}[author]
\bauthor{\bsnm{Isaki},~\bfnm{Cary~T.}\binits{C.~T.}} \AND
  \bauthor{\bsnm{Fuller},~\bfnm{Wayne~A.}\binits{W.~A.}}
(\byear{1982}).
\btitle{Survey design under the regression superpopulation model}.
\bjournal{J. Amer. Statist. Assoc.}
\bvolume{77}
\bpages{89--96}.
\end{barticle}
\endbibitem

\bibitem{Kalton1986}
\begin{barticle}[author]
\bauthor{\bsnm{Kalton},~\bfnm{Graham}\binits{G.}} \AND
  \bauthor{\bsnm{Anderson},~\bfnm{Dallas~W.}\binits{D.~W.}}
(\byear{1986}).
\btitle{Sampling Rare Populations}.
\bjournal{Journal of the Royal Statistical Society. Series A (General)}
\bvolume{149}
\bpages{pp. 65-82}.
\end{barticle}
\endbibitem

\bibitem{RSSA:RSSA12136}
\begin{barticle}[author]
\bauthor{\bsnm{Keiding},~\bfnm{Niels}\binits{N.}} \AND
  \bauthor{\bsnm{Louis},~\bfnm{Thomas~A.}\binits{T.~A.}}
(\byear{2016}).
\btitle{Perils and potentials of self-selected entry to epidemiological studies
  and surveys}.
\bjournal{Journal of the Royal Statistical Society: Series A (Statistics in
  Society)}
\bvolume{179}
\bpages{319--376}.
\bdoi{10.1111/rssa.12136}
\end{barticle}
\endbibitem

\bibitem{MR2915160}
\begin{barticle}[author]
\bauthor{\bsnm{Kim},~\bfnm{Gunky}\binits{G.}} \AND
  \bauthor{\bsnm{Chambers},~\bfnm{Raymond}\binits{R.}}
(\byear{2012}).
\btitle{Regression analysis under incomplete linkage}.
\bjournal{Comput. Statist. Data Anal.}
\bvolume{56}
\bpages{2756--2770}.
\bmrnumber{2915160}
\end{barticle}
\endbibitem

\bibitem{MR2724368}
\begin{bbook}[author]
\bauthor{\bsnm{Kosorok},~\bfnm{Michael~R.}\binits{M.~R.}}
(\byear{2008}).
\btitle{Introduction to Empirical Processes and Semiparametric Inference}.
\bseries{Springer Series in Statistics}.
\bpublisher{Springer, New York}.
\bdoi{10.1007/978-0-387-74978-5}.
\bmrnumber{2724368 (2012b:62005)}
\end{bbook}
\endbibitem

\bibitem{MR2156832}
\begin{barticle}[author]
\bauthor{\bsnm{Lahiri},~\bfnm{P.}\binits{P.}} \AND
  \bauthor{\bsnm{Larsen},~\bfnm{Michael~D.}\binits{M.~D.}}
(\byear{2005}).
\btitle{Regression analysis with linked data}.
\bjournal{J. Amer. Statist. Assoc.}
\bvolume{100}
\bpages{222--230}.
\bmrnumber{2156832}
\end{barticle}
\endbibitem

\bibitem{MR996990}
\begin{barticle}[author]
\bauthor{\bsnm{Levental},~\bfnm{Shlomo}\binits{S.}}
(\byear{1989}).
\btitle{A uniform {CLT} for uniformly bounded families of martingale
  differences}.
\bjournal{J. Theoret. Probab.}
\bvolume{2}
\bpages{271--287}.
\bmrnumber{996990}
\end{barticle}
\endbibitem

\bibitem{MR2324141}
\begin{barticle}[author]
\bauthor{\bsnm{Lohr},~\bfnm{Sharon}\binits{S.}} \AND
  \bauthor{\bsnm{Rao},~\bfnm{J.~N.~K.}\binits{J.~N.~K.}}
(\byear{2006}).
\btitle{Estimation in multiple-frame surveys}.
\bjournal{J. Amer. Statist. Assoc.}
\bvolume{101}
\bpages{1019--1030}.
\bdoi{10.1198/016214506000000195}.
\bmrnumber{2324141}
\end{barticle}
\endbibitem

\bibitem{Lu2012}
\begin{binproceedings}[author]
\bauthor{\bsnm{Lu},~\bfnm{Y.}\binits{Y.}}
(\byear{2012}).
\btitle{Regression Coefficient Estimation in Dual Frame Surveys}.
In \bbooktitle{Proceedings of the Section on Survey Research Methods, American
  Statistical Association}
\bpages{4687--4695}.
\end{binproceedings}
\endbibitem

\bibitem{lu2010}
\begin{barticle}[author]
\bauthor{\bsnm{Lu},~\bfnm{Yan}\binits{Y.}} \AND
  \bauthor{\bsnm{Lohr},~\bfnm{SL}\binits{S.}}
(\byear{2010}).
\btitle{Gross flow estimation in dual frame surveys}.
\bjournal{Survey Methodology}
\bvolume{36}
\bpages{13--22}.
\end{barticle}
\endbibitem

\bibitem{MR2202406}
\begin{barticle}[author]
\bauthor{\bsnm{Ma},~\bfnm{Shuangge}\binits{S.}} \AND
  \bauthor{\bsnm{Kosorok},~\bfnm{Michael~R.}\binits{M.~R.}}
(\byear{2005}).
\btitle{Robust semiparametric {M}-estimation and the weighted bootstrap}.
\bjournal{J. Multivariate Anal.}
\bvolume{96}
\bpages{190--217}.
\bdoi{10.1016/j.jmva.2004.09.008}.
\bmrnumber{2202406}
\end{barticle}
\endbibitem

\bibitem{Metcalf2009}
\begin{barticle}[author]
\bauthor{\bsnm{Metcalf},~\bfnm{Patricia}\binits{P.}} \AND
  \bauthor{\bsnm{Scott},~\bfnm{Alastair}\binits{A.}}
(\byear{2009}).
\btitle{Using multiple frames in health surveys}.
\bjournal{Statistics in Medicine}
\bvolume{28}
\bpages{1512--1523}.
\bdoi{10.1002/sim.3566}
\end{barticle}
\endbibitem

\bibitem{MR1245301}
\begin{barticle}[author]
\bauthor{\bsnm{Pr{\ae}stgaard},~\bfnm{Jens}\binits{J.}} \AND
  \bauthor{\bsnm{Wellner},~\bfnm{Jon~A.}\binits{J.~A.}}
(\byear{1993}).
\btitle{Exchangeably weighted bootstraps of the general empirical process}.
\bjournal{Ann. Probab.}
\bvolume{21}
\bpages{2053--2086}.
\end{barticle}
\endbibitem

\bibitem{Ranalli2015}
\begin{barticle}[author]
\bauthor{\bsnm{Ranalli},~\bfnm{M.~Giovanna}\binits{M.~G.}},
  \bauthor{\bsnm{Arcos},~\bfnm{Antonio}\binits{A.}},
  \bauthor{\bsnm{Rueda},~\bfnm{MaríadelMar}\binits{M.}} \AND
  \bauthor{\bsnm{Teodoro},~\bfnm{Annalisa}\binits{A.}}
(\byear{2016}).
\btitle{Calibration estimation in dual-frame surveys}.
\bjournal{Statistical Methods \& Applications}
\bvolume{25}
\bpages{321-349}.
\bdoi{10.1007/s10260-015-0336-5}
\end{barticle}
\endbibitem

\bibitem{Rao1994}
\begin{barticle}[author]
\bauthor{\bsnm{Rao},~\bfnm{J.~N.~K.}\binits{J.~N.~K.}}
(\byear{1994}).
\btitle{Estimating Totals and Distribution Functions Using Auxiliary
  Information at the Estimation Stage}.
\bjournal{J Off Stat}
\bvolume{10}
\bpages{153--165}.
\end{barticle}
\endbibitem

\bibitem{MR2796566}
\begin{barticle}[author]
\bauthor{\bsnm{Rao},~\bfnm{J.~N.~K.}\binits{J.~N.~K.}} \AND
  \bauthor{\bsnm{Wu},~\bfnm{Changbao}\binits{C.}}
(\byear{2010}).
\btitle{Pseudo-empirical likelihood inference for multiple frame surveys}.
\bjournal{J. Amer. Statist. Assoc.}
\bvolume{105}
\bpages{1494--1503}.
\bdoi{10.1198/jasa.2010.tm09534}.
\bmrnumber{2796566 (2012b:62035)}
\end{barticle}
\endbibitem

\bibitem{MR2253102}
\begin{barticle}[author]
\bauthor{\bsnm{Rubin-Bleuer},~\bfnm{Susana}\binits{S.}} \AND
  \bauthor{\bsnm{Schiopu~Kratina},~\bfnm{Ioana}\binits{I.}}
(\byear{2005}).
\btitle{On the two-phase framework for joint model and design-based inference}.
\bjournal{Ann. Statist.}
\bvolume{33}
\bpages{2789--2810}.
\bdoi{10.1214/009053605000000651}.
\bmrnumber{2253102}
\end{barticle}
\endbibitem

\bibitem{SW2013supp}
\begin{barticle}[author]
\bauthor{\bsnm{Saegusa},~\bfnm{Takumi}\binits{T.}} \AND
  \bauthor{\bsnm{Wellner},~\bfnm{Jon~A.}\binits{J.~A.}}
(\byear{2013}).
\btitle{Supplementary material to "{W}eighted likelihood estimation under
  two-phase sampling"}.
\bdoi{10.1214/12-AOS1073SUPP}
\end{barticle}
\endbibitem

\bibitem{MR3059418}
\begin{barticle}[author]
\bauthor{\bsnm{Saegusa},~\bfnm{Takumi}\binits{T.}} \AND
  \bauthor{\bsnm{Wellner},~\bfnm{Jon~A.}\binits{J.~A.}}
(\byear{2013}).
\btitle{Weighted likelihood estimation under two-phase sampling}.
\bjournal{Ann. Statist.}
\bvolume{41}
\bpages{269--295}.
\bdoi{10.1214/12-AOS1073}.
\bmrnumber{3059418}
\end{barticle}
\endbibitem

\bibitem{MR1147105}
\begin{barticle}[author]
\bauthor{\bsnm{Skinner},~\bfnm{C.~J.}\binits{C.~J.}}
(\byear{1991}).
\btitle{On the efficiency of raking ratio estimation for multiple frame
  surveys}.
\bjournal{J. Amer. Statist. Assoc.}
\bvolume{86}
\bpages{779--784}.
\bmrnumber{1147105 (92i:62020)}
\end{barticle}
\endbibitem

\bibitem{MR1394091}
\begin{barticle}[author]
\bauthor{\bsnm{Skinner},~\bfnm{C.~J.}\binits{C.~J.}} \AND
  \bauthor{\bsnm{Rao},~\bfnm{J.~N.~K.}\binits{J.~N.~K.}}
(\byear{1996}).
\btitle{Estimation in dual frame surveys with complex designs}.
\bjournal{J. Amer. Statist. Assoc.}
\bvolume{91}
\bpages{349--356}.
\bdoi{10.2307/2291414}.
\bmrnumber{1394091 (97a:62018)}
\end{barticle}
\endbibitem

\bibitem{MR1915446}
\begin{bincollection}[author]
\bauthor{\bparticle{van~der} \bsnm{Vaart},~\bfnm{Aad}\binits{A.}}
(\byear{2002}).
\btitle{Semiparametric statistics}.
In \bbooktitle{Lectures on probability theory and statistics ({S}aint-{F}lour,
  1999)}.
\bseries{Lecture Notes in Math.}
\bvolume{1781}
\bpages{331--457}.
\bpublisher{Springer}, \baddress{Berlin}.
\end{bincollection}
\endbibitem

\bibitem{MR1333176}
\begin{barticle}[author]
\bauthor{\bparticle{van~der} \bsnm{Vaart},~\bfnm{A.~W.}\binits{A.~W.}}
(\byear{1995}).
\btitle{Efficiency of infinite-dimensional {$M$}-estimators}.
\bjournal{Statist. Neerlandica}
\bvolume{49}
\bpages{9--30}.
\bdoi{10.1111/j.1467-9574.1995.tb01452.x}.
\bmrnumber{1333176}
\end{barticle}
\endbibitem

\bibitem{MR1652247}
\begin{bbook}[author]
\bauthor{\bparticle{van~der} \bsnm{Vaart},~\bfnm{A.~W.}\binits{A.~W.}}
(\byear{1998}).
\btitle{Asymptotic statistics}.
\bseries{Cambridge Series in Statistical and Probabilistic Mathematics}
\bvolume{3}.
\bpublisher{Cambridge University Press}, \baddress{Cambridge}.
\end{bbook}
\endbibitem

\bibitem{MR1385671}
\begin{bbook}[author]
\bauthor{\bparticle{van~der} \bsnm{Vaart},~\bfnm{Aad~W.}\binits{A.~W.}} \AND
  \bauthor{\bsnm{Wellner},~\bfnm{Jon~A.}\binits{J.~A.}}
(\byear{1996}).
\btitle{Weak {C}onvergence and {E}mpirical {P}rocesses}.
\bseries{Springer Series in Statistics}.
\bpublisher{Springer-Verlag}, \baddress{New York}.
\end{bbook}
\endbibitem

\bibitem{MR1857319}
\begin{bincollection}[author]
\bauthor{\bparticle{van~der} \bsnm{Vaart},~\bfnm{Aad~W.}\binits{A.~W.}} \AND
  \bauthor{\bsnm{Wellner},~\bfnm{Jon~A.}\binits{J.~A.}}
(\byear{2000}).
\btitle{Preservation theorems for {G}livenko-{C}antelli and uniform
  {G}livenko-{C}antelli classes}.
In \bbooktitle{High dimensional probability, II (Seattle, WA, 1999)}.
\bseries{Progr. Probab.}
\bvolume{47}
\bpages{115--133}.
\bpublisher{Birkh\"auser Boston}, \baddress{Boston, MA}.
\end{bincollection}
\endbibitem

\bibitem{Winkler1995}
\begin{binbook}[author]
\bauthor{\bsnm{Winkler},~\bfnm{William~E.}\binits{W.~E.}}
(\byear{1995}).
\btitle{Matching and Record Linkage}
\bpages{353--384}.
\bpublisher{John Wiley \& Sons, Inc.}
\bdoi{10.1002/9781118150504.ch20}
\end{binbook}
\endbibitem

\bibitem{MR1473875}
\begin{barticle}[author]
\bauthor{\bsnm{Ziegler},~\bfnm{Klaus}\binits{K.}}
(\byear{1997}).
\btitle{Functional central limit theorems for triangular arrays of
  function-indexed processes under uniformly integrable entropy conditions}.
\bjournal{J. Multivariate Anal.}
\bvolume{62}
\bpages{233--272}.
\bmrnumber{1473875}
\end{barticle}
\endbibitem

\bibitem{MR1834583}
\begin{barticle}[author]
\bauthor{\bsnm{Ziegler},~\bfnm{Klaus}\binits{K.}}
(\byear{2001}).
\btitle{Uniform laws of large numbers for triangular arrays of function-indexed
  processes under random entropy conditions}.
\bjournal{Results Math.}
\bvolume{39}
\bpages{374--389}.
\bmrnumber{1834583}
\end{barticle}
\endbibitem

\end{thebibliography}

\end{document}